\documentclass[times]{fldauth}

\usepackage{moreverb}

\newcommand\BibTeX{{\rmfamily B\kern-.05em \textsc{i\kern-.025em b}\kern-.08em
T\kern-.1667em\lower.7ex\hbox{E}\kern-.125emX}}

\def\volumeyear{2010}

\usepackage{amsfonts,amssymb,stmaryrd,amsmath,amssymb}
\usepackage{graphicx,color}
\usepackage{multimedia}
\usepackage{float}
\usepackage{placeins}

\newtheorem{remark}{Remark}

\newtheorem{lemma}{Lemma}

\newtheorem{proposition}{Proposition}

\newcommand \bn {\boldsymbol{\mathrm{n}}}
\newcommand \bw {\boldsymbol{\mathrm{w}}}
\newcommand \bh {\boldsymbol{\mathrm{h}}}

\newcommand \bu {\boldsymbol{\mathrm{u}}}
\newcommand \bv {\boldsymbol{\mathrm{v}}}
\newcommand \bg {\boldsymbol{\mathrm{g}}}
\newcommand \bff {\boldsymbol{\mathrm{f}}}
\newcommand \blambda {\boldsymbol{\mathrm{\lambda}}}
\newcommand \bmu {\boldsymbol{\mathrm{\mu}}}

\newcommand \bvarphi {\boldsymbol{\mathrm{\varphi}}}

\newcommand \bpsi {\boldsymbol{\mathrm{\psi}}}

\newcommand \p {\partial}

\newcommand \x {\mathrm{x}}
\newcommand \y {\mathrm{y}}

\newcommand \R {\mathbb{R}}

\renewcommand \L {\mathrm{L}}

\renewcommand \H {\mathrm{H}}

\renewcommand \d {\mathrm{d}}

\renewcommand \div {\mathrm{div}}

\newcommand \trace {\mathrm{trace}}

\begingroup\makeatletter\ifx\SetFigFont\undefined%
\gdef\SetFigFont#1#2#3#4#5{%
  \reset@font\fontsize{#1}{#2pt}%
  \fontfamily{#3}\fontseries{#4}\fontshape{#5}%
  \selectfont}%
\fi\endgroup%

\begin{document}

\runningheads{S.~Court, M.~Fourni\'e, A.~Lozinski}{A fictitious domain approach for the Stokes problem}

\title{A fictitious domain approach for the Stokes problem based on the extended finite element method\footnotemark[2]}

\author{S\'ebastien Court\corrauth, Michel Fourni\'e, Alexei Lozinski}

\address{Institut de Math\'ematiques de Toulouse, Universit\'e de Toulouse, France.}

\cgsn{This work is partially supported by the foundation STAE in the context of the RTRA platform SMARTWING, and the ANR-project CISIFS}{09-BLAN-0213-03.}

\corraddr{sebastien.court@math.univ-toulouse.fr}

\begin{abstract}
In the present work, we propose to extend to the Stokes problem a fictitious domain approach inspired by eXtended Finite Element Method \cite{MoesD} and studied for the Poisson problem in \cite{HaslR}. The method allows computations in domains whose boundaries do not match. A mixed finite element method is used for the fluid flow. The interface between the fluid and the structure is localized by a level-set function. Dirichlet boundary conditions are taken into account using Lagrange multiplier. A stabilization term is introduced to improve the approximation of the normal trace of the Cauchy stress tensor at the interface and avoid the inf-sup condition between the spaces for the velocity and the Lagrange multiplier. Convergence analysis is given and several numerical tests are performed to illustrate the capabilities of the method.
\end{abstract}

\keywords{Fictitious domain, Xfem, Mixed method, Stabilization technique, Fluid-structure interactions.}

\maketitle


\vspace{-6pt}

\section{Introduction}
\textcolor{black}{
Fluid-Structure Interactions (FSI) are of great relevance in many fields of applied scientific and engineering disciplines. A comprehensive study of such problems remains a challenge and justifies the attention made over the last decades to propose efficient and robust numerical methods. We refer to \cite{Hou&Wang&Layton} where different numerical procedures to solve the FSI problems are reviewed. One classification of FSI solution procedures can be based upon the treatment of the meshes with conforming or non-conforming mesh methods. For the first ones, meshes are conformed to the interface where the physical boundary conditions are imposed \cite{LT, SMSTT0, SST}. As the geometry of the fluid domain changes through the time, re-meshing is needed, what is excessively time-consuming, in particular for complex systems.\\
In the present paper, we are interested in non-conforming mesh methods with a fictitious domain approach where the mesh is cut by the boundary. Most of the non-conforming mesh methods are based upon the framework of the immersed methods where force-equivalent terms are added to the fluid equations in order to represent the fluid-structure interaction \cite{Peskinacta, Mittal&Iaccarino}. Many related numerical methods have been developed, in particular the popular distributed Lagrange multiplier method, introduced for rigid bodies moving in an incompressible flow \cite{Glowinski}. In this method, the fluid domain is extended to cover the rigid domain where the fluid velocity is required to be equal to the rigid body velocity. This constraint is enforced by using distributed Lagrange multipliers, which should be approximated on a mesh covering the structure and sufficiently coarse with respect to the mesh used for the fluid velocity, in order to satisfy the inf-sup condition.\\
More recently eXtended Finite Element Method (XFEM) introduced by Mo\"{e}s, Dolbow and Belytschko in \cite{MoesD} (see a review of such methods in \cite{reviewXfem}) has been adapted to FSI problems in \cite{MoesB, SukumarC, Gerstenberger2008, Choi2010}. The idea is similar to the fictitious domain / Lagrange multiplier method mentioned above, but the fluid velocity is no longer extended inside the structure domain and its equality with the structure velocity is enforced by a Lagrange multiplier only on the fluid-structure interface. One thus gets rid of unnecessary fluid unknowns and moreover one easily recovers the normal trace of the Cauchy stress tensor on the interface. We note that XFEM has been originally developed for problems in structural mechanics mostly in the context of cracked domains, see for example \cite{HaslR, MoesG, Stazi, SukumarM, Stolarska}. The specificity of the method is that it combines a level-set representation of the geometry of the crack with an enrichment of a finite element space by singular and discontinuous functions. Several strategies can be considered in order to improve the original XFEM. Some of these strategies are mathematically analyzed in \cite{Laborde, Chahine}.\\
In the context of fluid-structure interactions, the difficulty that present the applications of such techniques lies in the choice of the Lagrange multiplier space used in order to take into account the interface, which is not trivial due to the fact that the interface cuts the mesh \cite{Bechet2009}. Indeed, the multiplier space, besides having good approximation properties, should satisfy an uniform inf-sup condition (similarly to more traditional fictitious domain methods \cite{Girault}). In a straightforward discretization, it implies that the mesh for the multiplier should be sufficiently coarse in comparison with the mesh for the primal variables. Thus, the natural mesh given by the points of intersection of the interface with the global mesh cannot be used directly. An algorithm to construct a multiplier space satisfying the inf-sup condition is developed in \cite{Bechet2009}, but its implementation can be difficult in practice. It may be thus preferable to work on the natural, easily constructible mesh, as outlined above.  This is achieved in a stabilized version of the method proposed in \cite{HaslR} (an extension to the contact problems in elastostatics is also available in \cite{HildR}). In the present paper, we are interested in extending the method of \cite{HaslR} to the Stokes problem. An important feature of this method (based on the XFEM approach, similarly to \cite{Gerstenberger2008, Choi2010}) is that the Lagrange multiplier is identified with the normal trace of the Cauchy stress tensor $\sigma(\bu,p)\bn$ at the interface. With the aid of the stabilization technique presented in this present paper (never studied in that context in our knowledge), we have a good numerical approximation of this quantity, that is crucial in FSI since it gives the force exerted by the viscous fluid on the structure. By the way we note that alternative methods based on the work of Nitsche \cite{Nitsche} (such as \cite{Burman1, Burman3} in the context of the Poisson problem and \cite{Massing} in the context of the Stokes problems) do not introduce the Lagrange multiplier and thus do not necessarily provide a good numerical approximation of this force.\\
The outline of the paper is as follows. The continuous problem is set in section \ref{secStokes} in the weak sense, and the functional spaces are given. We recall the corresponding variational formulation with the introduction of a Lagrange multiplier to impose the boundary condition in the interface. Next in section \ref{seccut} the fictitious domain method is introduced. In particular the discrete spaces are defined and the discrete variational problem is studied without the stabilization technique. This latter - which is an augmented Lagrangian method - is introduced in section \ref{secstab}, and we show that theoretically it enables us to recover the convergence for the multiplier associated with the Dirichlet condition (see Lemma \ref{lemmainfsup}). The convergence analysis for the stabilized method is given in section \ref{seccvstab} and optimal error estimates are proved. Section \ref{secnum} is devoted to numerical tests. Rates of convergence are computed with or without stabilization and the behavior of the method is studied for different geometric configurations. Moreover we compare our method with a classical one which uses a boundary-fitted mesh. Technical aspects of the implementation are discussed in section \ref{seccomments}. Finally in section \ref{secFSI} we perform simulations in a simplified unsteady case, what gives a glimpse of the future perspectives. The conclusion is given in section \ref{secCCL}.
}


\section{Setting of the problem} \label{secStokes}
In a bounded domain of $\R^2$, denoted by $\mathcal{O}$, we consider a full solid immersed in a viscous incompressible fluid. The domain occupied by the solid is denoted by $\mathcal{S}$, and we denote by $\Gamma$ its boundary. The fluid surrounding the structure occupies the domain $\mathcal{O} \setminus \overline{\mathcal{S}} = \mathcal{F}$, \textcolor{black}{where $\overline{\mathcal{S}}$ denotes $\mathcal{S} \cup \partial \mathcal{S}$} (see figure \ref{figDomain}).

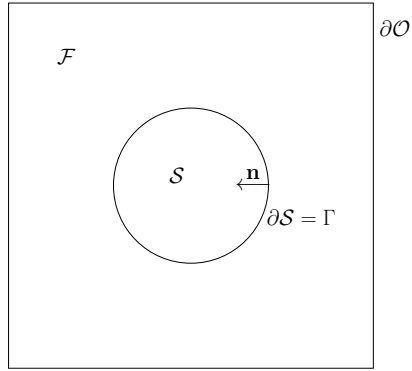
\begin{figure}[!h]
\begin{center}
\scalebox{0.4}{
\begin{picture}(0,0)%
\includegraphics{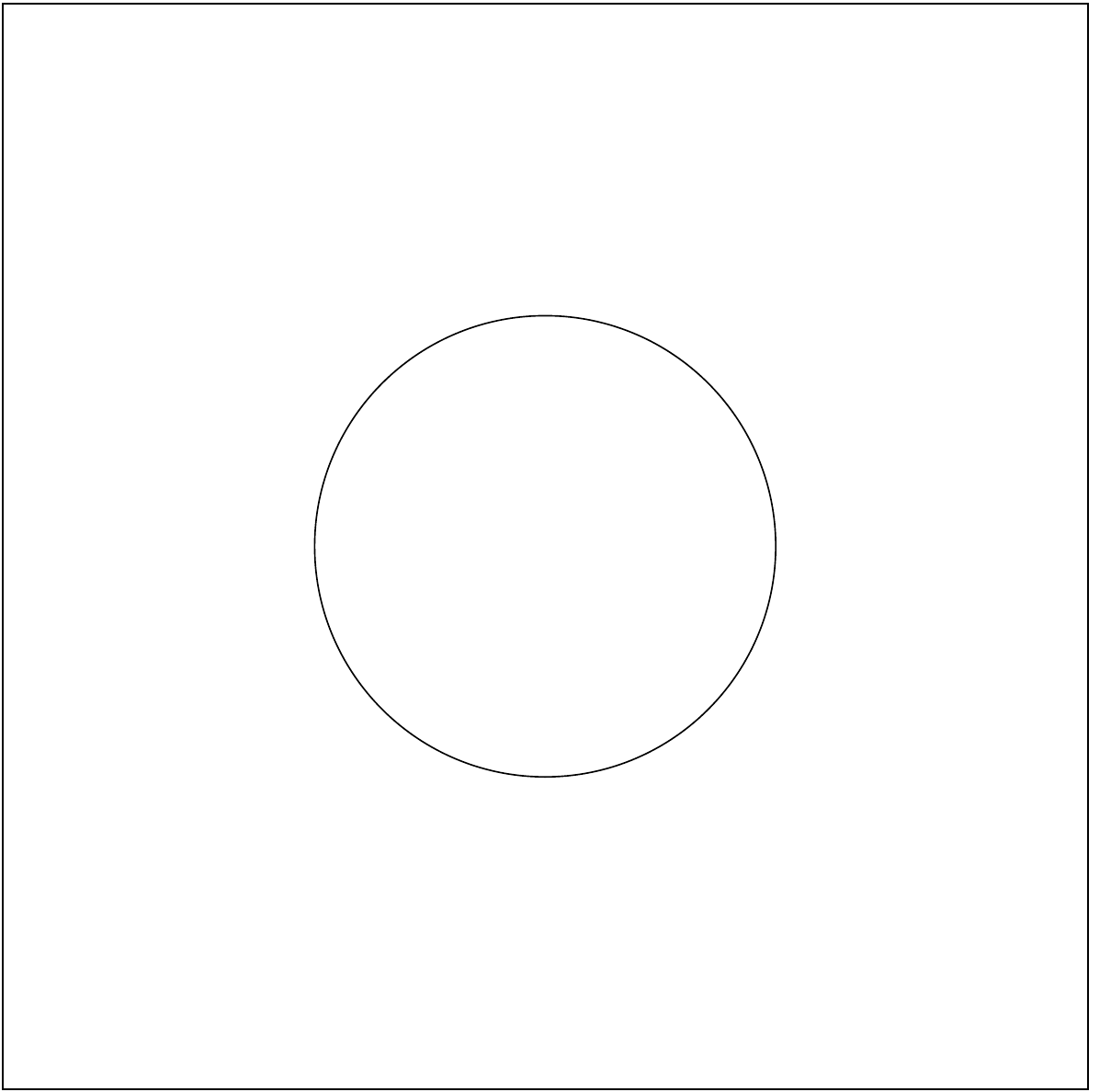}%
\end{picture}%
\setlength{\unitlength}{4144sp}%
\begin{picture}(5517,5424)(889,-5473)
\put(1621,-961){\makebox(0,0)[lb]{\smash{{\SetFigFont{20}{24.0}{\familydefault}{\mddefault}{\updefault}{\color[rgb]{0,0,0}$\mathcal{F}$}%
}}}}
\put(3286,-2716){\makebox(0,0)[lb]{\smash{{\SetFigFont{20}{24.0}{\familydefault}{\mddefault}{\updefault}{\color[rgb]{0,0,0}$\mathcal{S}$}%
}}}}
\put(6391,-556){\makebox(0,0)[lb]{\smash{{\SetFigFont{20}{24.0}{\familydefault}{\mddefault}{\updefault}{\color[rgb]{0,0,0}$\partial \mathcal{O}$}%
}}}}
\put(4726,-3346){\makebox(0,0)[lb]{\smash{{\SetFigFont{20}{24.0}{\familydefault}{\mddefault}{\updefault}{\color[rgb]{0,0,0} $\partial \mathcal{S} = \Gamma$}%
}}}}
\put(4262,-2825){\makebox(0,0)[lb]{\smash{{\SetFigFont{20}{24.0}{\familydefault}{\mddefault}{\updefault}{\color[rgb]{0,0,0} $\longleftarrow$}%
}}}}
\put(4420,-2660){\makebox(0,0)[lb]{\smash{{\SetFigFont{20}{24.0}{\familydefault}{\mddefault}{\updefault}{\color[rgb]{0,0,0} $\bn$}%
}}}}
\end{picture}%
 }
\caption{Domain for fluid and structure.} \label{figDomain}
\end{center}
\end{figure}

We denote by $\bu$ and $p$ the velocity field and the pressure of the fluid respectively. In this paper, we are interested in the following Stokes problem
\begin{eqnarray}
- \nu \Delta \bu + \nabla p & = & \bff \quad \text{in } \mathcal{F}, \label{system1} \\
\div \ \bu & = & 0 \quad \text{in }\mathcal{F}, \label{system2} \\
\bu & = & 0 \quad \text{on } \p \mathcal{O}, \label{system3} \\
\bu & = & \bg \quad \text{on } \Gamma, \label{system5}
\end{eqnarray}
where $\bff\in \mathbf{L}^2(\mathcal{F})$, $\bg\in \mathbf{H}^{1/2}(\Gamma)$.
The boundary conditions on $\Gamma$ is nonhomogeneous.
The homogeneous Dirichlet condition we consider on $\p \mathcal{O}$ has a physical sense, but can be replaced by a nonhomogeneous one, without more difficulty.\\
With regard to the incompressibility condition, the boundary datum $\bf g$ must obey
\begin{eqnarray*}
\int_{\Gamma} \bg \cdot \bn \d \Gamma & = & 0.
\end{eqnarray*}

We consider this nonhomogeneous condition as a Dirichlet one imposed on $\Gamma$. Notice that other boundary conditions are possible on $\Gamma$, such as Neumann conditions, as it is done in \cite{HaslR} where mixed boundary conditions are considered. Equation \eqref{system1} is the linearized form, in the stationary case, of the underlying incompressible Navier-Stokes equations $$ \frac{\p \bu}{\p t} + (\bu\cdot \nabla) \bu - \nu \Delta \bu + \nabla p  =  \bff \quad \text{in } \mathcal{F}.$$
The scalar constant $\nu$ denotes the dynamic viscosity of the fluid. In our presentation, for more simplicity, we only consider the stationary case, and the solid is supposed to be fixed.\\
The solution of \eqref{system1}-\eqref{system5} can be viewed as the stationary point of the Lagrangian
\begin{eqnarray}
L_0(\bu,p,\blambda) & = & \nu \int_{\mathcal{F}} \left|D(\bu) \right|^2\d \mathcal{F} - \int_{\mathcal{F}}p\div\ \bu \d \mathcal{F} - \int_{\mathcal{F}}\bff\cdot \bu\d \Gamma - \int_{\Gamma} \blambda \cdot (\bu - \bg) \d \Gamma. \nonumber \\ \label{Lag0}
\end{eqnarray}
Note that we should assume some additional smoothness in \eqref{Lag0} to make sense, for example $\bu\in \mathbf{H}^2(\mathcal{F})$, $p\in \H^1(\mathcal{F})$, $\blambda\in \mathbf{L}^2(\Gamma)$. The exact solution normally has this smoothness provided that $\bff \in \mathbf{L}^2(\mathcal{F})$ and $\bg \in \mathbf{H}^{3/2}(\Gamma)$.\\
The multiplier $\blambda$, associated with the Dirichlet condition \eqref{system5}, represents the normal trace on $\Gamma$ of the Cauchy stress tensor. Its expression is given by
\begin{eqnarray*}
\blambda(\bu,p)  =  \sigma(\bu,p)\bn
 =  2\nu D(\bu)\bn -p\bn,
\end{eqnarray*}
where
\begin{eqnarray*}
D(\bu) & = & \frac{1}{2} \left(\nabla \bu + \nabla \bu^T \right).
\end{eqnarray*}
The vector $\bn$ denotes the outward unit normal vector to $\p \mathcal{F}$ (see figure \ref{figDomain}).

\begin{remark}
Notice that if we have the incompressibility condition \eqref{system2}, then, as a multiplier for the Dirichlet condition on $\Gamma$, considering $\sigma(\bu,p)\bn$ is equivalent to considering $\displaystyle \nu\frac{\p \bu}{\p \bn} - p \bn$, as it is shown in \cite{Gunzburger} or \cite{Girault0}. It is mainly due to the equality $ \displaystyle \div \left(\nabla \bu + \nabla \bu^T \right) = \Delta \bu$, when $\div \ \bu = 0$.\\
\end{remark}

A finite element method based on the weak formulation derived from \eqref{Lag0} does not guarantee, {\it a priori}, the convergence for the quantity $\sigma(\bu,p)\bn$ in $\mathbf{L}^2(\Gamma)$. As it has been done in \cite{Barbosa1, Barbosa2}, our approach consists in considering an augmented Lagrangian in adding a quadratic term to the one given in \eqref{Lag0}, as follows
\begin{eqnarray}
L(\bu,p,\blambda) & = & L_0(\bu,p,\blambda) - \frac{\gamma}{2}\int_{\Gamma}\left|\blambda - \sigma(\bu,p)\bn \right|^2\d \Gamma \label{Lag1}.
\end{eqnarray}
The goal is to recover the optimal rate of convergence for the multiplier $\blambda$. The constant $\gamma$ represents a stabilization parameter (see numerical investigations in section \ref{secgamma}). It has to be chosen judiciously.\\


Let us give the functional spaces we use for the continuous problem (\ref{system1})-(\ref{system5}). For the velocity $\bu$ we consider the following spaces
\begin{eqnarray*}
\mathbf{V} = \left\{ \bv\in \mathbf{H}^1(\mathcal{F}) \mid \bv=0 \text{ on } \p \mathcal{O} \right\}, & \quad & \mathbf{V}_0 = \mathbf{H}_0^1(\mathcal{F}), \\
\mathbf{V}^{\#}  =  \left\{ \bv\in \mathbf{V} \mid \div \ \bv = 0 \ \text{in } \mathcal{F} \right\},
& \quad & \mathbf{V}_0^{\#}  =  \left\{ \bv\in \mathbf{H}^1_0(\mathcal{F}) \mid \div \ \bv = 0 \ \text{in } \mathcal{F} \right\}.
\end{eqnarray*}
The pressure $p$ is viewed as a multiplier for the incompressibility condition $\div \ \bu = 0$, and belongs to $\L^2(\mathcal{F})$. It is determined up to a constant that we fix such that $p$ belongs to
\begin{eqnarray*}
Q = \L^2_0(\mathcal{F}) = \left\{p \in \L^2(\mathcal{F}) \mid \int_{ \mathcal{F}}p\ \d \mathcal{F} = 0 \right\}.
\end{eqnarray*}
The functional space for the multiplier is chosen as
\begin{eqnarray*}
\mathbf{W} = \mathbf{H}^{-1/2}(\Gamma) = \left(\mathbf{H}^{1/2}(\Gamma) \right)'.
\end{eqnarray*}
\begin{remark}
If we want to impose other boundary conditions, as in \cite{HaslR} for instance, the functional spaces $\mathbf{V}_0$ and $\mathbf{H}^{1/2}(\Gamma)$ must be adapted, but there is no particular difficulty.\\
\end{remark}

The weak formulation of problem \eqref{system1}-\eqref{system5} is given by:
\begin{eqnarray}
& & \text{Find } (\bu,p,\blambda) \in \mathbf{V} \times Q \times \mathbf{W} \text{ such that }\nonumber \\
& & \left\{ \begin{array} {lll}
a(\bu,\bv) + b(\bv,p) + c(\bv,\blambda) = \mathcal{L}(\bv) & \quad & \forall \bv \in \mathbf{V}, \\
b(\bu,q)  = 0 & \quad & \forall q \in Q, \\
c(\bu,\bmu) = \mathcal{G}(\bmu), & \quad & \forall \bmu \in \mathbf{W},
\end{array} \right. \label{pbexact}
\end{eqnarray}
where
\begin{eqnarray}
a(\bu,\bv) & = & 2\nu \int_{\mathcal{F}} D(\bu):D(\bv)\d \mathcal{F}, \label{defa}\\
b(\bu,q) & = & -\int_{\mathcal{F}}q\div \ \bu \d \mathcal{F}, \label{defb} \\
c(\bu,\bmu) & = & -\int_{\Gamma}\bmu \cdot \bu \d \Gamma, \label{defc} \\
\mathcal{L}(\bv) & = & \int_{\mathcal{F}}\bff\cdot \bv\d \mathcal{F}, \label{defaL} \\
\mathcal{G}(\bmu) & = & -\int_{\Gamma} \bmu \cdot \bg \d \Gamma. \label{defG}
\end{eqnarray}
The expression $D(\bu):D(\bv) = \trace \left(D(\bu)D(\bv)^T\right)$ denotes the classical inner product for matrices. Let us note that Problem \eqref{pbexact} is well-posed (see \cite{Gunzburger} for instance). The solution of Problem \eqref{system1}-\eqref{system5} can be viewed as the stationary point of the Lagrangian on $\mathbf{V} \times Q \times \mathbf{W}$
\begin{eqnarray}
L_0(\bu,p,\blambda) & = & \nu \int_{\mathcal{F}} \left|D(\bu) \right|^2\d \mathcal{F} - \int_{\mathcal{F}}p\div\ \bu \d \mathcal{F} - \int_{\mathcal{F}}\bff\cdot \bu \d \mathcal{F}  - \int_{\Gamma} \blambda \cdot (\bu-\bg) \d \Gamma. \nonumber \\ \label{Lago}
\end{eqnarray}


\section{The fictitious domain method without stabilization} \label{seccut}
\subsection{Presentation of the method} \label{seccutpres}
The fictitious domain for the fluid is considered on the whole domain $\mathcal{O}$. Let us introduce three discrete finite element spaces, $\tilde{\mathbf{V}}^h \subset \mathbf{H}^1(\mathcal{O})$ and $\tilde{Q}^h \subset \L^2_0(\mathcal{O})$ on the fictitious domain, and $\tilde{\mathbf{W}}^h \subset \mathbf{L}^2(\mathcal{O})$. Since $\mathcal{O}$ can be a rectangular domain, this spaces can be defined on the same structured mesh, that can be chosen uniform (see figure \ref{meshimage2}). The construction of the mesh is highly simplified (no particular mesh is required). We set
\begin{eqnarray}
\tilde{\mathbf{V}}^h & = & \left\{\bv^h \in C(\overline{\mathcal{O}})\mid \bv^h_{\left| \p \mathcal{O}\right.} = 0, \  \bv^h_{\left| T\right.} \in P(T), \ \forall T \in \mathcal{T}^h \right\}, \label{defvtilde}
\end{eqnarray}
where $P(T)$ is a finite dimensional space of regular functions such that $P(T) \supseteq P_k(T)$ for some integer $k \geq 1$. For more details, see \cite{Ern} for instance. The mesh parameter stands for $\displaystyle h = \max_{T\in \mathcal{T}^h} h_T$, where $h_T$ is the diameter of $T$.\\

\begin{figure}[!h]
\centering
\includegraphics[scale = 0.50]{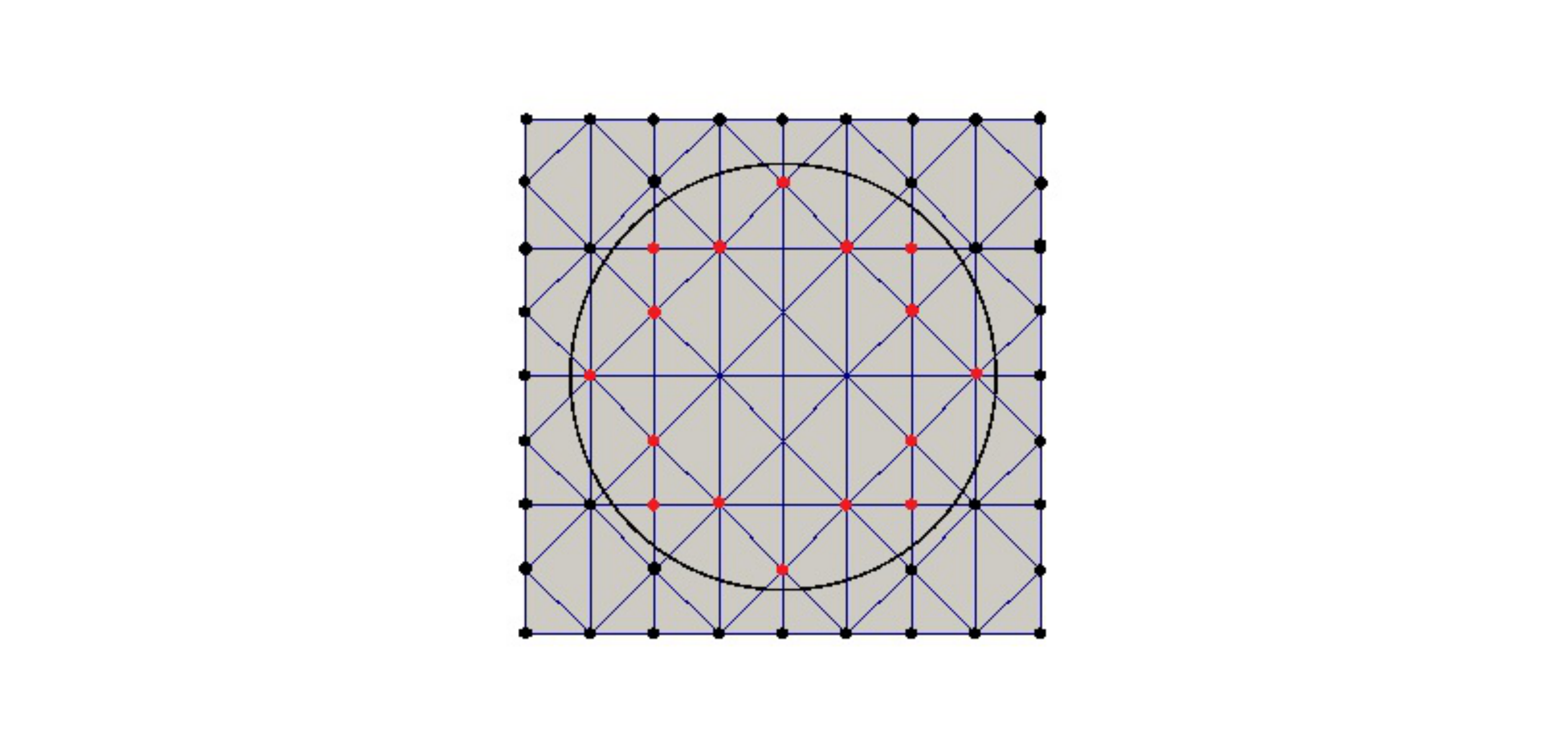}
\caption{An example of a mesh on a fictitious domain. The standard degrees of freedom are in black, the virtual ones are in red, and the remaining ones are removed.}
\label{meshimage2}
\end{figure}

\FloatBarrier

Then we define
\begin{eqnarray*}
\mathbf{V}^h := \tilde{\mathbf{V}}^h_{\left| \mathcal{F} \right.}, \quad Q^h := \tilde{Q}^h_{\left|\mathcal{F}\right.}, \quad  \mathbf{W}^h := \tilde{\mathbf{W}}^h_{\left| \Gamma \right.},
\end{eqnarray*}
which are natural discretizations of $\mathbf{V}$, $\L^2(\mathcal{F})$ and $\mathbf{H}^{-1/2}(\Gamma)$, respectively. \textcolor{black}{This approach is equivalent to XFEM as proposed in \cite{Choi2010} or \cite{Gerstenberger2008} where the standard FEM basis functions are multiplied by the Heaviside function ($H(\x) = 1$ for $\x \in \mathcal{F}$ and $H(\x)=0$ for $\x\in \mathcal{O} \setminus \mathcal{F}$)and the products are substituted in the variational formulation of the problem. Thus the degrees of freedom inside the fluid domain $\mathcal{F}$ are used in the same way as in the standard FEM, whereas the degrees of freedom in the solid domain $\mathcal{S}$ at the vertices of the elements cut by the interface (the so called virtual degrees of freedom) do not define the field variable at these nodes, but they are necessary to define the fields on $\mathcal{F}$ and to compute the integrals over $\mathcal{F}$. The remaining degrees of freedom, corresponding to the basis functions with support completely outside of the fluid, are eliminated (see figure \ref{meshimage2}). We refer to the papers mentioned above for more details.}


\begin{figure}[!h]
\centering
\includegraphics[scale = 0.50]{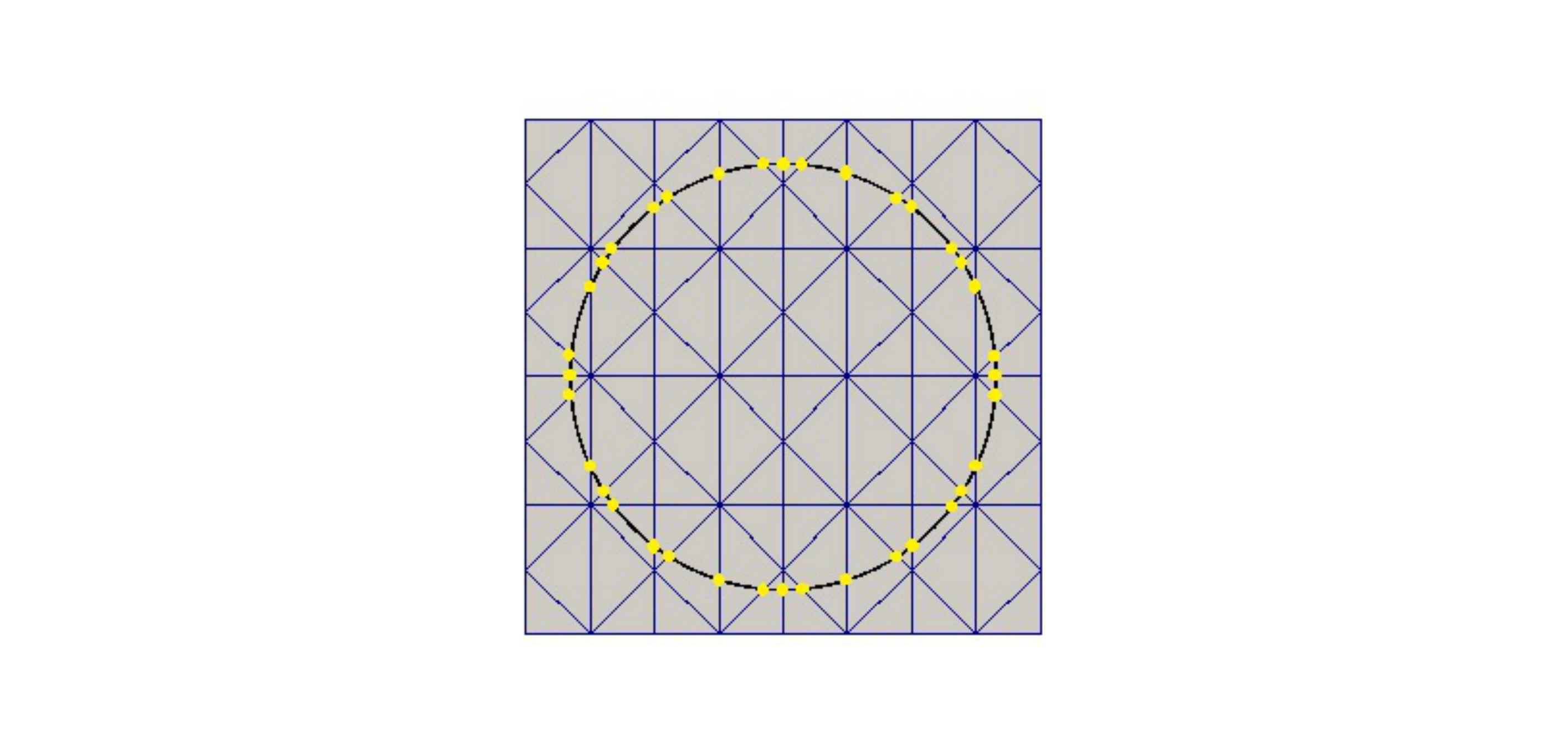}
\caption{Base nodes used for the multiplier space $\mathbf{W}^h$.}
\label{meshimage3}
\end{figure}

\FloatBarrier

An approximation of problem \eqref{pbexact} is defined as follows:
\begin{eqnarray}
& & \text{Find } (\bu^h,p^h,\blambda^h) \in \mathbf{V}^h \times Q^h \times \mathbf{W}^h \text{ such that } \nonumber \\
& & \left\{ \begin{array} {lll}
a(\bu^h,\bv^h)+b(\bv^h,p^h) + c(\bv^h,\blambda^h) = \mathcal{L}(\bv^h) & \quad & \forall \bv^h \in \mathbf{V}^h, \\
b(\bu^h,q^h)  = 0 & \quad & \forall q^h \in Q^h, \\
c(\bu^h,\bmu^h) = \mathcal{G}(\bmu^h) & \quad & \forall \bmu^h \in \mathbf{W}^h.
\end{array} \right. \label{pbapprox}
\end{eqnarray}

In matrix notation, the previous formulation corresponds to
\begin{eqnarray*}
\left( \begin{matrix}
A^0_{\bu\bu} & A^0_{\bu p} & A^0_{\bu\blambda} \\
A^{{0}^T}_{\bu p} & 0 & 0 \\
A^{{0}^T}_{\bu\blambda} &0& 0\\
\end{matrix} \right)
\left( \begin{matrix}
\boldsymbol{U} \\
\boldsymbol{P} \\
\boldsymbol{\Lambda} \\
\end{matrix} \right)
& = &
\left( \begin{matrix}
\boldsymbol{F^0} \\
0 \\
\boldsymbol{G^0}
\end{matrix} \right),
\end{eqnarray*}
where $\boldsymbol{U}$, $\boldsymbol{P}$ and $\boldsymbol{\Lambda}$ are the degrees of freedom of $\bu^h$, $p^h$ and $\blambda^h$ respectively. As it is done in \cite{Brezzi} or \cite{Ern} for instance, these matrices $A^0_{\bu\bu}$, $A^0_{\bu p}$,
$A^0_{\bu\blambda}$ and vectors $\boldsymbol{F^0}$, $\boldsymbol{G^0}$
are the discretization of (\ref{defa})-(\ref{defG}), respectively. Denoting $\{ \bvarphi_i\}$, $\{ \chi_i\}$ and $\{ \bpsi_i \}$ the selected basis functions of spaces $\tilde{\mathbf{V}}^h$, $\tilde{Q}^h$ and $\tilde{\mathbf{W}}^h$ respectively, we have
\begin{eqnarray*}
& \left(A^0_{\bu\bu}\right)_{ij} = 2\nu \int_{\mathcal{F}}D(\bvarphi_i):D(\bvarphi_j)\d \mathcal{F} ,
 \quad  \left(  A^0_{\bu p} \right)_{ij} = - \int_{\mathcal{F}} \chi_j \div \ \bvarphi_i \d \mathcal{F},
 \quad  \left(  A^0_{\bu\blambda} \right)_{ij} = - \int_{\Gamma} \bvarphi_i \cdot \bpsi_j \d \Gamma , &\\
& \left(\boldsymbol{F^0} \right)_i = \int_{\mathcal{F}} \bff \cdot \bvarphi_i \d \mathcal{F},
 \quad \left(\boldsymbol{G^0} \right)_i =- \int_{\Gamma} \bg \cdot \bpsi_i \d \Gamma. &
\end{eqnarray*}

\subsection{Convergence analysis} \label{seccvstab0}
Let us define
\begin{eqnarray*}
\mathbf{V}^{h}_0 & = & \left\{ \bv^h \in \mathbf{V}^h \mid c(\bv^h,\bmu^h) = 0 \ \forall \bmu^h \in \mathbf{W}^h \right\}, \\
\mathbf{V}^{h}_g & = & \left\{ \bv^h \in \mathbf{V}^h \mid c(\bv^h,\bmu^h) = c(\bv^h,g) \ \forall \bmu^h \in \mathbf{W}^h \right\}, \\
\mathbf{V}^{\# , h} & = & \left\{ \bv^h \in \mathbf{V}^h  \mid b(\bv^h,q^h) = 0 \ \forall q^h \in Q^h \right\}, \\
\mathbf{V}^{\# , h}_0 & = & \left\{ \bv^h \in \mathbf{V}^h \mid b(\bv^h,q^h) = 0 \ \forall q^h \in Q^h, c(\bv^h,\bmu^h) = 0 \ \forall \bmu^h \in \mathbf{W}^h \right\}.
\end{eqnarray*}
The spaces $\mathbf{V}^{h}_0$, $\mathbf{V}^{\# , h}$ and $\mathbf{V}^{\# , h}_0$ can be viewed as the respective discretizations of the spaces $\mathbf{V}_0$, $\mathbf{V}^{\#}$ and $\mathbf{V}^{\#}_0$.\\
Let us assume that the following inf-sup condition is satisfied, for some constant $\beta >0$ independent of $h$:
\begin{flushleft}
\hfill \\
$\mathbf{H1} \qquad \qquad \displaystyle \inf_{0 \neq q^h \in Q^h} \sup_{0 \neq \bv^h \in \mathbf{V}^h_0} \frac{b(\bv,q)}{\| \bv^h \|_{\mathbf{V^h}} \| q^h \|_{Q^h}}  \geq  \beta. $ \\ 
\hfill \\
\end{flushleft}
Note that this inf-sup condition concerns only the couple $(\bu,p)$, and it implies the following property
\begin{eqnarray}
& & \overline{q}^h \in Q^h: \ b(\bv^h, \overline{q}^h) = 0 \ \forall \bv^h \in \mathbf{V}^h_0 \Longrightarrow \overline{q}^h = 0. \label{hypHR1}
\end{eqnarray}
We shall further assume that the spaces $\tilde{\mathbf{V}}^h$, $\tilde{Q}^h$ and $\tilde{\mathbf{W}}^h$ are chosen in such a way that the following condition is satisfied, for all $h>0$
\begin{flushleft}
\hfill \\
$\mathbf{H2} \qquad \qquad \displaystyle \overline{\bmu}^h \in \mathbf{W}^h: \ c(\bv^h, \overline{\bmu}^h) = 0 \ \forall \bv^h \in \mathbf{V}^{h}\Longrightarrow \overline{\bmu}^h = 0.$ \\ 
\hfill \\
\end{flushleft}
Note that this hypothesis is not as strong as an inf-sup condition for the couple {\it velocity/multiplier}. It only demands that the space $\mathbf{V}^h$ is rich enough with respect to the space $\mathbf{W}^h$.\\


\begin{remark}
We assume only the inf-sup condition for the couple {\it velocity/pressure}, not the one for the couple {\it velocity/multiplier}. Indeed, the purpose of our work is to stabilize the multiplier associated with the Dirichlet condition on $\Gamma$, not the multiplier associated with the incompressibility condition. The stabilization of the pressure - on the domain $\mathcal{F}$ - would be another issue (see page 424 of \cite{Quarteroni} for instance).
\end{remark}

\begin{lemma} \label{lemmaVoh} \label{lemmapetree}
The bilinear form $a$ introduced in \eqref{defa} as
\begin{eqnarray*}
a : (\bu,\bv) \mapsto 2\nu\int_{\mathcal{F}} D(\bu) : D(\bv) \d \mathcal{F}
\end{eqnarray*}
is uniformly $\mathbf{V}^{h}$-elliptic, that is to say there exists $\alpha > 0$ independent of $h$ such that for all $\bv^h \in \mathbf{V}^{h}$
\begin{eqnarray*}
a(\bv^h, \bv^h) & \geq & \alpha \left\| \bv^h \right\|^2_{\mathbf{V}} .
\end{eqnarray*}
\end{lemma}
\begin{proof}
Notice that $\mathbf{V}^{h} \subset \mathbf{V}$. Then it is sufficient to prove that the bilinear form $a$ is coercive on the space $\mathbf{V}$, that is to say there exists $\alpha > 0$ such that for all $\bv \in \mathbf{V}$
\begin{eqnarray*}
a(\bv,\bv) & \geq & \alpha \|\bv \|^2_{\mathbf{V}} .
\end{eqnarray*}
By absurd, suppose that for all $n \in \mathbb{N}$ there exists $(\bv_n)_n$ such that
\begin{eqnarray*}
n\|D(\bv_n)\|_{[\L^2(\mathcal{F})]^4} & < & \|\bv_n\|_{\mathbf{V}}.
\end{eqnarray*}
Without loss of generality, we can assume that $\|\bv_n\|_{\mathbf{V}} =1$. In particular, $D(\bv_n)$ converges to $0$ in $[\L^2(\mathcal{F})]^4$. Then, from the Rellich's theorem, we can extract a subsequence $\bv_m$ which converges in $\mathbf{L}^2(\mathcal{F})$. Using the fact that $\div \ \bv_m = 0$, the Korn inequality (see \cite{Ern} for instance) enables us to write
\begin{eqnarray*}
\|\bv_m - \bv_p\|^2_{\mathbf{H}^1(\mathcal{F})} & \leq & C\left(\|\bv_m - \bv_p\|^2_{\mathbf{L}^2(\mathcal{F})} + \left\| D(\bv_m) - D(\bv_p) \right\|^2_{[\L^2(\mathcal{F})]^4} \right),
\end{eqnarray*}
where $C$ denotes a positive constant\footnote{In the following, the symbol $C$ will denote a generic positive constant which does not depend on the mesh size $h$. It can depend, however, on the geometry of $\cal F$ and $\Gamma$, on the physical parameters, on the mesh regularity and on other quantities clear from the context. It can take different values at different places.}. This implies that $(\bv_m)_m$ is a Cauchy sequence in $\mathbf{H}^1(\mathcal{F})$. Thus it converges to some $\bv_{\infty}$ which satisfies $\|D(\bv_{\infty})\|_{\mathbf{L}^2(\mathcal{F})}=0$. The trace theorem implies that we have also $\bv_{\infty} = 0$ on $\p \mathcal{O}$. Let us notice that $\bv \mapsto \|D(\bv)\|_{[\L^2(\mathcal{F})]^4}$ is a norm on $\mathbf{V}$. Indeed, if $\|D(\bv_{\infty})\|_{[\L^2(\mathcal{F})]^4} = 0$, then $\bv_{\infty}$ is reduced to a rigid displacement, that is to say $\bv_{\infty} = l + \omega \wedge x$ in $\mathcal{F}$. Then, the condition $\bv_{\infty}=0$ on $\p \mathcal{O}$ leads us to $\bv_{\infty}=0$. It belies the fact that $\|\bv_m\|_{\mathbf{V}} =1$.
\end{proof}

\begin{proposition}
Assume that the properties $\mathbf{H1}$ and $\mathbf{H2}$ are satisfied. Then there exists a unique solution $(\bu^h, p^h, \blambda^h)$ to Problem \eqref{pbapprox}.
\end{proposition}

\begin{proof}
Since Problem \eqref{pbapprox} is of finite dimension, existence of the solution will follow from its uniqueness. To prove uniqueness, it is sufficient to consider the case $\bff=0$ and $\bg=0$, and to prove that it leads to $(\bu^h, p^h,\blambda^h) = (0,0,0)$. The last two equations in \eqref{pbapprox} show then immediately that $\bu^h\in{\bf V}^{\#,h}_0$, so that taking $\bv^h=\bu^h$ in the first equation leads to $\bu^h = 0$ by Lemma \ref{lemmaVoh}. Taking any test function from $\mathbf{V}^{h}_0$ in the first equation of \eqref{pbapprox} shows now that $p^h = 0$, by condition \eqref{hypHR1} (hypothesis {\bf H1}). And finally the same equation yields $\blambda^h = 0$ by Hypothesis $\mathbf{H2}$.
\end{proof}

We recall the following basic result from the theory of saddle point problems \cite{Ern, Girault}.
\begin{lemma} \label{lemmaGR}
Let $X$ and $M$ be Hilbert spaces and $A(\cdot, \cdot):X \times X \to\mathbb{R}$ and $B(\cdot, \cdot):X \times M \to\mathbb{R}$ be bounded bilinear forms such that $A$ is coercive
$$
A(u,u) \ge \alpha \| u \|_X^2, \quad \forall  u \in X
$$
and $B$ has the following inf-sup property
\begin{eqnarray*}
\inf_{0 \neq q \in M} \sup_{0 \neq u \in X} \frac{B(u,q)}{\| u \|_X \| q \|_M} & \geq & \beta,
\end{eqnarray*}
with some $\alpha, \beta>0$. Then, for all $\phi \in X'$ and $\psi \in M'$, the problem:
\begin{eqnarray*}
& & \text{Find $u \in X$ and $p\in M$ such that} \\
& & \left\{ \begin{array} {lll}
A(u,v)+ B(v,p)  =  \langle {\phi} , v \rangle, & \quad & \forall v \in X
\\
B(u,q)    =  \langle {\psi} , q \rangle, &\quad & \forall q \in M
\end{array} \right.
\end{eqnarray*}
has a unique solution which satisfies
$$
\| u \|_X + \| q \|_M \le C (\| \phi \|_{X'} + \| \psi \|_{M'})
$$
with a constant $C>0$ that depends only on $\alpha, \beta$ and on the norms of $A$ and $B$.
\end{lemma}

We can now prove the abstract error estimate for velocity and pressure.

\begin{proposition}
Assume Hypothesis $\mathbf{H1}$. Let $(\bu,p,\blambda)$ and $(\bu^h,p^h,\blambda^h)$ be solutions to Problems \eqref{pbexact} and \eqref{pbapprox} respectively. There exists a constant $C>0$ independent of $h$ such that
\begin{eqnarray}
\|\bu-\bu^h \|_{\mathbf{V}} + \| p -p^h \|_{\L^2(\mathcal{F})} & \leq & C\left(\inf_{\bv^h\in \mathbf{V}^h_g} \|\bu-\bv^h \|_{\mathbf{V}} 
\right.\nonumber \\
& & \left.+ \inf_{q^h \in Q^h} \| p -q^h \|_{\L^2(\mathcal{F})} + \inf_{\bmu^h \in W^h} \| \blambda -\bmu^h \|_{\mathbf{H}^{-1/2}(\Gamma)} \right).\nonumber \\
\label{estborne}
\end{eqnarray}
\end{proposition}

\begin{proof}
Take any $\bv^h \in \mathbf{V}^h_g$, $q^h \in Q^h$ and $\bmu^h \in \mathbf{W}^h$. Comparing the first lines in systems \eqref{pbexact} and \eqref{pbapprox}, we can write
\begin{equation}\label{1st}
a(\bu^h-\bv^h,\bw^h) + b(\bw^h,p^h-q^h) = a(\bu-\bv^h,\bw^h) + b(\bw^h,p-q^h)+ c(\blambda - \bmu^h,\bw^h) \quad \forall \bw^h \in \mathbf{V}_0^h.
\end{equation}
We have used here the fact that $c(\blambda^h,\bw^h) = c(\bmu^h,\bw^h)=0$ for all $\bw^h \in \mathbf{V}_0^h$. Similarly, the second lines in systems \eqref{pbexact} and \eqref{pbapprox} imply
\begin{equation}\label{2nd}
b(\bu^h-\bv^h,s^h) = b(\bu-\bv^h,s^h) \quad \forall s^h \in {Q}^h.
\end{equation}
Now consider the problem:
\begin{eqnarray*}
& & \text{Find ${\bf x}^h\in{\bf V}^h_0$ and $t^h\in Q^h$ such that} \\
& &  \left\{ \begin{array} {lll}
a({\bf x}^h,\bw^h) + b(\bw^h,t^h) = a(\bu-\bv^h,\bw^h) + b(\bw^h,p-q^h)+ c(\blambda - \bmu^h,\bw^h) & & \forall \bw^h \in \mathbf{V}_0^h,
\\
b({\bf x}^h,s^h) = b(\bu-\bv^h,s^h) & & \forall s^h \in {Q}^h.
\end{array} \right.
\end{eqnarray*}
Using Lemma \ref{lemmaGR} with $A=a$, $B=b$, $X = \mathbf{V}_0^h$ and $M = Q^h$, the solution $({\bf x}^h, t^h)$ exists and is unique. Moreover, it satisfies
\begin{eqnarray*}
\|{\bf x}^h\|_{\mathbf{V}} +\|t^h \|_{\L^2(\mathcal{F})} & \leq &
 C \left( \|\bu-\bv^h\|_{\mathbf{V}} + \|p-q^h\|_{\L^2(\mathcal{F})} + \|\blambda-\bmu^h\|_{{\mathbf{H}^{-1/2}}(\Gamma)} \right).
\end{eqnarray*}
Comparing the system of equations for $({\bf x}^h, t^h)$ with \eqref{1st}--\eqref{2nd} and noting that $\bu-\bv^h \in{\bf V}^h_0$, we can identify
$$
{\bf x}^h=\bu^h-\bv^h, \quad t^h=p^h-q^h.
$$
In combination with the triangle inequality, this gives
\begin{eqnarray*}
\|\bu-\bu^h\|_{\mathbf{V}} +\|p-p^h \|_{\L^2(\mathcal{F})} & \leq &
 C \left( \|\bu-\bv^h\|_{\mathbf{V}} + \|p-q^h\|_{\L^2(\mathcal{F})} + \|\blambda-\bmu^h\|_{{\mathbf{H}^{-1/2}}(\Gamma)} \right).
\end{eqnarray*}
Since $\bv^h \in \mathbf{V}^h_g$, $q^h \in Q^h$ and $\bmu^h \in \mathbf{W}^h$ are arbitrary, this is equivalent to the desired result.

\end{proof}

In summary, the results of this section tell us that, under Hypotheses $\mathbf{H1}$ and $\mathbf{H2}$, Problem \eqref{pbapprox} has a unique solution which satisfies the a priori estimate \eqref{estborne}. However, we have no estimate for the multiplier $\blambda^h$.

\subsection{The theoretical order of convergence}
The estimation of the convergence rate proposed for the Poisson problem in \cite{HaslR} can be straightforwardly transposed to the Stokes problem. Proposition 3 of \cite{HaslR} ensures an order of convergence at least equal to $\sqrt{h}$. It can be adapted to our case as follows.\\

\begin{proposition}
Assume Hypotheses $\mathbf{H1}$, $\mathbf{H2}$. Let $(\bu,p,\blambda)$ be the solution of Problem \eqref{pbexact} for $\bg = 0$, such that $\bu \in \mathbf{H}^{2+\varepsilon}(\mathcal{F})\cap \mathbf{H}^1_0(\mathcal{F})$ for some $\varepsilon > 0$. Assume that
\begin{eqnarray*}
\inf_{q^h \in Q^h} \| p -q^h \|_{Q} & \leq & Ch^{\delta}, \\
\inf_{\bmu^h \in W^h} \| \blambda -\bmu^h \|_{\mathbf{W}} & \leq & Ch^{\delta},
\end{eqnarray*}
for some $\delta \geq 1/2$. Then
\begin{eqnarray*}
\left\|\bu-\bu^h \right\|_{\mathbf{V}} + \| p -p^h \|_{\L^2(\mathcal{F})} & \leq & C\sqrt{h}.
\end{eqnarray*}
\end{proposition}

\begin{proof}
As is shown in \cite{HaslR}, Section 3, for any $\bu \in \mathbf{H}^{2+\varepsilon}(\mathcal{F}) \cap \mathbf{H}^1_0(\mathcal{F})$ there exists a finite element interpolating function $\bv^h \in {\bf V}^h_0$ such that
\begin{equation}\label{sqrth}
\| \bu - \bv^h \|_{\mathbf{V}}  \leq  C \sqrt{h}.
\end{equation}
In fact, $\bv^h$ is constructed as a standard interpolating vector of $(1-\eta_h)\bu$ where $\eta_h$ is a cut-off function equal to $1$ in a vicinity of the boundary $\Gamma$, more precisely in a band of width $\frac{3h}{2} $, so that $\bv^h$ vanishes on all the triangles cut by $\Gamma$. This ensures that $\bv^h$ vanishes on $\Gamma$ so that $\bv^h \in {\bf V}^h_0$. Now, the estimate of the present proposition follows from  \eqref{estborne} combined with \eqref{sqrth} (note that $\mathbf{V}^{h}_g=\mathbf{V}^{h}_0$ under our assumptions) and the hypotheses on the interpolating functions $q^h$ and $\bmu^h$.
\end{proof}

Let us quote other references that treat of this kind of phenomena, as \cite{Girault, Ramiere, Ramiere2, Maury}. We note, however, that the estimate of the order of convergence in $\sqrt{h}$ seems too pessimistic in view of the numerical tests presented in \cite{HaslR} for the Poisson problem (with the possible exception of the lowest order finite elements). In our numerical experiments for the Stokes problem, we do not observe the order of convergence as slow as $\sqrt{h}$.


\section{The fictitious domain method with stabilization} \label{secstab}
\subsection{Presentation of the method}
The main purpose of the stabilization method we introduce consists in recovering the convergence on the multiplier $\blambda$. For that, the idea is to insert in our formulation a term which takes into account this requirement. Following the idea used in \cite{Barbosa1, Barbosa2}, we extend the classical Lagrangian $L_0$ given in \eqref{Lago}, as
\begin{eqnarray*}
L(\bu,p,\blambda) & = & \nu \int_{\mathcal{F}} \left|D(\bu) \right|^2\d \mathcal{F} - \int_{\mathcal{F}}p\div \ \bu\d \mathcal{F} - \int_{\mathcal{F}}\bff\cdot \bu\d \mathcal{F} - \int_{\Gamma} \blambda \cdot (\bu-\bg) \d \Gamma \\
& & - \frac{\gamma}{2}\int_{\Gamma}\left|\blambda - \sigma(\bu,p)\bn \right|^2\d \Gamma.
\end{eqnarray*}
Note that this extended Lagrangian coincides with the previous one on an exact solution. The quadratic term so added enables us to take into account an additional cost. Minimizing $L$ leads to forcing $\blambda$ to reach the desired value corresponding to $\sigma(\bu,p)\bn$. The constant $\gamma >0$ represents the importance we give to this demand. However, notice that this additional term affects the positivity of $L$. This is the reason why we cannot choose $\gamma$ too large, and so this approach is not a penalization method. We discuss on this choice of $\gamma$ in section \ref{secgamma}.\\
The computations of the first variations leads us to
\begin{eqnarray*}
\frac{\delta L}{\delta \bu}(\bv) & = & 2\nu\int_{\mathcal{F}}D(\bu):D(\bv)\d \mathcal{F} - \int_{\mathcal{F}}p\div \ \bv\d \mathcal{F} - \int_{\mathcal{F}}\bff\cdot \bv\d \mathcal{F} - \int_{\Gamma} \blambda\cdot \bv\d \Gamma \\
  & & + 2\nu\gamma \int_{\Gamma}(\blambda-\sigma(\bu,p)\bn)\cdot \left( D(\bv)\bn\right)\d \Gamma , \\
\frac{\delta L}{\delta p}(q) & = & - \int_{\mathcal{F}}q\div\ \bu\d \mathcal{F} -\gamma \int_{\Gamma}q\left(\blambda - \sigma(\bu,p)\bn\right)\cdot \bn \d \Gamma, \\
\frac{\delta L}{\delta \blambda}(\bmu) & = & - \int_{\Gamma}\bmu\cdot (\bu-\bg)\d \Gamma - \gamma \int_{\Gamma}\left(\blambda-\sigma(\bu,p)\bn\right)\cdot \bmu \d \Gamma.
\end{eqnarray*}
Thus the stabilized formulation is:
\begin{eqnarray}
& & \text{Find } (\bu,p,\blambda) \in \mathbf{V} \times Q \times \mathbf{W} \text{ such that } \nonumber \\
& & \left\{ \begin{array} {lll}
\mathcal{A}((\bu,p,\blambda);\bv)  = \mathcal{L}(\bv) & \quad & \forall \bv \in \mathbf{V}, \\
\mathcal{B}((\bu,p,\blambda);q)  = 0 & \quad & \forall q \in Q, \\
\mathcal{C}((\bu,p,\blambda);\bmu) = \mathcal{G}(\bmu), & \quad & \forall \bmu \in \mathbf{W},
\end{array} \right. \label{FVaugmented}
\end{eqnarray}
where
\begin{eqnarray*}
\mathcal{A}((\bu,p,\blambda);\bv) & = & 2\nu\int_{\mathcal{F}}D(\bu):D(\bv)\d \mathcal{F} - \int_{\mathcal{F}}p\div \ \bv\d \mathcal{F}  - \int_{\Gamma} \blambda\cdot \bv\d \Gamma \\
& & -4\nu^2\gamma \int_{\Gamma}\left( D(\bu)\bn \right)\cdot \left( D(\bv)\bn \right)\d \Gamma +2\nu \gamma \int_{\Gamma}p \left(D(\bv)\bn\cdot \bn \right)\d \Gamma +2\nu \gamma \int_{\Gamma} \blambda \cdot \left(D(\bv)\bn\right)\d \Gamma , \\
\mathcal{B}((\bu,p,\blambda);q) & = & - \int_{\mathcal{F}}q\div\ \bu\d \mathcal{F} +2\nu \gamma \int_{\Gamma}q\left(D(\bu)\bn\cdot \bn \right)\d \Gamma -\gamma \int_{\Gamma}pq \d \Gamma - \gamma \int_{\Gamma} q\blambda \cdot \bn\d \Gamma , \\
\mathcal{C}((\bu,p,\blambda);\bmu) & = & -\int_{\Gamma} \bmu \cdot \bu \d \Gamma +2\nu \gamma \int_{\Gamma}\bmu \cdot (D(\bu)\bn)\d \Gamma -\gamma \int_{\Gamma}p(\bmu\cdot \bn)\d \Gamma - \gamma \int_{\Gamma} \blambda \cdot \bmu \d \Gamma .
\end{eqnarray*}


In matrix notation, the previous formulation corresponds to
\begin{eqnarray*}
\left( \begin{matrix}
A_{\bu\bu} & A_{\bu p} & A_{\bu\blambda} \\
A^T_{\bu p} & A_{pp} & A_{p\blambda} \\
A^T_{\bu\blambda} & A^T_{p\blambda} & A_{\blambda \blambda}
\end{matrix} \right)
\left( \begin{matrix}
\boldsymbol{U} \\
\boldsymbol{P} \\
\boldsymbol{\Lambda} \\
\end{matrix} \right)
& = &
\left( \begin{matrix}
\boldsymbol{F} \\
0 \\
\boldsymbol{G}
\end{matrix} \right),
\end{eqnarray*}
where $\boldsymbol{U}$, $\boldsymbol{P}$ and $\boldsymbol{\Lambda}$ are already introduced in section \ref{seccutpres}. As it is done in \cite{Brezzi} or \cite{Ern} for instance, these matrices are discretizations of the following bilinear forms
\begin{eqnarray*}
\mathcal{A}_{\bu\bu} : (\bu,\bv) & \longmapsto & 2\nu\int_{\mathcal{F}}D(\bu):D(\bv)\d \mathcal{F} - 4\nu^2\gamma \int_{\Gamma}\left( D(\bu)\bn \right)\cdot \left( D(\bv)\bn\right)\d \Gamma , \\
\mathcal{A}_{\bu p} : (\bv,p) & \longmapsto & - \int_{\mathcal{F}}p\div \ \bv \d \mathcal{F} + 2\nu \gamma \int_{\Gamma}p \left(D(\bv)\bn\cdot \bn \right)\d \Gamma , \\
\mathcal{A}_{\bu\blambda} : (\bv,\blambda) & \longmapsto & - \int_{\Gamma} \blambda\cdot \bv\d \Gamma + 2\nu \gamma \int_{\Gamma} \blambda \cdot \left(D(\bv)\bn\right)\d \Gamma , \\
\mathcal{A}_{pp} : (p,q) & \longmapsto & -\gamma \int_{\Gamma}pq \d \Gamma , \\
\mathcal{A}_{p\blambda} : (q,\blambda) & \longmapsto & - \gamma \int_{\Gamma} q\blambda \cdot \bn\d \Gamma , \\
\mathcal{A}_{\blambda\blambda} : (\blambda,\bmu) & \longmapsto & -\gamma \int_{\Gamma} \blambda \cdot \bmu \d \Gamma ,
\end{eqnarray*}
and the vectors $\boldsymbol{F}$ and $\boldsymbol{G}$ are the discretization of the following linear forms
\begin{eqnarray*}
\mathcal{L} : \bv & \longmapsto &  \int_{\mathcal{F}} \bff\cdot \bv \d \Gamma , \\
\mathcal{G} : \bmu & \longmapsto & -\int_{\Gamma} \bmu \cdot \bg \d \Gamma.
\end{eqnarray*}
Denoting $\{ \bvarphi_i\}$, $\{ \chi_i\}$ and $\{ \bpsi_i \}$ the selected basis functions of spaces $\tilde{\mathbf{V}}^h$, $\tilde{Q}^h$ and $\tilde{\mathbf{W}}^h$ respectively, we have
\begin{eqnarray*}
& & \left(A_{\bu\bu}\right)_{ij}  =  2\nu \int_{\mathcal{F}}D(\bvarphi_i):D(\bvarphi_j)\d \mathcal{F}
- 4 \nu^2 \gamma \int_{\Gamma} (D(\bvarphi_i) \bn) \cdot (D(\bvarphi_j \bn )) \d \Gamma, \\
& & \left(  A_{\bu p} \right)_{ij}  =  - \int_{\mathcal{F}} \chi_j \div \ \bvarphi_i \d \mathcal{F}
+ 2\nu \gamma \int_{\Gamma} \chi_j(D(\bvarphi_i) \bn \cdot \bn)  \d \Gamma, \\
& & \left(  A_{\bu\blambda} \right)_{ij}  =  - \int_{\Gamma} \bvarphi_i \cdot \bpsi_j \d \Gamma
+ 2\nu \gamma \int_{\Gamma} (D(\bvarphi_i) \bn )\cdot \bpsi_j  \d \Gamma  , \\
& & \left(  A_{p p} \right)_{ij}  =  -\gamma \int_{\Gamma} \chi_i \chi_j \d \Gamma, \\
& & \left(  A_{p \blambda} \right)_{ij}  =  -\gamma \int_{\Gamma} \chi_i ( \bpsi_j \cdot \bn ) \d \Gamma, \\
& & \left(  A_{\blambda \blambda} \right)_{ij}  =  -\gamma \int_{\Gamma} \bpsi_i \bpsi_j \d \Gamma, \\
& & \left(\boldsymbol{F} \right)_i = \int_{\mathcal{F}} \bff \cdot \bvarphi_i \d \mathcal{F},
\quad \left(\boldsymbol{G} \right)_i = - \int_{\Gamma} \bg \cdot \bpsi_i \d \Gamma.
\end{eqnarray*}

\subsection{A theoretical analysis of the stabilized method} \label{seccvstab}
Let us take $ \gamma = \gamma_0 h$ with some constant $\gamma_0 >0$. We first observe that the discrete problem can be rewritten in the following compact form:
\begin{eqnarray*}
& & \text{Find $(\bu^h,p^h,\blambda^h) \in \mathbf{V}^h \times Q^h \times \mathbf{W}^h$ such that} \\
& & \mathcal{M}((\bu^h,p^h,\blambda^h);(\bv^h,q^h,\bmu^h))  =
{\cal H}(\bv^h,q^h,\bmu^h), \quad \forall (\bv^h,q^h,\bmu^h) \in \mathbf{V}^h \times Q^h \times \mathbf{W}^h,
\end{eqnarray*}
where
\begin{eqnarray*}
\mathcal{M}((\bu,p,\blambda);(\bv,q,\bmu)) &=&2\nu \int_{\mathcal{F}}D(\bu):D(\bv)\d \mathcal{F}-\int_{\mathcal{F}}(p\div \ \bv+q\div \ \bu) \d \mathcal{F}-\int_{\Gamma }(\blambda\cdot \bv+\bmu\cdot \bu)\d\Gamma  \\
&&-\gamma_0 h \int_{\Gamma }(2\nu D(\bu)\bn-p\bn-\blambda)\cdot \left( 2\nu D(\bv%
)\bn-q\bn-\bmu\right) \d\Gamma,
\end{eqnarray*}
and
\begin{eqnarray*}
\mathcal{H}(\bv,q,\bmu) = \int_{\mathcal{F}} \bff\cdot \bv \d \Gamma  -\int_{\Gamma} \bmu \cdot \bg \d \Gamma.
\end{eqnarray*}
In the following, we will need some assumptions for our theoretical analysis:
\begin{description}
\item[A1]
For all $\bv^{h}\in \mathbf{V}^{h}$ one has%
\begin{eqnarray*}
h\|D(\bv^{h})n\|_{\mathbf{L}^2(\Gamma)}^{2} & \leq & C\|\bv^{h}\|_{\mathbf{V}}^{2}.
\end{eqnarray*}
\item[A2]
For all $q^{h}\in Q^{h}$ one has%
\begin{eqnarray*}
h\|q^{h}\|_{\L^2(\Gamma)}^{2} & \leq & C\|q^{h}\|_{\L^2(\mathcal{F})}^{2}.
\end{eqnarray*}
\item[A3]
One has the following inf-sup condition for the velocity-pressure pair of finite
element spaces
\begin{eqnarray*}
\inf_{q^{h}\in Q^{h}}\sup_{\bv^{h}\in \mathbf{V}_{0}^{h}}\frac{b(\bv^h,q^h)}{\|q^{h}\|_{\L^2(\mathcal{F})}\|\bv^{h}\|_{\mathbf{V}}} & \geq & \beta,
\end{eqnarray*}
with $\beta >0$ independent of $h$.
\end{description}
\hfill \\
\textcolor{black}{
Assumptions {\bf A1} and {\bf A2} will be discussed in section \ref{secnumstab} by performing some numerical tests.
}
\hfill \\
Note that assumption {\bf A1} is the same as those introduced in \cite{HaslR} (cf. equations (5.1) and (5.5) respectively) in the study of the fictitious domain approach for the Laplace equation stabilized \`a la Barbosa-Hughes. Our assumption {\bf A2} is also similar in nature to those two, and all these three assumptions can be in fact established if one assumes that the intersections of ${\cal F}$ with the triangles of the mesh are not "too
small" (see Appendix B of \cite{HaslR} and section \ref{seccomments}). Although all these assumptions can be violated in practice if a mesh triangle is cut by the boundary $\Gamma$ so that only its tiny portion happens to be inside of ${\cal F}$. \textcolor{black}{The numerical experiments for the Laplace equation in \cite{HaslR} show that such accidents occur rather rarely and their impact on the overall behavior of the method is practically negligible}. This conclusion can be safely transposed to the case of Stokes problem. However, we have now the additional difficulty in the form
of the inf-sup condition {\bf A3}. Of course this condition is verified if one chooses the classical stable pair of finite element spaces, like for instance the Taylor-Hood elements P2/P1 pair for velocity/pressure, and if the boundary $\Gamma$ does not cut the edges of the triangles of the mesh. However, in the general case of an arbitrary geometry, we have by now no evidence of the fulfillment of the inf-sup condition {\bf A3}.\\

\textcolor{black}{We also need the following result for the $L^{2}$-orthogonal projector from ${\bf H}^{1/2}(\Gamma )$ to
$W^{h}$:
\begin{lemma} \label{lemmaA3}
For all $\bv\in \mathbf{H}^{1/2}(\Gamma)$ one has
\begin{eqnarray*}
\|P^{h}\bv-\bv\|_{\mathbf{L}^2(\Gamma)} & \leq & Ch^{1/2}\|\bv\|_{\mathbf{H}^{1/2}(\Gamma)},
\end{eqnarray*}
where $P^{h}$ denotes the $L^{2}$-orthogonal projector from $\mathbf{H}^{1/2}(\Gamma )$ to $\mathbf{W}^{h}$.
\end{lemma}
}
\textcolor{black}{
\begin{proof}
This result is well-known, but we provide for completeness a sketch of the proof in the case when discontinuous finite elements are chosen for the space $\tilde{\bf W}_h$, so that ${\bf W}_h$ contains piecewise constant functions on the mesh ${\cal T}^h_\Gamma$ on $\Gamma$ induced by the mesh ${\cal T}^h$ on $\mathcal{O}$ (the elements of ${\cal T}^h_\Gamma$ are the arcs of $\Gamma$ obtained by intersecting $\Gamma$ with the triangles from ${\cal T}^h$). The proof in the case of continuous finite elements is similar but slightly more technical.\\
Let $I^{h}$ be the interpolation operator to the space of piecewise constant functions on ${\cal T}^h_\Gamma$. For all sufficiently smooth function $ \bv $ on $\Gamma$ and for all element $\tau_T$ of the curve $\Gamma$ obtained by intersection with a triangle $T\in{\cal T}^h$, we set
\begin{eqnarray*}
I^{h} \bv |{\tau_T}= \bv (\x_T),
\end{eqnarray*}
where $\x_T$ is the middle point of $\tau_T$. We have then $I^h  \bv \in {\bf W}_h$ and
\begin{eqnarray*}
\|P^{h}\bv -\bv \|_{\mathbf{L}^{2}(\Gamma)} \leq
\|I^{h}\bv -\bv\|_{\mathbf{L}^{2}(\Gamma)} \leq C h \|\bv\|_{\mathbf{H}^{1}(\Gamma)}, \quad \forall \bv \in \mathbf{H}^{1}(\Gamma),
\end{eqnarray*}
by the standard interpolation estimates. Moreover,
\begin{eqnarray*}
\|P^{h}\bv -\bv \|_{\mathbf{L}^{2}(\Gamma)} \leq \|\bv\|_{\mathbf{L}^{2}(\Gamma)}, \quad \forall \bv\in \mathbf{L}^{2}(\Gamma).
\end{eqnarray*}
Interpolating between the last two estimates (see the last chapter of \cite{bren}) we get the desired result.
\end{proof}
}

We prove in this subsection the following inf-sup result, which is an adaptation of Lemma 3 from \cite{HaslR}.\\
\begin{lemma} \label{lemmainfsup}
Under assumptions {\bf A1}--{\bf A3}, there exists for $\gamma_0$ small enough a mesh-independent constant $c>0$ such that
\begin{eqnarray*}
\inf_{(\bu^{h},p^{h},\blambda^{h})\in \mathbf{V}^{h}\times Q^{h}\times \mathbf{W}^{h}} \sup_{(\bv^{h},q^{h},\bmu^{h})\in \mathbf{V}^{h}\times Q^{h}\times \mathbf{W}^{h}}\frac{%
\mathcal{M}((\bu^{h},p^{h},\blambda^{h});(\bv^{h},q^{h},\bmu ^{h}))}{%
|||\bu^{h},p^{h},\blambda^{h}|||\,|||\bv^{h},q^{h},\bmu ^{h}|||}\geq c,
\end{eqnarray*}%
where the triple norm is defined by%
\begin{eqnarray*}
|||\bu,p,\blambda |||=\left( \|\bu\|_{\mathbf{V}}^{2}+\|p\|_{\L^2(\mathcal{F})}^{2}+h\|D(\bu)\bn%
\|_{\mathbf{L}^2(\Gamma) }^{2}+h\|p\|_{\L^2(\Gamma) }^{2}+h\|\blambda \|_{\mathbf{L}^2(\Gamma) }^{2}+%
\frac{1}{h}\|\bu\|_{\mathbf{L}^2(\Gamma) }^{2}\right) ^{1/2},
\end{eqnarray*}%
and $c$ is a mesh-independent constant.
\end{lemma}

\begin{proof}
We observe that%
\begin{eqnarray*}
\mathcal{M}((\bu^{h},p^{h},\blambda^{h});(\bu^{h},-p^{h},-\blambda ^{h}))
&=&2\nu \|\bu^{h}\|_{\mathbf{V}}^{2}-\gamma_{0}h\int_{\Gamma }4\nu ^{2}|D(\bu^{h})\bn |^{2}\d \Gamma +\gamma_{0}h\int_{\Gamma}|p^{h}\bn-\blambda^{h}|^{2}\d\Gamma \\
&\geq &\nu \|\bu^{h}\|_{\mathbf{V}}^{2}+\gamma _{0}h\|p^{h}\bn+\blambda ^{h}\|_{\mathbf{L}^2(\Gamma)}^{2}
\end{eqnarray*}%
where we have used assumption {\bf A1} and the fact that $\gamma_{0}$ can be taken
sufficiently small. More precisely, we can choose $\gamma_0$ such that $4\nu ^{2}\gamma _{0}C\leq \nu $, where $C$ is
the constant of assumption {\bf A1}. The inf-sup condition {\bf A3} implies that for all $p^{h}\in Q^{h}$ there exists $\bv_{p}^{h}\in \mathbf{V}_{0}^{h}$ such that%
\begin{eqnarray}
-\int_{\mathcal{F}}p^{h}\div \ \bv_{p}^{h}\d \mathcal{F}=\|p^{h}\|_{\L^2(\mathcal{F})}^{2}%
& \quad \text{and }&\|\bv_{p}^{h}\|_{\mathbf{V}}\leq C\|p^{h}\|_{\L^2(\mathcal{F})}.  \label{vhp}
\end{eqnarray}%
Now let us observe that%
\begin{eqnarray*}
\mathcal{M}((\bu^{h},p^{h},\blambda^{h});(\bv_{p}^{h},0,0)) &=&2\nu \int_{%
\mathcal{F}}D(\bu^{h}):D(\bv_{p}^{h})\d \mathcal{F}+\|p^{h}\|_{\L^2(\mathcal{F})}^{2}\\
& & -2\nu \gamma _{0}h\int_{\Gamma }(2\nu D(\bu^{h})\bn-p^{h}\bn-\blambda^{h})\cdot D(%
\bv_{p}^{h})\bn \d\Gamma  \\
&\geq &\|p^{h}\|_{\L^2(\mathcal{F})}^{2}-\nu \alpha \|\bu^{h}\|_{\mathbf{V}}^{2}-\frac{\nu }{\alpha
}\|\bv_{p}^{h}\|_{\mathbf{V}}^{2} \\
& & -\nu \gamma _{0}h\alpha \|2\nu D(\bu^{h})\bn-p^{h}\bn%
-\blambda^{h}\|_{\mathbf{L}^2(\Gamma)}^{2}-\frac{\nu \gamma _{0}h}{\alpha }\|D(\bv%
_{p}^{h})\bn\|_{\mathbf{L}^2(\Gamma)}^{2}.
\end{eqnarray*}%
We have used here the Young inequality which is valid for any $\alpha >0$.
In particular, we can choose $\alpha $ large enough so that we can conclude
with the aid of assumptions {\bf A1} and {\bf A2} (the constant $C$ here will be
independent of $\alpha $ and $h$, but dependent on $\gamma _{0}$ and on the
constants in the inequalities {\bf A1} and {\bf A2}). We get
\begin{eqnarray*}
\mathcal{M}((\bu^{h},p^{h},\blambda^{h});(\bv_{p}^{h},0,0)) &\geq
&\|p^{h}\|_{\L^2(\mathcal{F})}^{2}-\nu \alpha \|\bu^{h}\|_{\mathbf{V}}^{2}-\frac{C}{\alpha }%
\|p^{h}\|_{\L^2(\mathcal{F})}^{2}-C\alpha h\|D(\bu^{h})\bn\|_{\mathbf{L}^2(\Gamma)}^{2}\\
& & -\nu \gamma_{0}h\alpha \|p^{h}\bn+\blambda^{h}\|_{\mathbf{L}^2(\Gamma)}^{2}-\frac{C}{\alpha }%
\|p^{h}\|_{\L^2(\Gamma)}^{2} \\
&\geq &\frac{1}{2}\|p^{h}\|_{\L^2(\mathcal{F})}^{2}-C\alpha \|\bu^{h}\|_{\mathbf{V}}^{2}-\nu \gamma
_{0}h\alpha \|p^{h}\bn+\blambda^{h}\|_{\mathbf{L}^2(\Gamma)}^{2}.
\end{eqnarray*}%
Let us now take $\bar{\bmu}_{h}=-\frac{1}{h}P^{h}\bu^{h}$ where $P^{h}$ is the
projector from $\mathbf{H}^{1/2}(\Gamma)$ to $\mathbf{W}^{h}$. Observe that, in using assumption {\bf A1}, we have%
\begin{eqnarray*}
\mathcal{M}((\bu^{h},p^{h},\blambda^{h});(0,0,\bar{\bmu}^{h})) &=&\frac{1}{h}%
\|P^{h}\bu^{h}\|_{\mathbf{L}^2(\Gamma)}^{2}-\gamma _{0}\int_{\Gamma }(2\nu D(\bu^{h})\bn%
-p^{h}\bn-\blambda^{h})\cdot P^{h}\bu^{h}\d \Gamma  \\
&\geq &\frac{1}{h}\|P^{h}\bu^{h}\|_{\mathbf{L}^2(\Gamma)}^{2} \\
& & -\gamma _{0}\left( \sqrt{h}%
\|D(\bu^{h})\bn\|_{\mathbf{L}^2(\Gamma)}+\sqrt{h}\|p^{h}\bn+\blambda^{h}\|_{\mathbf{L}^2(\Gamma)
}\right) \frac{1}{\sqrt{h}}\|P^{h}\bu^{h}\|_{\mathbf{L}^2(\Gamma)} \\
&\geq &\frac{1}{2h}\|P^{h}\bu^{h}\|_{\mathbf{L}^2(\Gamma)}^{2}-C\|\bu%
^{h}\|_{\mathbf{V}}^{2}-Ch\|p^{h}\bn+\blambda^{h}\|_{\mathbf{L}^2(\Gamma)}^{2}.
\end{eqnarray*}%
Combining the above inequalities and taking some small enough numbers $\kappa
>0$ and $\eta >0$, we can obtain%
\begin{eqnarray*}
& &
\mathcal{M}((\bu^{h},p^{h},\blambda^{h});(\bu^{h}+\kappa \bv_{p}^{h},-p^{h},-\blambda ^{h}+\eta \bar{\bmu}^{h})) \\
& & \geq \nu \|\bu^{h}\|_{\mathbf{V}}^{2}+\gamma _{0}h\|p^{h}\bn+\blambda^{h}\|_{\mathbf{L}^2(\Gamma) }^{2}
+\frac{\kappa }{2}\|p^{h}\|_{\L^2(\mathcal{F})}^{2}-C\alpha \kappa \|\bu^{h}\|_{\mathbf{V}}^{2}-\nu \gamma_{0}h\alpha \kappa \|p^{h}\bn+\blambda^{h}\|_{\mathbf{L}^2(\Gamma) }^{2}\\
& &  \quad +\frac{\eta }{2h}\|P^{h}\bu^{h}\|_{\mathbf{L}^2(\Gamma)}^{2} - C\eta \|\bu^{h}\|_{\mathbf{V}}^{2} - C\eta h\|p^{h}\bn+\blambda^{h}\|_{\mathbf{L}^2(\Gamma)}^{2} \\
& & \geq \frac{\nu }{2}\|\bu^{h}\|_{\mathbf{V}}^{2}+\frac{\kappa }{2} \|p^{h}\|_{\L^2(\mathcal{F})}^{2}
+\frac{\gamma _{0}}{2}h\|p^{h}\bn+\blambda^{h}\|_{\mathbf{L}^2(\Gamma)}^{2}+\frac{\eta }{2h}\|P^{h}\bu^{h}\|_{\mathbf{L}^2(\Gamma)}^{2} \\
& & \geq \frac{\nu }{4}\|\bu^{h}\|_{\mathbf{V}}^{2} +\frac{\eta }{2h}\|P^{h}\bu^{h}\|_{\mathbf{L}^2(\Gamma)}^{2}
+\frac{\nu }{4C}h\|D(\bu^{h})n\|_{\mathbf{L}^2(\Gamma)}^{2} \\
& & +\frac{\kappa }{4}\|p^{h}\|_{\L^2(\mathcal{F})}^{2}+\frac{\kappa }{%
4C}h\|p^{h}\|_{\L^2(\Gamma)}^{2}+\frac{\gamma _{0}}{2}h\|p^{h}\bn+\blambda^{h}\|_{\mathbf{L}^2(\Gamma)}^{2}.
\end{eqnarray*}%
In the last line, we have used again assumptions {\bf A1} and {\bf A2} (with the corresponding
constant $C$). We now rework the last two terms in order to split $p^h$ and $\blambda^h$. Denoting $\displaystyle t=\frac{\kappa }{2C\gamma _{0}}$, we have%
\begin{eqnarray*}
\frac{\kappa }{4C}h\|p^{h}\|_{\L^2(\Gamma)}^{2}+\frac{\gamma _{0}}{2}h\|p^{h}\bn%
+\blambda^{h}\|_{\mathbf{L}^2(\Gamma)}^{2} &=&\frac{\gamma _{0}}{2}h\left(
(t+1)\|p^{h}\|_{\L^2(\Gamma)}^{2}+\|\blambda^{h}\|_{\mathbf{L}^2(\Gamma)}^{2}+2\int_{\Gamma }p^{h}\bn\cdot \blambda^{h}\d \Gamma \right)  \\
&\geq &\frac{\gamma _{0}}{2}h\left( (t+1)\|p^{h}\|_{\L^2(\Gamma)}^{2}
+\|\blambda^{h}\|_{\mathbf{L}^2(\Gamma)}^{2}-(t/2+1)\|p^{h}\|_{\L^2(\Gamma)}^{2} \right.\\
& & \left.-\frac{1}{t/2+1}\|\blambda^{h}\|_{\mathbf{L}^2(\Gamma)}^{2}\right)  \\
&=&\frac{\gamma _{0}}{2}h\left( \frac{t}{2}\|p^{h}\|_{\L^2(\Gamma)}^{2}+\frac{t/2}{t/2+1}\|\blambda^{h}\|_{\mathbf{L}^2(\Gamma)}^{2}\right).
\end{eqnarray*}%
So we finally have%
\begin{align*}
&
\mathcal{M}((\bu^{h},p^{h},\blambda^{h});(\bu^{h}+\kappa
\bv_{p}^{h},-p^{h},-\blambda ^{h}+\eta \bmu ^{h}))\\
&\quad
\geq c\left( \|\bu^{h}\|_{\mathbf{V}}^{2}+\|p^{h}\|_{\L^2(\mathcal{F})}^{2}+h\|D(\bu^{h})n\|_{\mathbf{L}^2(\Gamma)
}^{2}+h\|p^{h}\|_{\L^2(\Gamma)}^{2}+h\|\blambda^{h}\|_{\mathbf{L}^2(\Gamma)}^{2}+\frac{1}{h%
}\|P^{h}\bu^{h}\|_{\mathbf{L}^2(\Gamma)}^{2}\right).
\end{align*}%
We can now eliminate the projector $P^{h}$ in this estimate by the
following calculation, which is valid for some $\beta >0$ small enough%
\begin{eqnarray*}
\|\bu^{h}\|_{\mathbf{V}}^{2}+\frac{1}{h}\|P^{h}\bu^{h}\|_{\mathbf{L}^2(\Gamma)}^{2} &\geq &\|\bu%
^{h}\|_{\mathbf{V}}^{2}+\frac{\beta }{h}\|P^{h}\bu^{h}\|_{\mathbf{L}^2(\Gamma)}^{2}=\|\bu%
^{h}\|_{\mathbf{V}}^{2}+\frac{\beta }{h}\left( \|\bu^{h}\|_{\mathbf{L}^2(\Gamma)}^{2}-\|\bu^{h}-P^{h}%
\bu^{h}\|_{\mathbf{L}^2(\Gamma)}^{2}\right)  \\
&\geq &\|\bu^{h}\|_{\mathbf{V}}^{2}+\frac{\beta }{h}\|\bu^{h}\|_{\mathbf{L}^2(\Gamma)}^{2}-C\beta
\|\bu^{h}\|_{\mathbf{H}^{1/2}(\Gamma)}^{2} \\
& \geq & \|\bu^{h}\|_{\mathbf{V}}^{2}+\frac{\beta }{h}%
\|\bu^{h}\|_{\mathbf{L}^2(\Gamma)}^{2}-C\beta \|\bu^{h}\|_{\mathbf{V}}^{2} \\
&\geq &\frac{1}{2}\|\bu^{h}\|_{\mathbf{V}}^{2}+\frac{\beta }{h}\|\bu^{h}\|_{\mathbf{L}^2(\Gamma)}^{2}.
\end{eqnarray*}%
We have used \textcolor{black}{the result of Lemma \ref{lemmaA3}} and the trace inequality.

In summary, we have obtained that taking%
\begin{eqnarray*}
(\bv^{h},q^{h},\bmu ^{h}) & = & (\bu^{h}+\kappa \bv_{p}^{h},-p^{h},-\blambda ^{h}+\eta
\bar{\bmu}^{h})
\end{eqnarray*}%
one has%
\begin{eqnarray}
\mathcal{M}((\bu^{h},p^{h},\blambda^{h});(\bv^{h},q^{h},\bmu ^{h})) & \geq & c|||\bu%
^{h},p^{h},\blambda^{h}|||^{2}.  \label{al1}
\end{eqnarray}%
On the other hand,%
\begin{eqnarray}
|||\bv^{h},q^{h},\bmu ^{h}||| & \leq & M|||\bu^{h},p^{h},\blambda^{h}|||  \label{al2}
\end{eqnarray}%
with some $M>0$ independent of $h$. Indeed, we have
\begin{eqnarray*}
|||\bv^{h},q^{h},\bmu ^{h}||| &\leq &|||\bu^{h},p^{h},\blambda ^{h}|||+\kappa
|||\bv_{p}^{h},0,0|||+\eta |||0,0,\bar{\bmu}^{h}||| \\
&\leq &|||\bu^{h},p^{h},\blambda ^{h}|||+\kappa \left(
\|\bv_{p}^{h}\|_{\mathbf{V}}^{2}+h\|D(\bv_{p}^{h})\bn\|_{\mathbf{L}^2(\Gamma)}^{2}+\frac{1}{h}%
\|\bv_{p}^{h}\|_{\mathbf{L}^2(\Gamma)}^{2}\right)^{1/2}+\eta \sqrt{h}\|\bar{\bmu}%
^{h}\|_{\mathbf{L}^2(\Gamma)}.
\end{eqnarray*}%
Now, by assumption {\bf A1} and the fact that $\bv_{p}^{h}\in \mathbf{V}_{0}^{h}$ so that $%
P^{h}\bv_{p}^{h}=0$, we have
\begin{eqnarray*}
\|\bv_{p}^{h}\|_{\mathbf{V}}^{2}+h\|D(\bv_{p}^{h})\bn\|_{\mathbf{L}^2(\Gamma)}^{2}+\frac{1}{h}%
\|\bv_{p}^{h}\|_{\mathbf{L}^2(\Gamma)}^{2} & \leq & C\|\bv_{p}^{h}\|_{\mathbf{V}}^{2}+\frac{1}{h}%
\|\bv_{p}^{h}-P^{h}\bv_{p}^{h}\|_{\mathbf{L}^2(\Gamma)}^{2}.
\end{eqnarray*}%
Furthermore, \textcolor{black}{by Lemma \ref{lemmaA3}} and by the definition of $%
\bv_{p}^{h} \in \mathbf{V}_0^h$ given in \eqref{vhp}, we have
\begin{eqnarray*}
\|\bv_{p}^{h}\|_{\mathbf{V}}^{2}+h\|D(\bv_{p}^{h})\bn\|_{\mathbf{L}^2(\Gamma)}^{2}+\frac{1}{h}%
\|\bv_{p}^{h}\|_{\mathbf{L}^2(\Gamma)}^{2} & \leq & C\|\bv_{p}^{h}\|_{\mathbf{V}}^{2}+C\|\bv_{p}^{h}\|_{\mathbf{H}^{1/2}(\Gamma)}^{2} \\
& \leq & C\|\bv_{p}^{h}\|_{\mathbf{V}}^{2}\leq C\|p^{h}\|_{\L^2(\mathcal{F})}^{2}\leq C|||\bu^{h},p^{h},\blambda^{h}|||.
\end{eqnarray*}%
We have also
\begin{equation*}
\sqrt{h}\|\bar{\bmu}^{h}\|_{\mathbf{L}^2(\Gamma)}=\frac{1}{\sqrt{h}}\|P^{h}\bu^{h}\|_{\mathbf{L}^2(\Gamma)}
\leq \frac{1}{\sqrt{h}}\|\bu^{h}\|_{\mathbf{L}^2(\Gamma)}\leq
|||\bu^{h},p^{h},\blambda ^{h}|||,
\end{equation*}%
hence the inequality \eqref{al2}. Dividing \eqref{al1} by \eqref{al2} yields%
\begin{eqnarray*}
\frac{\mathcal{M}((\bu^{h},p^{h},\blambda^{h});(\bv^{h},q^{h},\bmu ^{h}))}{%
|||\bv^{h},q^{h},\bmu ^{h}|||} & \geq & \frac{c}{M}|||\bu^{h},p^{h},\blambda ^{h}|||,
\end{eqnarray*}%
which is the desired result.
\end{proof}

The lemma above, combined with the fact that the bilinear form ${\cal M}$ is bounded in the triple norm on ${\bf V}\times Q \times \mathbf{W}$ uniformly with respect to $h$, leads us by a C\'ea type lemma (cf. \cite{Ern} or Theorem 5.2 in \cite{HaslR}) to the following abstract error estimate
$$
|||\bu-\bu^{h},p-p^{h},\blambda-\blambda ^{h}|||
\le C
\inf_{(\bv^{h},q^{h},\bmu ^{h})\in \mathbf{V}^{h}\times Q^{h}\times \mathbf{W}^{h}}
|||\bu-\bv^{h},p-q^{h},\blambda-\bmu ^{h}|||.
$$
Using the extension theorem for the Sobolev spaces, the standard estimates for the nodal (or Cl\'ement if necessary) finite element interpolation operators, and the trace inequality $\|w\|_{\L^2(\Gamma)} \le C \left( h^{-1 } \|w\|_{\L^2(T)} +h \|w\|_{\L^2(T)}\right)$ for any $w\in \H^1(T)$ on any triangle $T\in \mathcal{T}_h$ (which is valid provided $\Gamma$ is sufficiently smooth - see Appendix A of \cite{HaslR} for a proof),
we obtain the following error estimate
\begin{eqnarray*}\label{eststab}
&&\max ( \|\bu-\bu^{h}\|_{\mathbf{V}} , \|p-p^h\|_{\L^2(\mathcal{F})}, h\|\blambda -\blambda^h\|_{\mathbf{L}^2(\Gamma)} )
 \le
|||\bu-\bu^{h},p-p^{h},\blambda-\blambda ^{h}|||
\\
&&\qquad \le C(
h^{k_u}\|\bu\|_{\mathbf{H}^{k_u+1}(\mathcal{F})} + h^{k_p+1}\|p\|_{\H^{k_p+1}(\mathcal{F})} +
h^{k_\lambda+1}\|\blambda\|_{\mathbf{H}^{k_{\lambda}+1/2}(\Gamma)}
),
\notag
\end{eqnarray*}
where $k_u$, $k_p$ and $k_\lambda$ are the degrees of finite elements used for velocity, pressure and multiplier $\blambda$ respectively.
The proof of this result is rather tedious but can be easily reproduced following the ideas of \cite{HaslR} (see, in particular, the proofs of Theorem 5.3 and Lemma 5.4 there).


\section{Numerical experiments} \label{secnum}

For numerical experiments, we consider the square $[0 , 1]\times [0 , 1]$ and choose as $\Gamma$ the circle whose level-set representation is
\begin{eqnarray*}
(x-0.5)^2 + (y-0.5)^2 = R^2,
\end{eqnarray*}
with $R=0.21$ (see figure \ref{meshimage2}).
The exact solutions are chosen equal to
\begin{eqnarray*}
\bu_{ex}(x,y) & = & \left( \begin{array} {ll}
\cos(\pi x)\sin(\pi y) \\
-\sin(\pi x)\cos(\pi y)
\end{array} \right), \\
p_{ex}(x,y) & = & (y - 1/2) \cos(2\pi x) + (x-1/2)\sin(2\pi y).
\end{eqnarray*}

The meshes and all the computations have been obtained with the C++ finite element library  {\sc Getfem++} \cite{Getfem}. In the numerical tests, we compare the discrete solutions with the exact solutions for different meshes (six imbricated uniform meshes).\\
We denote $\boldsymbol{U}_{ex}$, $\boldsymbol{P}_{ex}$ and $\boldsymbol{\Lambda}_{ex}$ the discrete forms of functions $\bu_{ex}$, $p_{ex}$ and $\blambda_{ex} = \sigma(\bu_{ex}, p_{ex})\bn$ respectively. For practical purposes, the error introduced by the approximation of the exact vector $\boldsymbol{\Lambda}_{ex}$ by $\boldsymbol{\Lambda}$ is given by the square root of
\begin{eqnarray*}
\| \boldsymbol{\Lambda}_{ex} - \boldsymbol{\Lambda} \|_{\mathbf{L}^2(\Gamma)}^2 & = & \int_{\Gamma} \left|\sigma(\boldsymbol{U}_{ex},\boldsymbol{P}_{ex})\bn - \boldsymbol{\Lambda} \right|^2\d \Gamma  .
\end{eqnarray*}
This scalar product is developed and using the assembling matrices we compute
\begin{eqnarray*}
\| \boldsymbol{\Lambda}_{ex} - \boldsymbol{\Lambda} \|_{\mathbf{L}^2(\Gamma)}^2 & = & \langle A_{\bu\bu} \boldsymbol{U}_{ex}, \boldsymbol{U}_{ex}\rangle + 2\langle A_{\bu p} \boldsymbol{P}_{ex}, \boldsymbol{U}_{ex}\rangle + 2\langle A_{\bu \blambda}\boldsymbol{\Lambda} , \boldsymbol{U}_{ex}\rangle + \\
& & \langle A_{pp} \boldsymbol{P}_{ex}, \boldsymbol{P}_{ex} \rangle - 2\langle A_{p \blambda} \boldsymbol{\Lambda} , \boldsymbol{P}_{ex} \rangle + \langle A_{\blambda \blambda} \boldsymbol{\Lambda}, \boldsymbol{\Lambda} \rangle,
\end{eqnarray*}
where $\langle \cdot , \cdot \rangle$ denotes the classical Euclidean scalar product in finite dimension. Then, the relative error is given by
\begin{eqnarray*}
\frac{\| \boldsymbol{\Lambda}_{ex} - \boldsymbol{\Lambda} \|_{\mathbf{L}^2(\Gamma)}}{\| \boldsymbol{\Lambda}_{ex} \|_{\mathbf{L}^2(\Gamma)}} & = &
\frac{\| \boldsymbol{\Lambda}_{ex} - \boldsymbol{\Lambda} \|_{\mathbf{L}^2(\Gamma)}}{
\left( \langle A_{\bu\bu} \boldsymbol{U}_{ex}, \boldsymbol{U}_{ex}\rangle + \langle A_{pp} \boldsymbol{P}_{ex}, \boldsymbol{P}_{ex}\rangle+ 2\langle A_{\bu p}\boldsymbol{P}_{ex}, \boldsymbol{U}_{ex}\rangle \right)^{1/2}}.
\end{eqnarray*}

\subsection{Numerical experiments for the method without stabilization} \label{secnumstab0}
We present numerical computations of errors when no stabilization are imposed.
We consider several choices of the finite element spaces $\tilde{\mathbf{V}}^h$, $\tilde{Q}^h$ and $\tilde{\mathbf{W}}^h$.
Four couples of spaces are studied (for ${\bf u}$/$p$/${\blambda}$), P1+/P1/P0 (a standard continuous P1 element for ${\bf u}$ enriched by a cubic bubble function, standard continuous P1 for the pressure $p$ and discontinuous P0 for the multiplier ${\blambda}$ element on a triangle), P2/P1/P0, for triangular meshes and Q1/Q0/Q0, Q2/Q1/Q0 for quadrangular meshes.
The elements chosen between velocity and pressure are the ones which ensure the discrete mesh-independent inf-sup condition $\mathbf{H1}$ in the case of uncut functions (except for the Q1/Q0 pair), that is to say the classical case where regular meshes are considered.
Low degrees are selected to control the memory (CPU time) which plays a crucial role in numerical simulations for fluid-structure
interactions, specially in an unsteady framework. For the multiplier introduced for the interface, since the stabilization is not used,
a discrete mesh-independent inf-sup condition must be satisfied. For instance, the couple of spaces Q1/Q0/Q0 does not satisfy this condition.
The error curves between the discrete solution and the exact one are given in figure \ref{withoutStabCv}
for different norms. The rates of convergence are reported.\\
\begin{minipage}{21cm}
\hspace*{-2cm} \includegraphics[trim = 0cm 0cm 0cm 5cm, clip, scale=0.4]{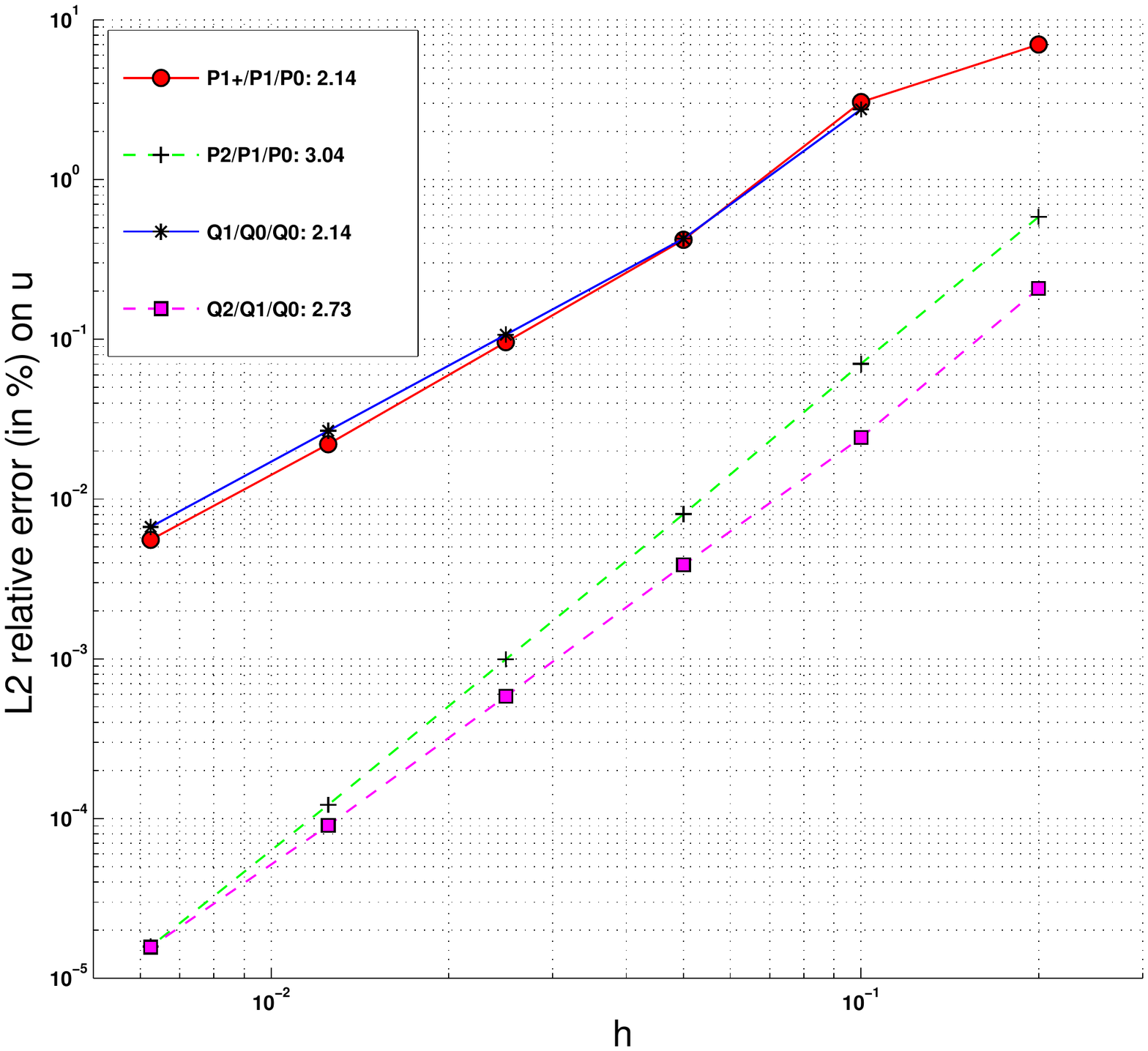} 
\includegraphics[trim = 0cm 0cm 0cm 5cm, clip, scale = 0.4]{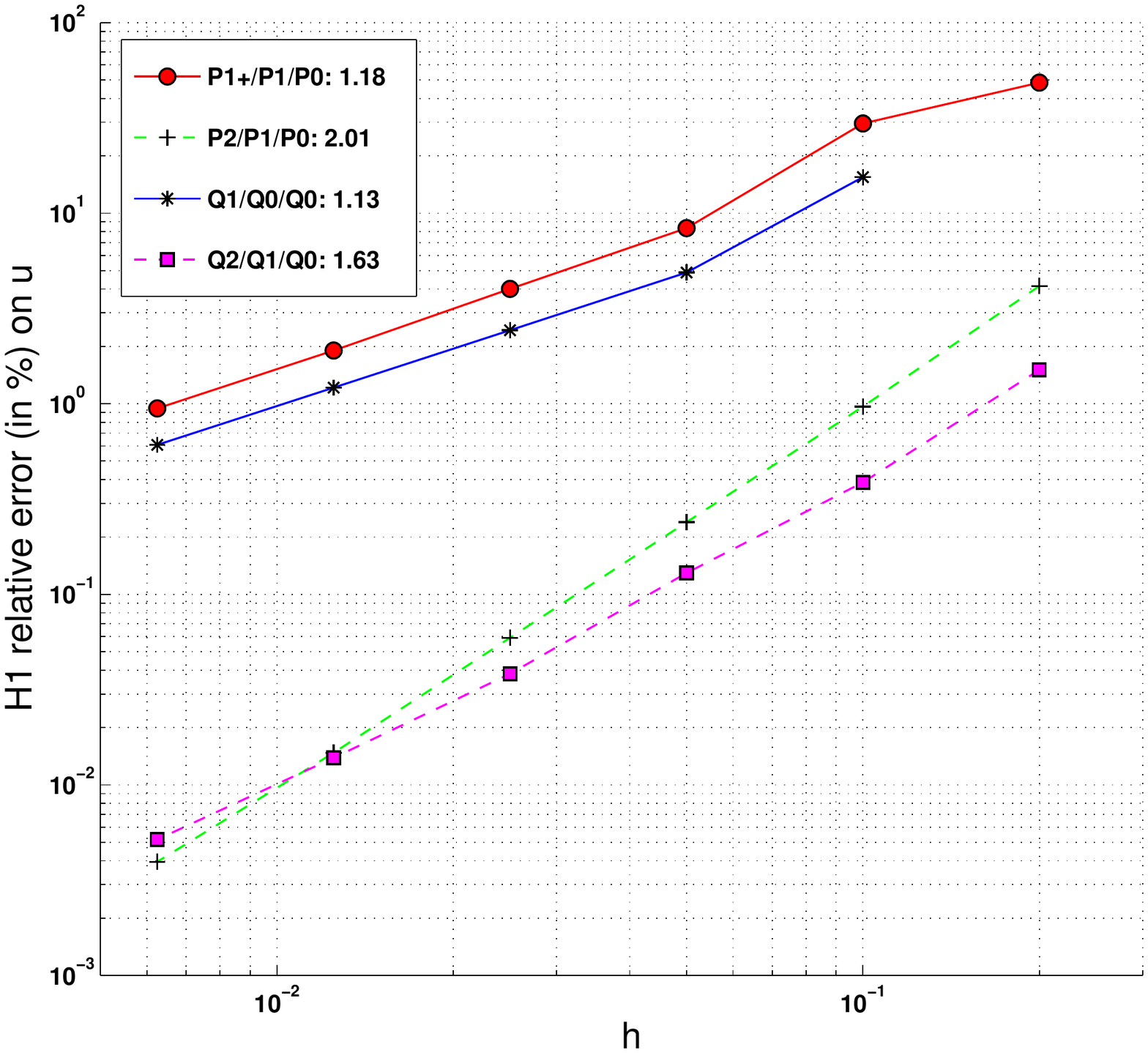}\\
\hspace*{-2cm}\includegraphics[scale = 0.4]{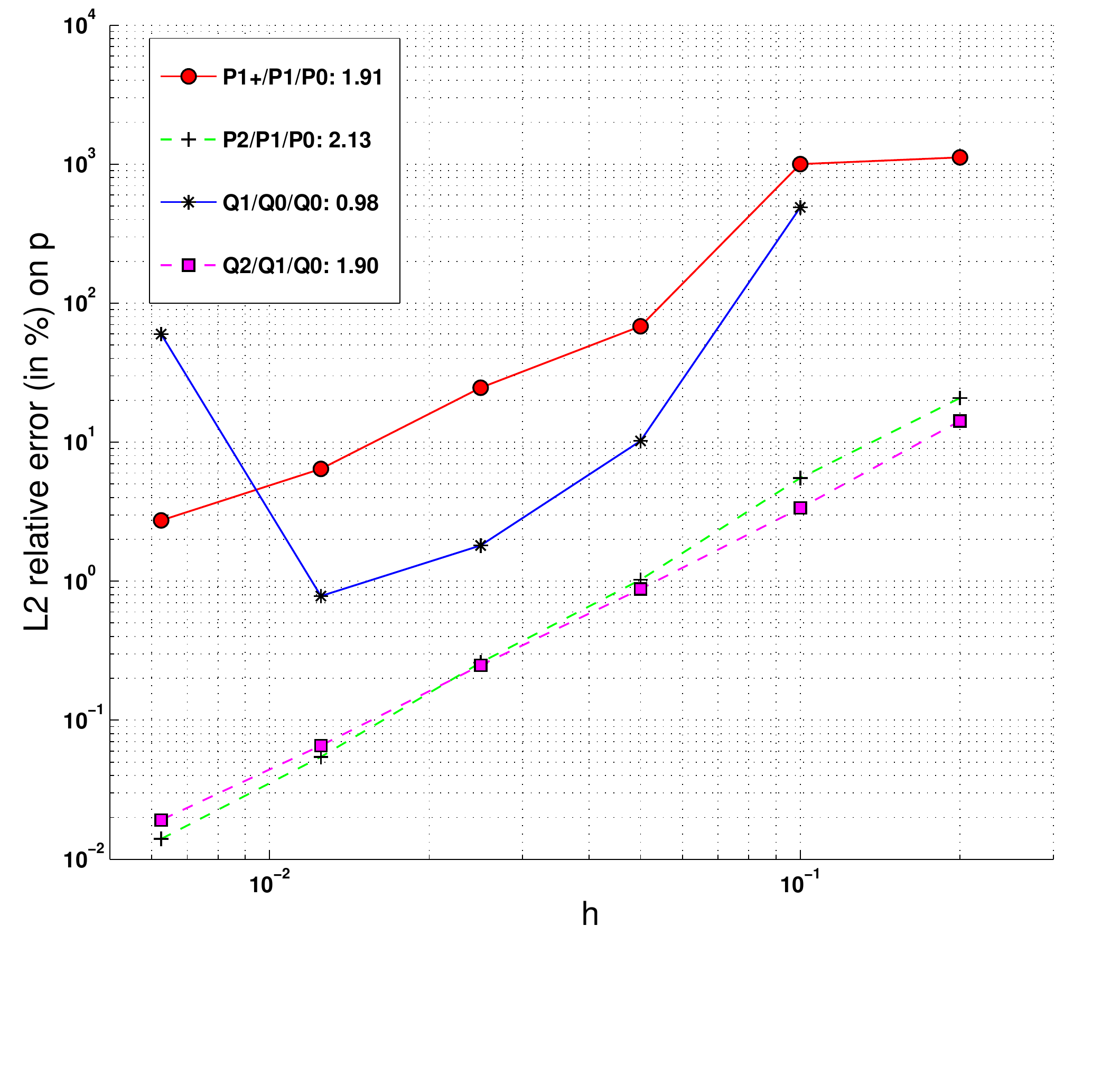} \includegraphics[scale = 0.4]{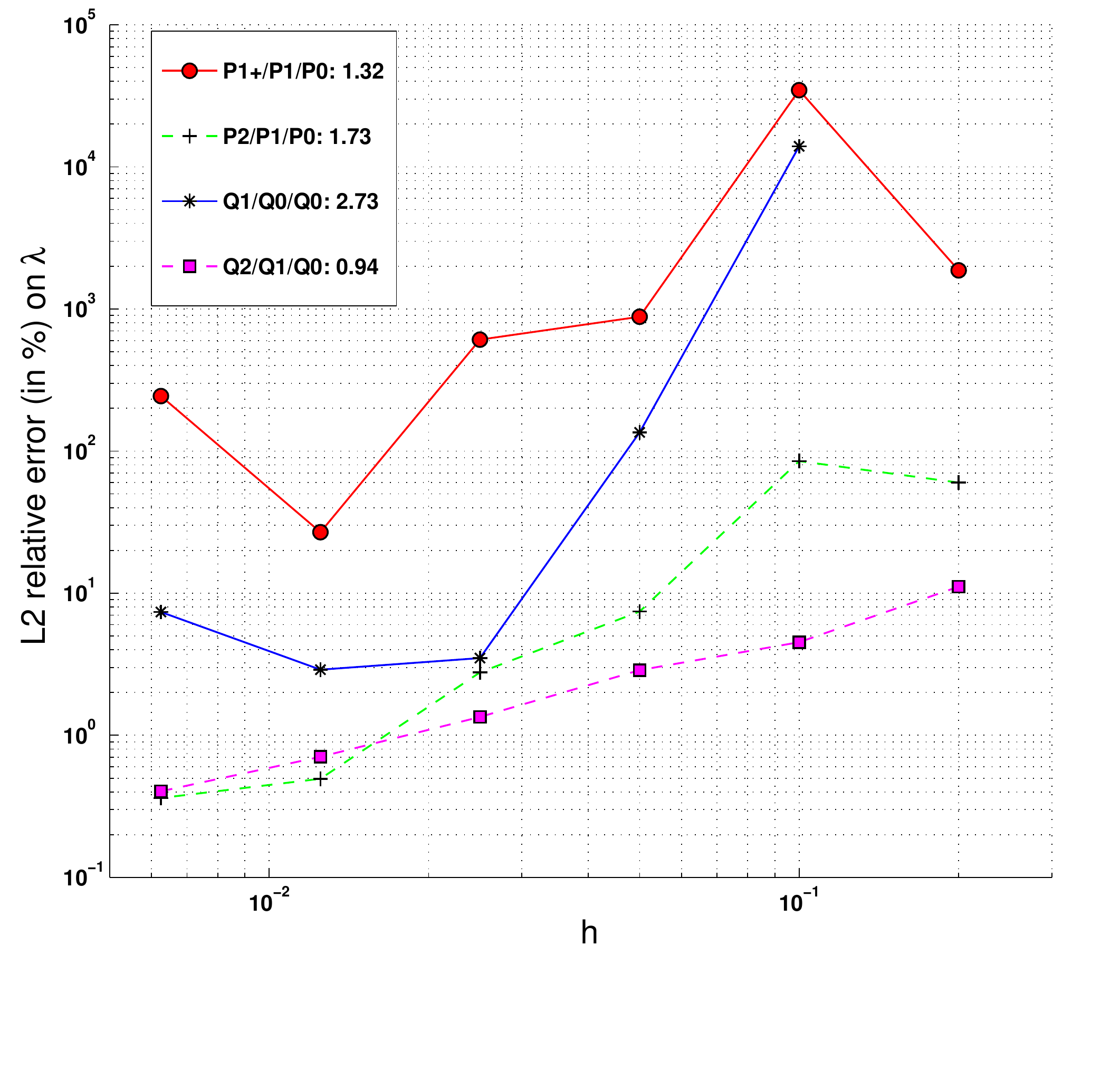}\\
\end{minipage}

\vspace*{-1cm}
\begin{figure}[!h]
\centering\caption{Rates of convergence without stabilization for the velocity/pressure/Lagrange multiplier,\\ for different triplets of finite element spaces.
}
\label{withoutStabCv}
\end{figure}

\FloatBarrier

The convergence for the fluid velocity is highlighted, whereas the convergence for the multiplier seems to not occur, in all cases.
We get the convergence for the pressure, but not for the test Q1/Q0/Q0 which does anyway not satisfy the inf-sup condition.
The rates of convergence are better than what we can expect by the theory for $\bu$ and $p$. The results are not so good for
the multiplier. Indeed, without stabilization, the the order of magnitude for the relative errors lets us think that
the multiplier is not well computed.

\subsection{Numerical experiments with stabilization} \label{secnumstab} \label{secgamma}
In this part, we consider the method with stabilization terms. Additional terms depending on the positive constant $\gamma$ are considered in the variational formulation \eqref{FVaugmented}. In the following, we fix $\gamma = h \gamma_0$, as it is suggested in the proof of Lemma \ref{lemmainfsup} ($\gamma$ is supposed to be constant, which is natural when uniform meshes are considered). The parameter $\gamma$ (or $\gamma_0$) has to respond to a compromise between the coercivity of the system and the weight of the stabilization term. First, the choice of  $\gamma$ is discussed.
We choose the P2/P1/P0 couple of spaces with the space step $h = 0.025$. To characterize a good range of values, we present the condition number (of the whole system) in figure \ref{cond_gamma}, and the relative errors on the multiplier $\blambda$ for $\gamma_0 \in [10^{-14} ; 10^{4}]$ and more precisely for $\gamma_0 \in [0.001 ; 0.200]$ in figure \ref{err_mult_gamma}.

\begin{figure}[!h]
\centering\includegraphics[trim = 0cm 2.5cm 0cm 5cm, clip, scale = 0.30]{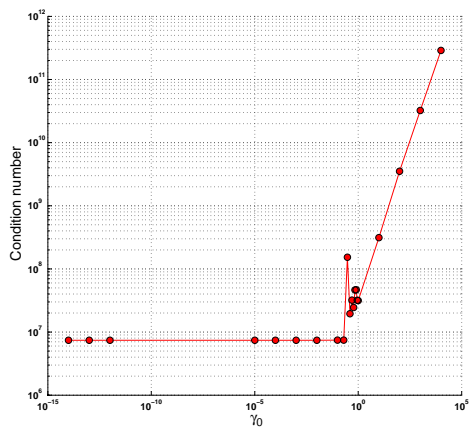}
\vspace*{-0cm}
\centering\caption{The condition number for $\gamma_0 \in [10^{-14} ; 10^{4}]$.}
\label{cond_gamma}
\end{figure}
\FloatBarrier

\vspace*{-0.5cm}
\begin{center}
\includegraphics[trim = 0cm 2.0cm 0cm 5.0cm, clip, scale = 0.30]{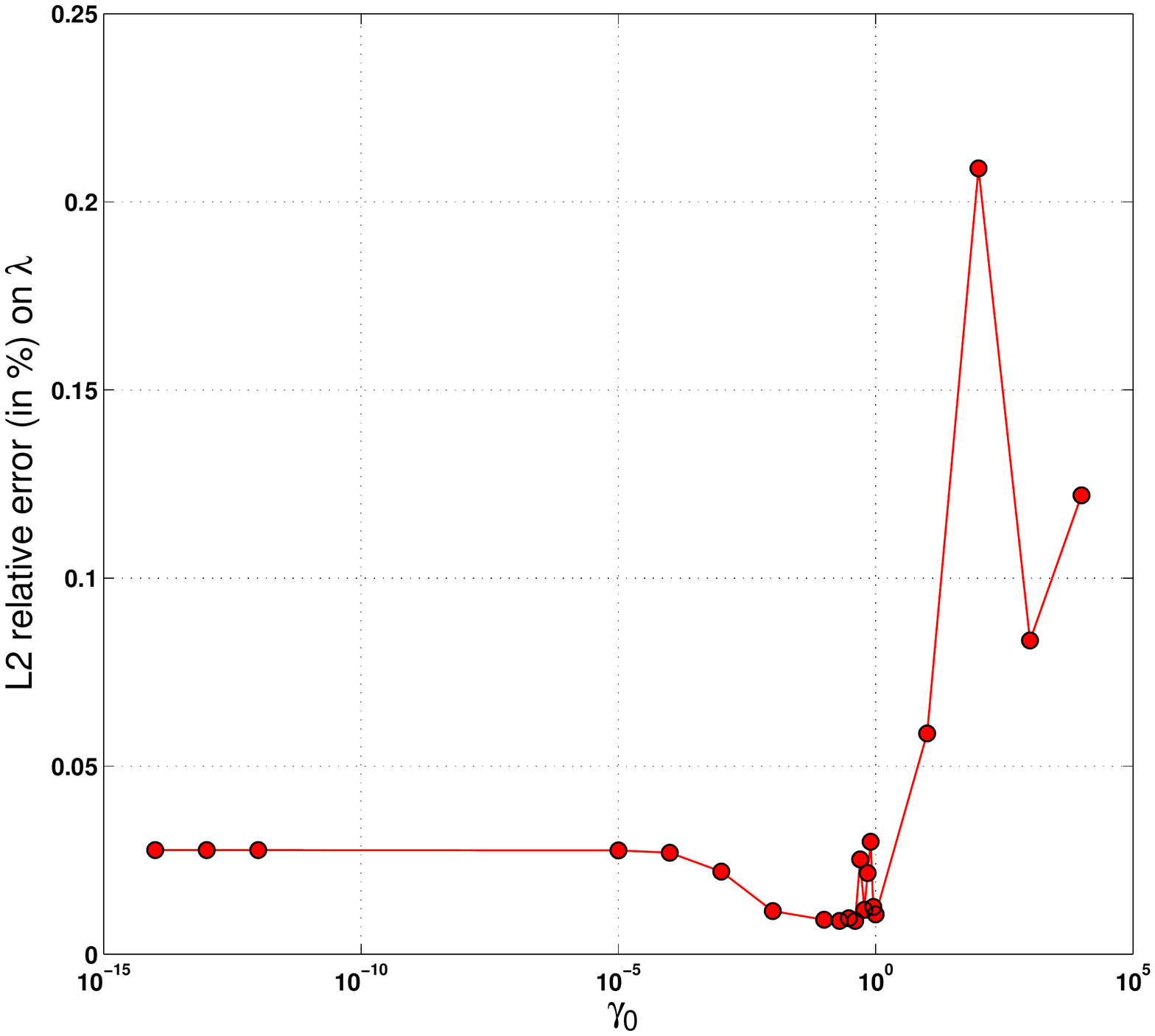}
\includegraphics[trim = 0cm 2.0cm 0cm 5.0cm, clip, scale = 0.30]{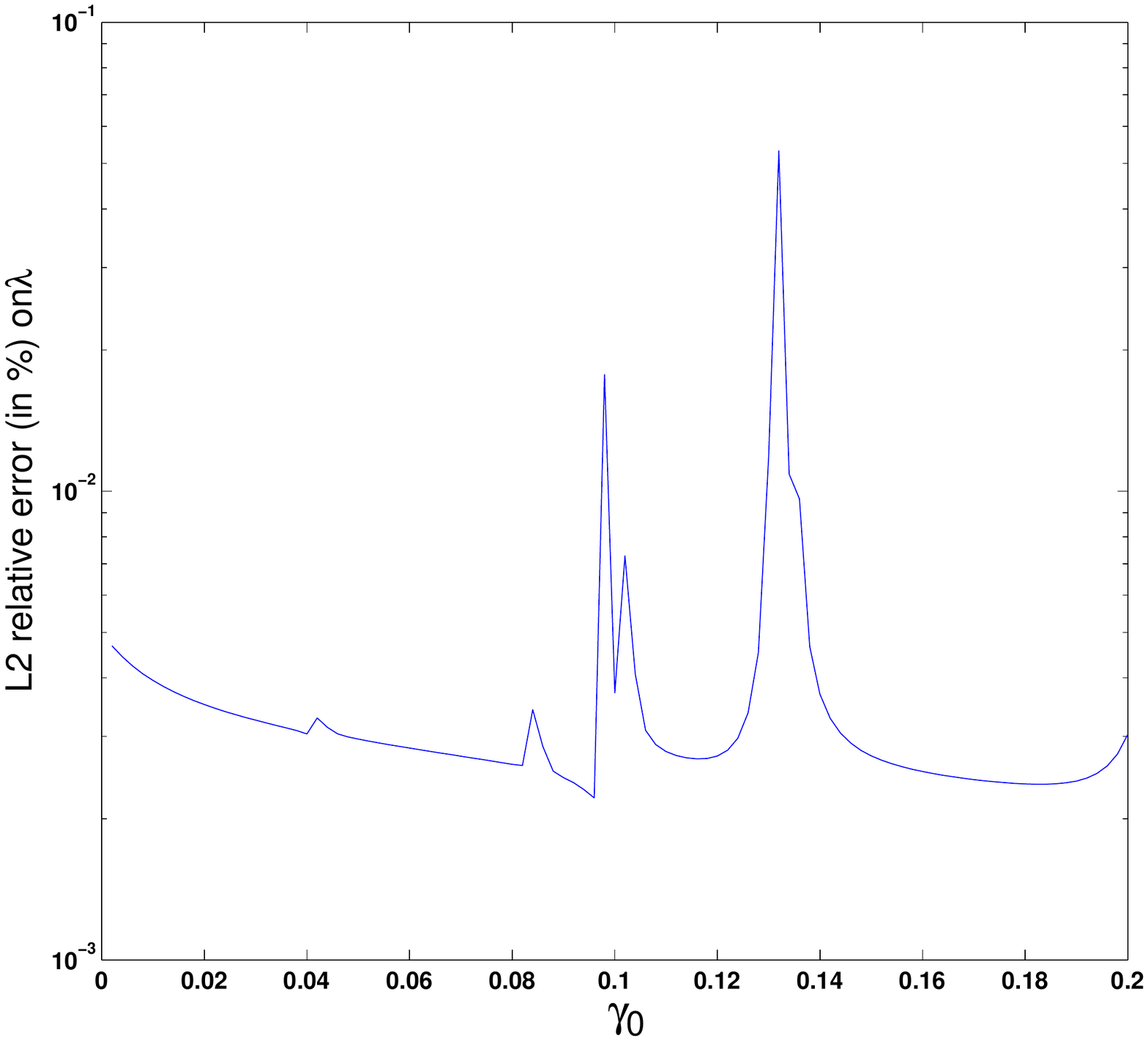}
\end{center}
\vspace*{-0.5cm}

\begin{figure}[!h]
\caption{The relative errors $\|\blambda-\blambda^h\|_{\mathbf{L}^2(\Gamma)}$  for $\gamma_0 \in [10^{-14} ; 10^{4}]$ (left),  $\gamma_0 \in [0.001 ; 0.200]$ (right).}
\label{err_mult_gamma}
\end{figure}

The condition number given for some very small $\gamma_0$ corresponds to the condition number of the system when no stabilization is
used. For all situations, the condition number is degraded when stabilization terms are considered
and can explode when $\gamma_0$ is too large. With regard to the errors on the multiplier $\blambda$, there is no improvement for the relative errors on the multiplier when $\gamma_0$ is too small. When $\gamma_0$ increases, the errors on the multiplier becomes interesting
even if some  peaks can appear (transition zone where the coercivity property is very poor).
Similar observations (same values for $\gamma_0$) are observed on the relative errors for the velocity.\\

With regard to the previous experiments, in the following, we choose $\gamma_0 = 0.05$ (so $\gamma = 0.05 \times h$) and we study
the numerical convergence analysis of the method when stabilization is used.
The following numerical experiments have been made in the same conditions as the one given in section \ref{seccut}. The results are reported in figure \ref{withStabCv}.


\begin{minipage}{21cm}
\hspace*{-2cm} \includegraphics[trim = 0cm 1cm 0cm 6.0cm, clip, scale=0.4]{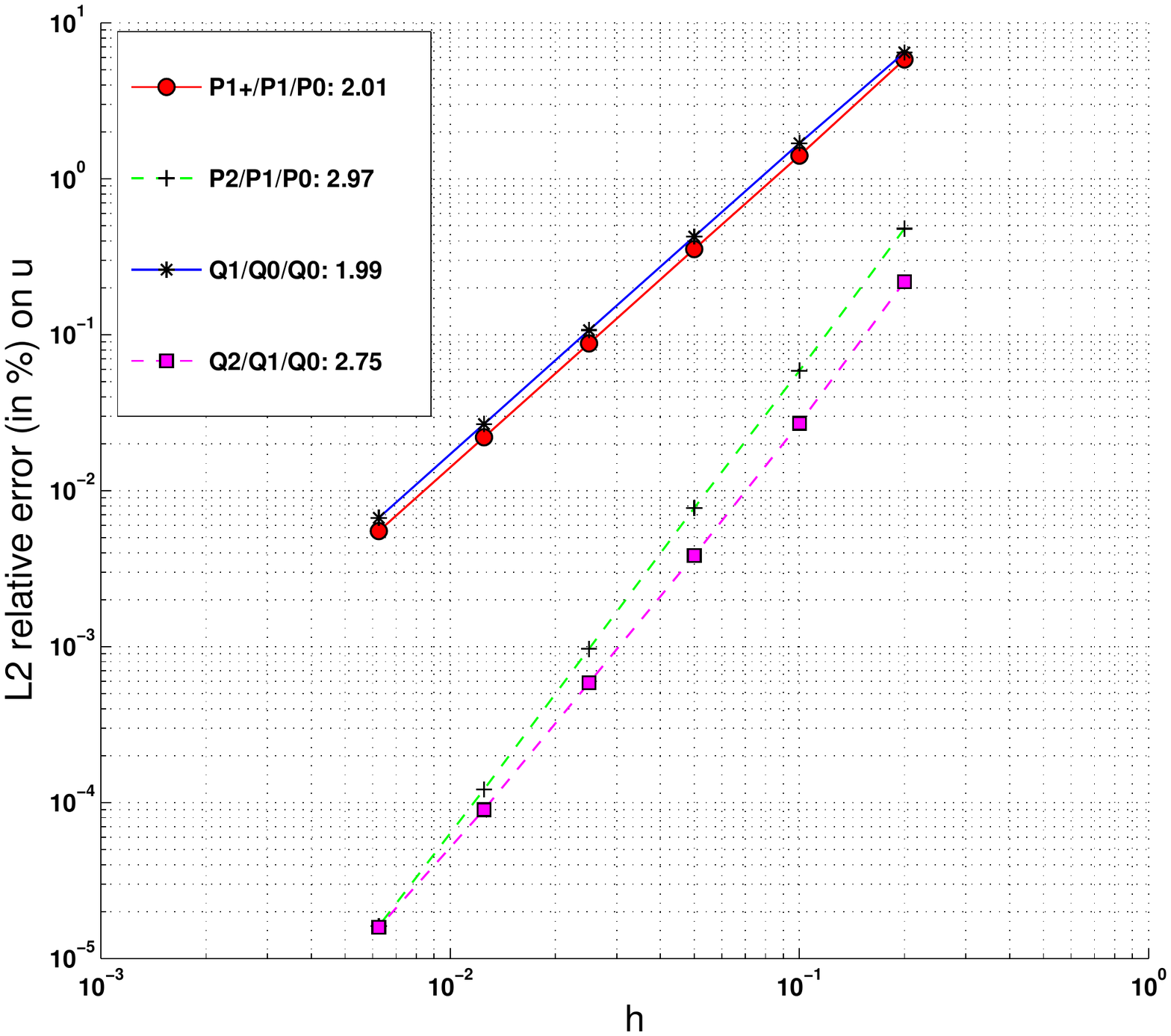}
\includegraphics[trim = 0cm 1cm 0cm 6.0cm, clip, scale = 0.4]{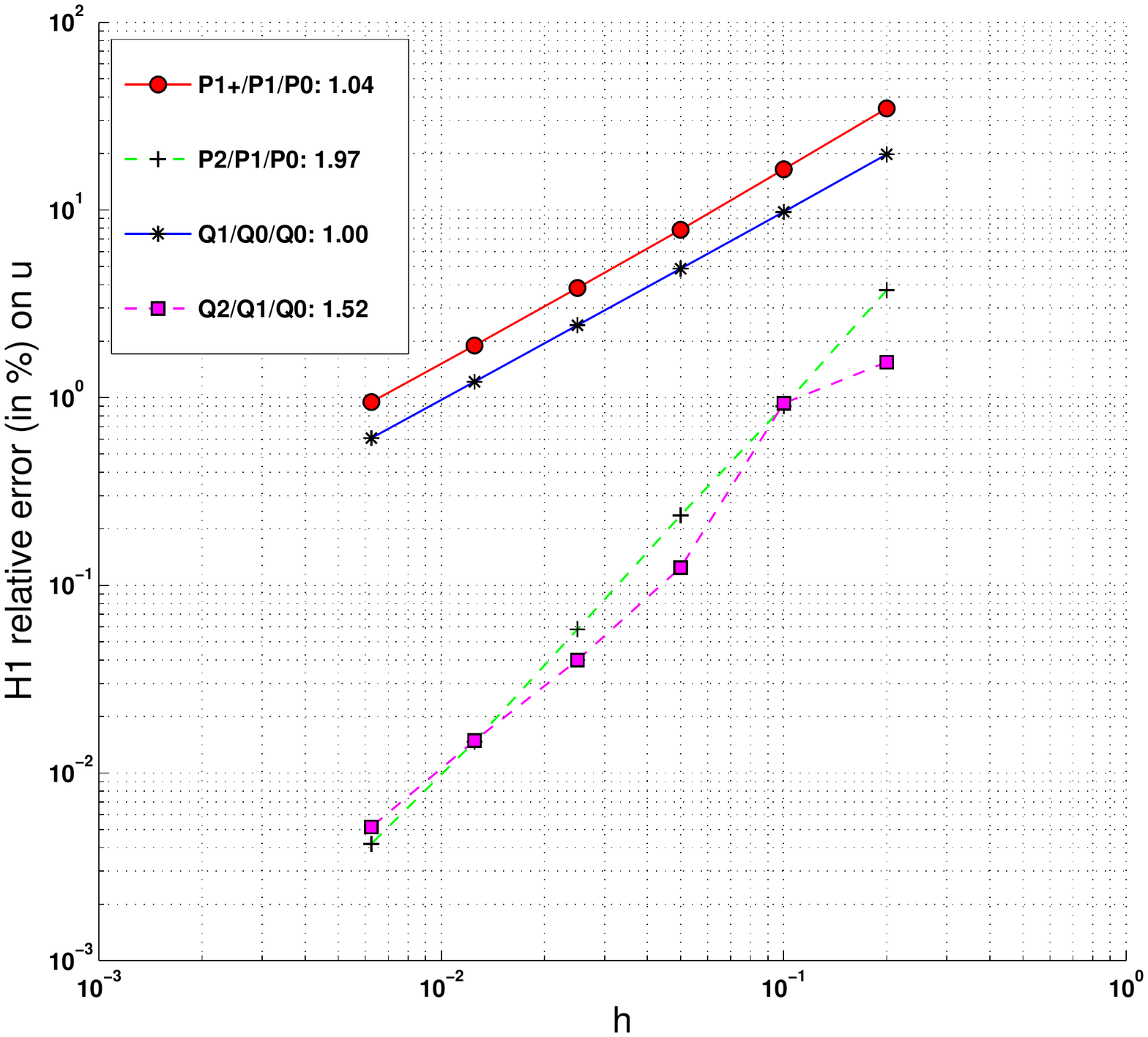}\\
\hspace*{-2cm}\includegraphics[trim = 0cm 0cm 0cm 0cm, clip, scale = 0.4]{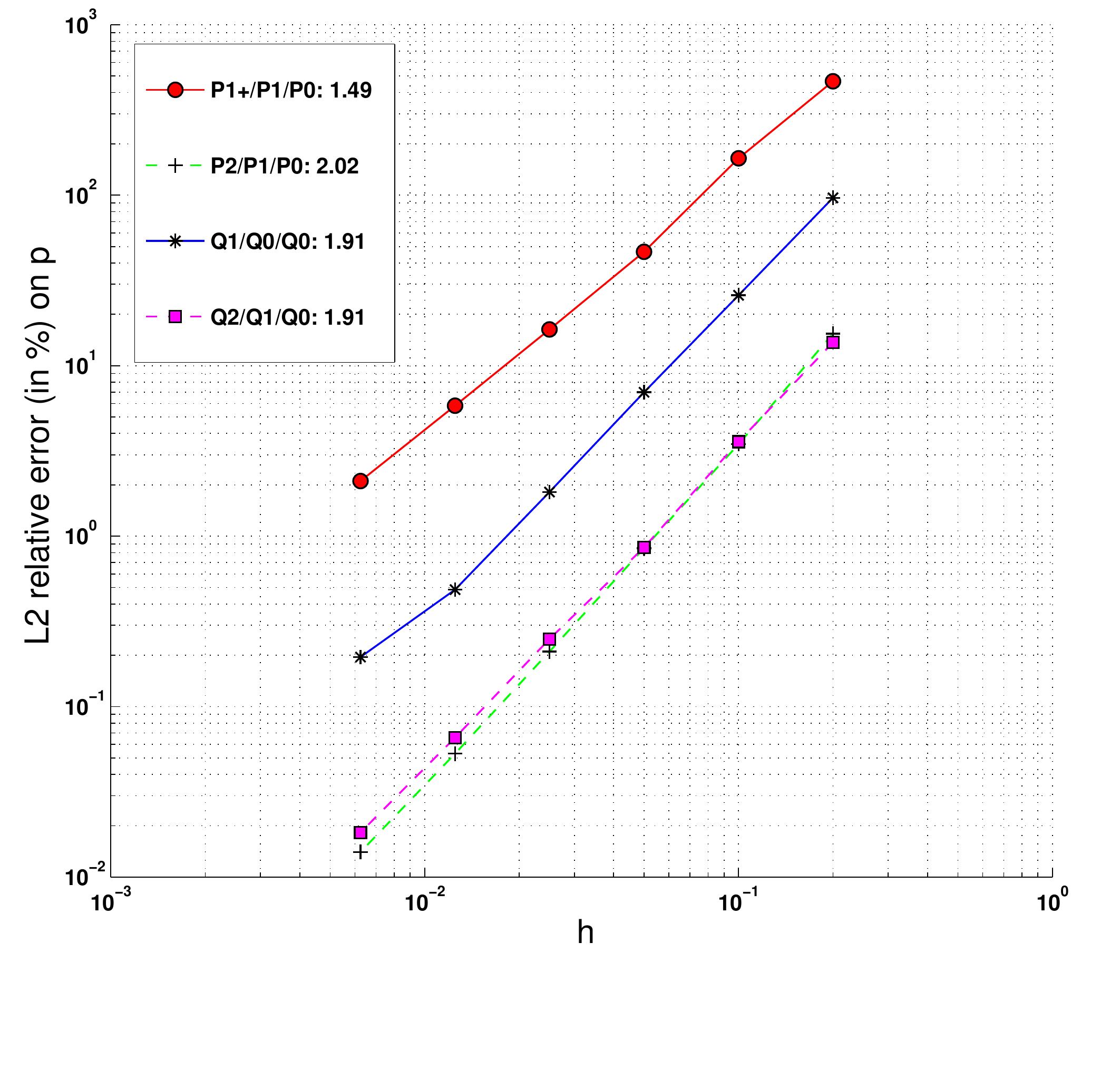} \includegraphics[scale = 0.4]{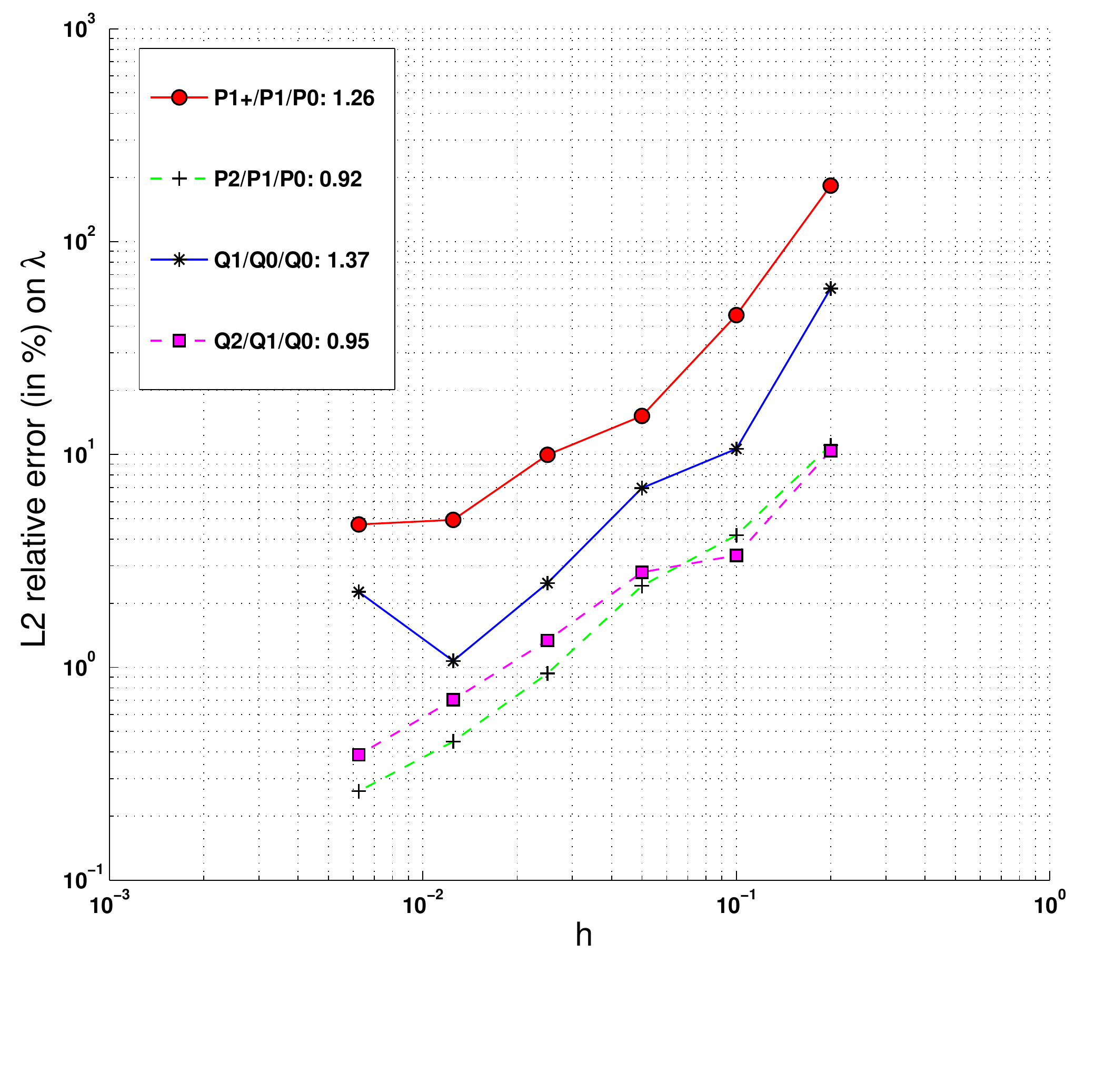}\\
\end{minipage}

\vspace*{-1.5cm}
\begin{figure}[!h]
\centering\caption{Rates of convergence with stabilization for the velocity/pressure/Lagrange multiplier,\\ for different triplets of finite element spaces.
}
\label{withStabCv}
\end{figure}

\FloatBarrier

\textcolor{black}{
We notice that we do not observe substantial differences on the rates of convergence for the errors on the fluid velocity.}
As regards to the pressure, a better behavior (compared to the first method without stabilization) is observed for the couple of spaces Q1/Q0/Q0 that do not satisfy the inf-sup condition. In all cases, the improvements appear for the multiplier. The method enables to recover the convergence for the multiplier.

\textcolor{black}{
\subsection{Tests for different geometric configurations}
In a framework where the solid moves in the fluid domain, we need to perform computations for different geometric configurations, in order to underline the interest of the stabilization method when different types of intersection between the level-set and the regular mesh can be achieved. For that, we compute the $\mathbf{L}^2(\Gamma)$ relative errors on the multiplier $\blambda$ for different positions of the center of the solid, with or without the stabilization technique. The perspective is to anticipate the behavior of the method in an unsteady case, and these tests enables us to avoid the complexity of a full unsteady problem.\\
For $h = 0.05$ and the finite elements triplet P2/P1/P0, we consider the solid as a circle, and we make the abscissa of the center of the circle - denoted by $x_C$ - vary between 0.5 and 0.7 (with a step equal to 0.0005). The variations of the relative error (in \%) on $\blambda$ are represented in blue (without stabilization) and in red (with stabilization).}

\begin{minipage}{21cm}
\hspace*{-0cm}\includegraphics[trim = 0cm 7cm 0cm 7cm, clip, scale=0.6]{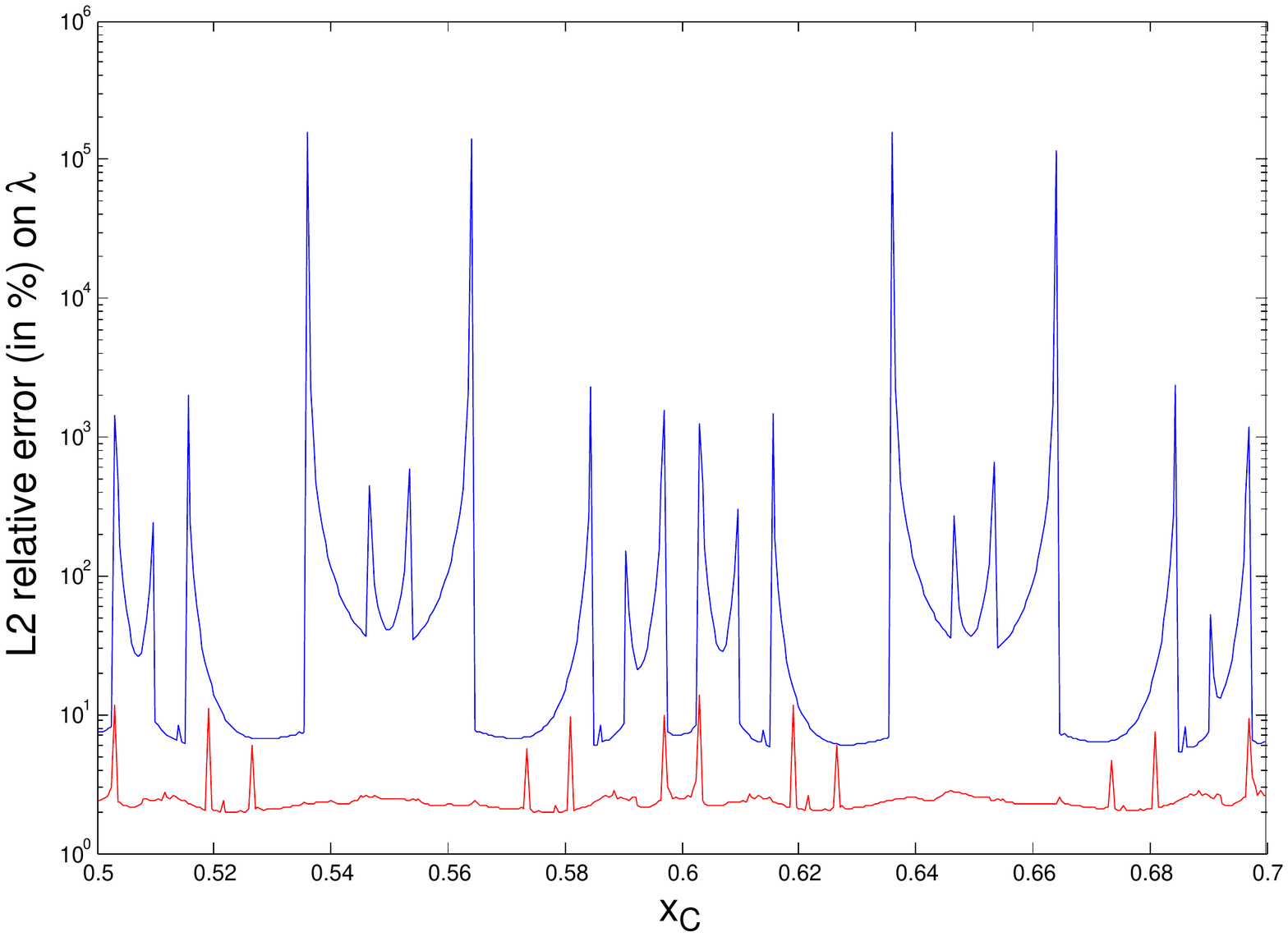}
\end{minipage}

\vspace*{-0.5cm}

\textcolor{black}{
\begin{figure}[!h]
\centering\caption{Behavior of the $\mathbf{L}^2(\Gamma)$ relative error on $\blambda$ (in semi-log scale), in red with the stabilization technique (with $\gamma_0 = 0.05$), in blue without.}
\label{moverror}
\end{figure}
}
\FloatBarrier

\textcolor{black}{In these tests the relevance of our approach using the stabilization technique is highlighted when the intersection between the level-set and the mesh varies. Without stabilization the errors are huge in many cases (see the curve in blue), whereas the robustness of the stabilization technique is demonstrated with regards to the constancy of the relative errors (see the curve in red).
}

\textcolor{black}{
\subsection{Comparison with a boundary-fitted mesh}
For three different values of $h$ and by using the elements P2/P1/P0, we compute the different relative errors (in \%) by using our method (with and without the stabilization technique) and by using a classical code which uses a standard mesh which fits closely the boundaries instead of being cut by the boundary of the solid. The results are given in Tables 1, 2, 3.}\\

\begin{minipage}{21cm}
\hspace*{-1cm}\begin{tabular} {lll}
\includegraphics[trim = 4.5cm 16cm 4.5cm 2.5cm, clip, scale=0.4]{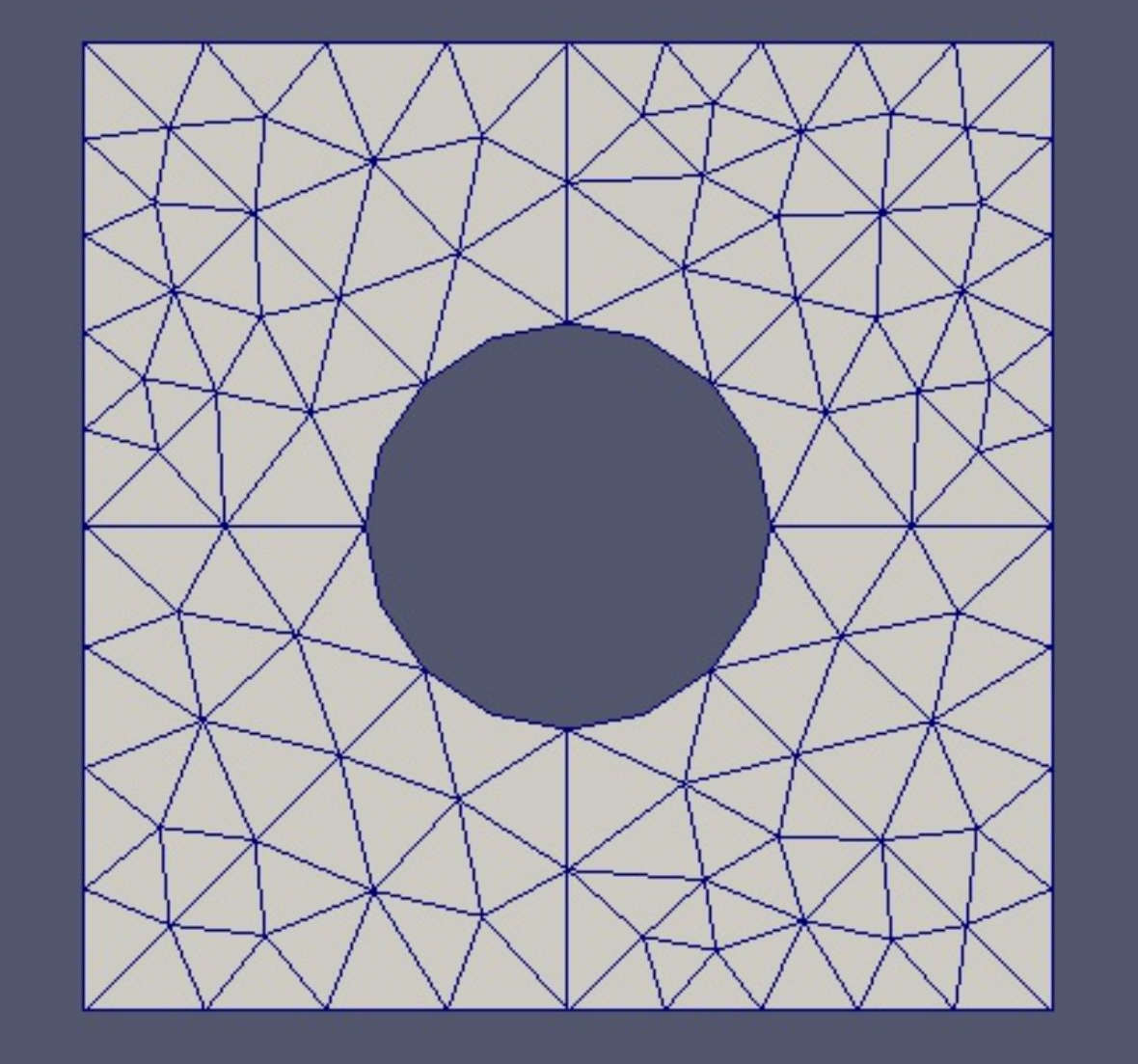} &
\includegraphics[trim = 4.5cm 16cm 4.5cm 2.5cm, clip, scale=0.4]{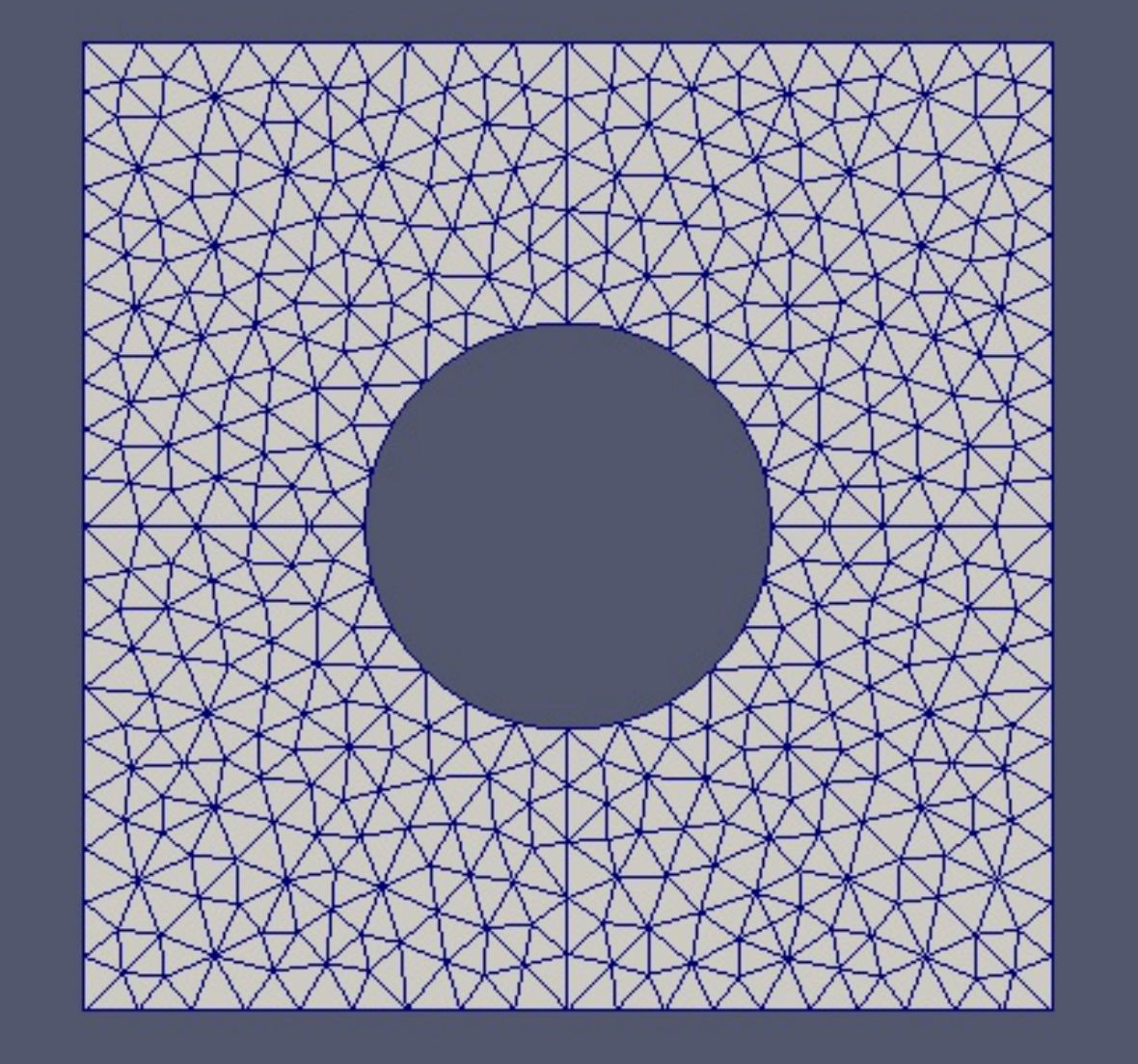} &
\includegraphics[trim = 4.5cm 16cm 4.5cm 2.5cm, clip, scale=0.4]{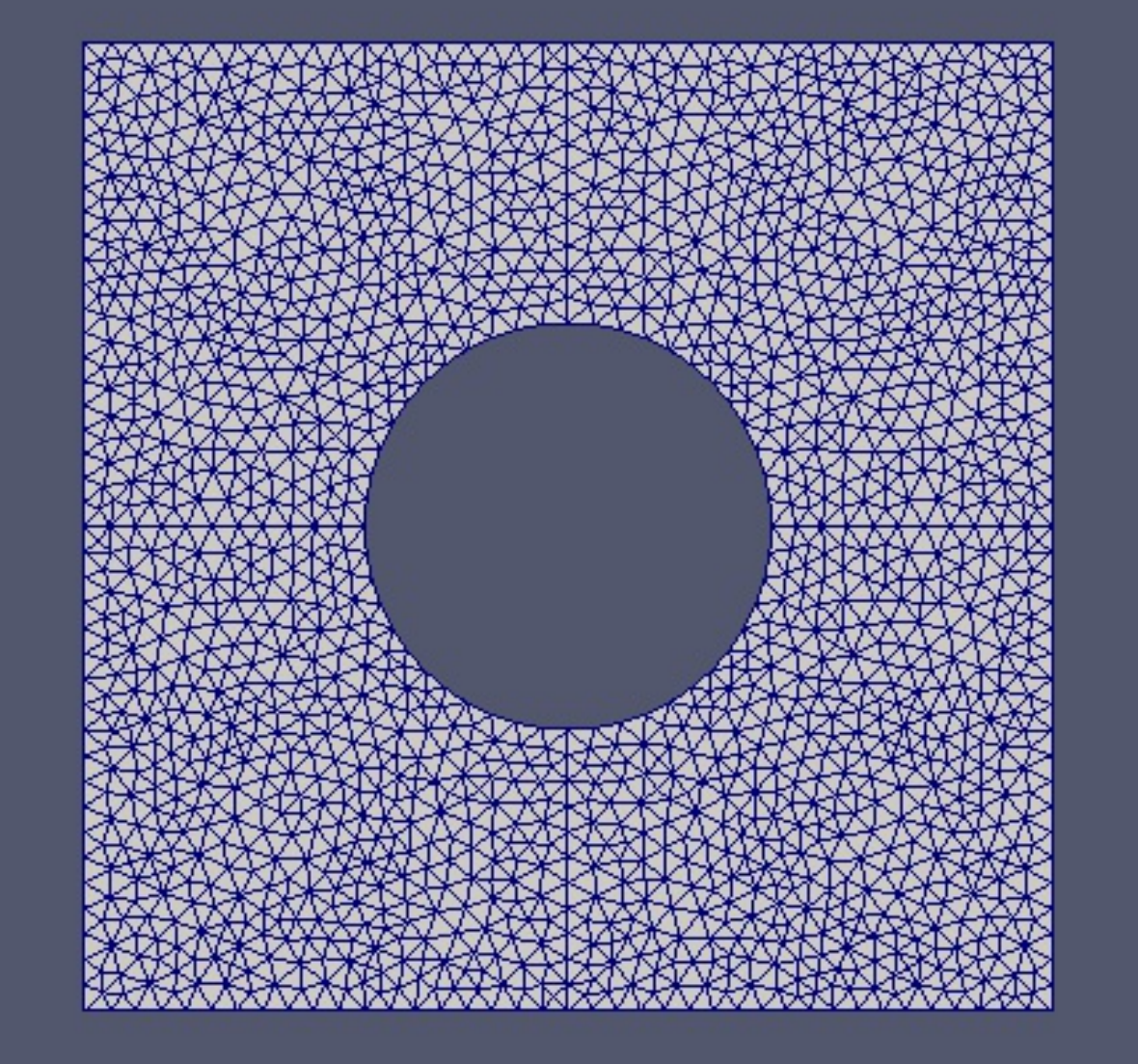}
\end{tabular}
\end{minipage}

\textcolor{black}{
\begin{figure}[!h]
\centering\caption{Different nonuniform boundary-fitted meshes, for which the triangles are not cut.}
\end{figure}
}
\FloatBarrier

\vspace*{0.5cm}

\begin{tabular} {|c|c|c|c|c|}
\hline
$h$ & $\mathbf{L}^2$ error on $\bu$ & $\mathbf{H}^1$ error on $\bu$ & $\L^2$ error on $p$ & $\mathbf{L}^2$ error on $\blambda$ \\
\hline
0.0358201 & 0.146643 & 1.56629 & 3.86771 & 9.61603\\
0.0152703 & 0.00371624 & 0.117115 & 0.751358 & 3.67841\\
0.0066282 &  0.00035697 & 0.0227257 & 0.187311 & 1.85277\\
\hline
\end{tabular}
\begin{center}
\textcolor{black}{Table 1. Errors for a standard uncut mesh.}
\end{center}
\vspace*{0.5cm}

\begin{tabular} {|c|c|c|c|c|}
\hline
$h$ & $\mathbf{L}^2$ error on $\bu$ & $\mathbf{H}^1$ error on $\bu$ & $\L^2$ error on $p$ & $\mathbf{L}^2$ error on $\blambda$ \\
\hline
0.036418 & 0.0353448 & 0.649583 & 2.59781 & 6.76061\\
0.0150695 & 0.00274948 & 0.123396 & 0.662703 & 13.9277\\
0.00662145 & 0.00024883 & 0.0276422 & 0.119263 & 1.57377\\
\hline
\end{tabular}
\begin{center}
\textcolor{black}{Table 2. Errors for a regular cut mesh, without stabilization.}
\end{center}
\vspace*{0.5cm}

\begin{tabular} {|c|c|c|c|c|}
\hline
$h$ & $\mathbf{L}^2$ error on $\bu$ & $\mathbf{H}^1$ error on $\bu$ & $\L^2$ error on $p$ & $\mathbf{L}^2$ error on $\blambda$ \\
\hline
0.036418 & 0.03485 & 0.644208 & 2.46321 & 6.61553\\
0.0150695 & 0.00282232 & 0.12423 & 0.556228 & 3.71191\\
0.00662145 & 0.000251731 & 0.0275953 & 0.104131 & 1.52906\\
\hline
\end{tabular}
\begin{center}
\textcolor{black}{Table 3. Errors for a regular cut mesh, with stabilization ($\gamma_0 = 0.05$).}
\end{center}

\vspace*{0.5cm}

\textcolor{black}{
The results obtained above show that our method enables us to get back the precision provided by a classical boundary-fitted mesh. With regards to the errors on the multiplier $\blambda$, notice that by using our method we need to perform the stabilization technique in order to recover a good approximation of this variable.
}

\textcolor{black}{
\subsection{Discussion of assumptions {\bf A1} and {\bf A2}}
In regard to the assumptions {\bf A1} and {\bf A2} considered for the proof of Lemma \ref{lemmainfsup}, let us also study the behavior of the constant $C$ of these assumptions with respect to the geometric configuration. In order to verify numerically {\bf A2} for instance, we want to solve the optimization problem%
\begin{eqnarray*}
\max_{q_{h}\in Q_{h}}\frac{h\Vert q^{h}\Vert _{\L ^{2}(\Gamma )}^{2}}{\Vert
q^{h}\Vert _{\L ^{2}(\mathcal{F})}^{2}} & = & \max_{q_{h}\in Q_{h}}\frac{%
h(q^{h},q^{h})_{\L ^{2}(\Gamma )}}{(q^{h},q^{h})_{\L ^{2}(\mathcal{F})}}.
\end{eqnarray*}%
One easily shows that the maximum is achieved on the eigenvector $q_{i}^{h}$ of the problem%
\begin{eqnarray*}
h\langle q_{i}^{h},\chi^{h}\rangle_{\L ^{2}(\Gamma )} & = & \lambda _{i}\langle q_{i}^{h},\chi^{h}\rangle_{%
\L ^{2}(\mathcal{F})}\quad \forall \chi^{h}\in Q_{h}
\end{eqnarray*}%
corresponding to the maximal eigenvalue $\lambda _{i}=\lambda _{\max }$. In matrix terms this is rewritten as%
\begin{eqnarray*}
h A_{\L ^{2}(\Gamma )}q^h_{i} = \lambda _{i} A_{\L ^{2}(\mathcal{F})} q^h_{i} & \Longleftrightarrow &
h A_{\L ^{2}(\mathcal{F})}^{-1}A_{\L^{2}(\Gamma)}q^h_{i}=\lambda _{i}q^h_{i},
\end{eqnarray*}%
where $A_{\L ^{2}(\Gamma )}$ and $A_{\L ^{2}(\mathcal{F})}$ are the mass matrices associated with the scalar products in $\L ^{2}(\Gamma )$ and $\L^{2}(\mathcal{F})$ respectively (see below). Hence the optimal constant in {\bf A2} can be calculated as $\lambda _{\max }(hA_{\L ^{2}(\mathcal{F})}^{-1}A_{\L^{2}(\Gamma )})$. The same thing can be done for {\bf A1}. Thus we consider the two following quantities
\begin{eqnarray*}
C_{\bu}(h) = \lambda _{\max }(h A_{\mathbf{H}^1(\mathcal{F})}^{-1} A_{\mathbf{L}^2(\Gamma)}),
& \quad &
C_{p}(h) = \lambda _{\max }(h A_{\L ^{2}(\mathcal{F})}^{-1} A_{\L ^{2}(\Gamma )}), \label{capacity}
\end{eqnarray*}
where $A_{\mathbf{L}^2(\Gamma)}$, $A_{\mathbf{H}^1(\mathcal{F})}$, $A_{\L^2(\Gamma)}$, $A_{\L^2(\mathcal{F})}$ denote the matrices respectively defined by
\begin{eqnarray*}
\begin{array} {ll}
\left(A_{\mathbf{L}^2(\Gamma)}\right)_{ij}  = \int_{\Gamma}D(\bvarphi_i):D(\bvarphi_j)\d \Gamma,
 &
\left( A_{\mathbf{H}^1(\mathcal{F})}\right)_{ij} =
\int_{\mathcal{F}}\nabla \bvarphi_i: \nabla \bvarphi_j \d \mathcal{F}+\int_{\mathcal{F}} \bvarphi_i \cdot \bvarphi_j \d \mathcal{F},\\
\left( A_{\L^2(\Gamma)}\right)_{ij} = \int_{\Gamma} \chi_i \cdot \chi_j \d \Gamma,
&
\left(A_{\L^2(\mathcal{F})}\right)_{ij} = \int_{\mathcal{F}} \chi_i \cdot \chi_j \d \mathcal{F}.
\end{array}
\end{eqnarray*}
For the particular configuration corresponding to $x_C = 0.500$, let us analyze the behavior of $\max(C_{\bu}(h),C_{p}(h))$ when the space step $h$ varies.
}

\begin{minipage}{21cm}
\includegraphics[trim = 0cm 5cm 0cm 7cm, clip, scale = 0.55]{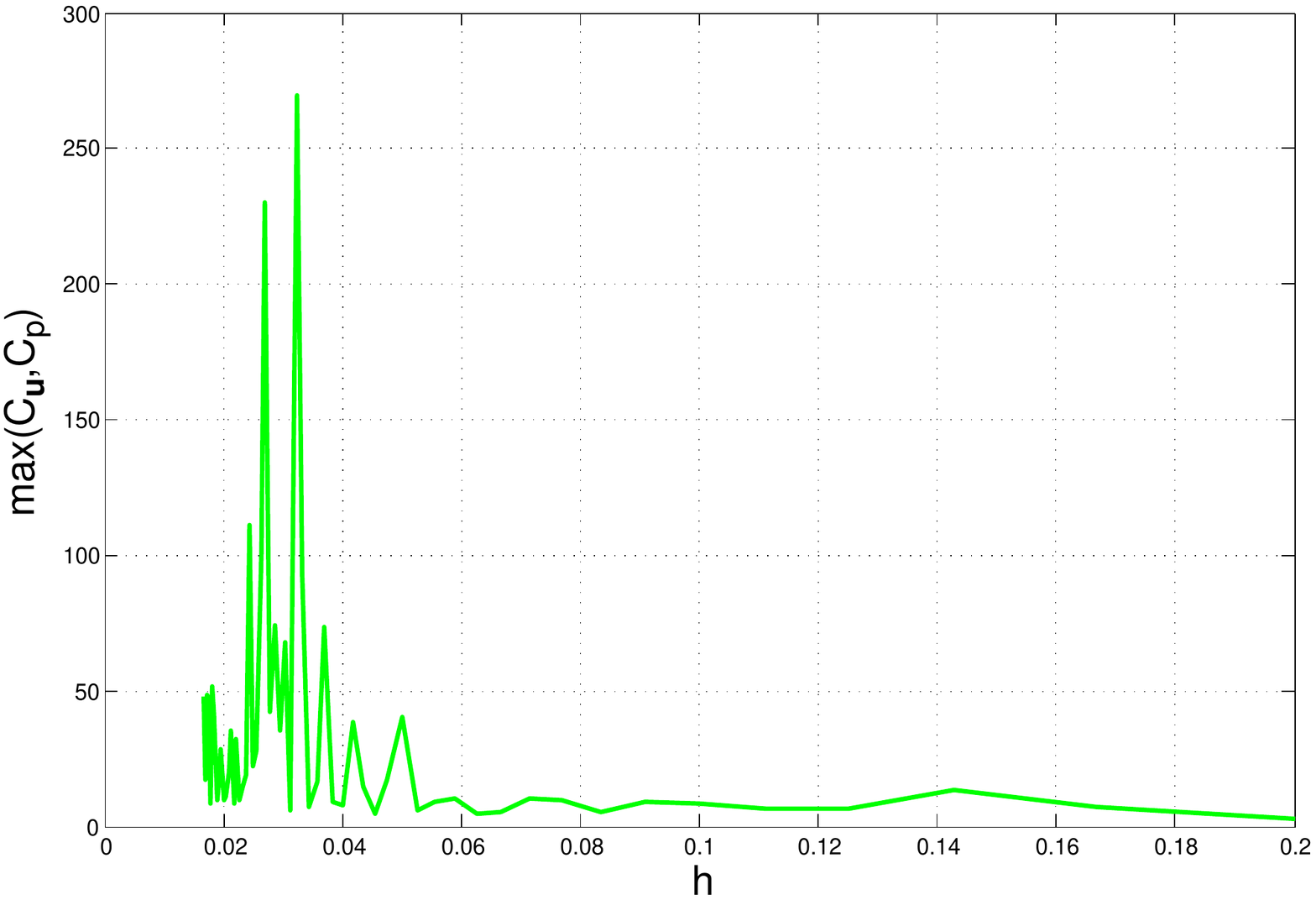}
\end{minipage}

\vspace*{-1.5cm}

\textcolor{black}{
\begin{figure}[!h]
\centering\caption{Numerical illustration of assumptions {\bf A1} and {\bf A2}: $\max(C_{\bu},C_{p})$ in function of $h$.}
\label{movpropre}
\end{figure}
}
\FloatBarrier

\textcolor{black}{
This graph lets us think that the quantities $C_{\bu}$ and $C_{p}$ are not constant with respect to $h$ (specially when $h$ becomes small), and thus the assumptions {\bf A1} and {\bf A2} are not satisfied in practice. However, concerning the value of $h$ for which they are not satisfied, we get numerically the convergence on the multiplier. At this stage we need to consider these assumptions only for proving the theoretical convergence of the stabilization technique (see Lemma \ref{lemmainfsup}).
}
\FloatBarrier


\section{Some practical remarks on the numerical implementation} \label{seccomments}

The numerical implementation of the method for Stokes problem is based on the code developed under {\sc Getfem++} Library \cite{Getfem}
for Poisson problem. The system is solved using the library SuperLU \cite{SuperLU}. The advantages of using the {\sc Getfem++} library (besides its simplicity of developing finite element codes) is that several specific difficulties have been already resolved.
Notably,
\begin{itemize}
\item[--] to define basis functions of ${\mathbf{W}}^h$  from traces on $\Gamma$ of the basis functions of $\tilde{{\mathbf{W}}^h}$.
Indeed, their independence is not ensured and numerical manipulations must be done in order to eliminate possible redundant functions (and avoid to
manipulate singular systems),
\item[--]  to localize the interface between the fluid and the structure, a level-set function which is already implemented (as it is done in \cite{Stolarska} for instance),
\item[--]  to compute properly the integrals over elements at the interface (during assembling) external call to {\sc Qhull} Library \cite{qhull}
is realized (see figure \ref{qhull_fig}).
\end{itemize}

\begin{figure}[!h]
\centering\includegraphics[scale=0.4]{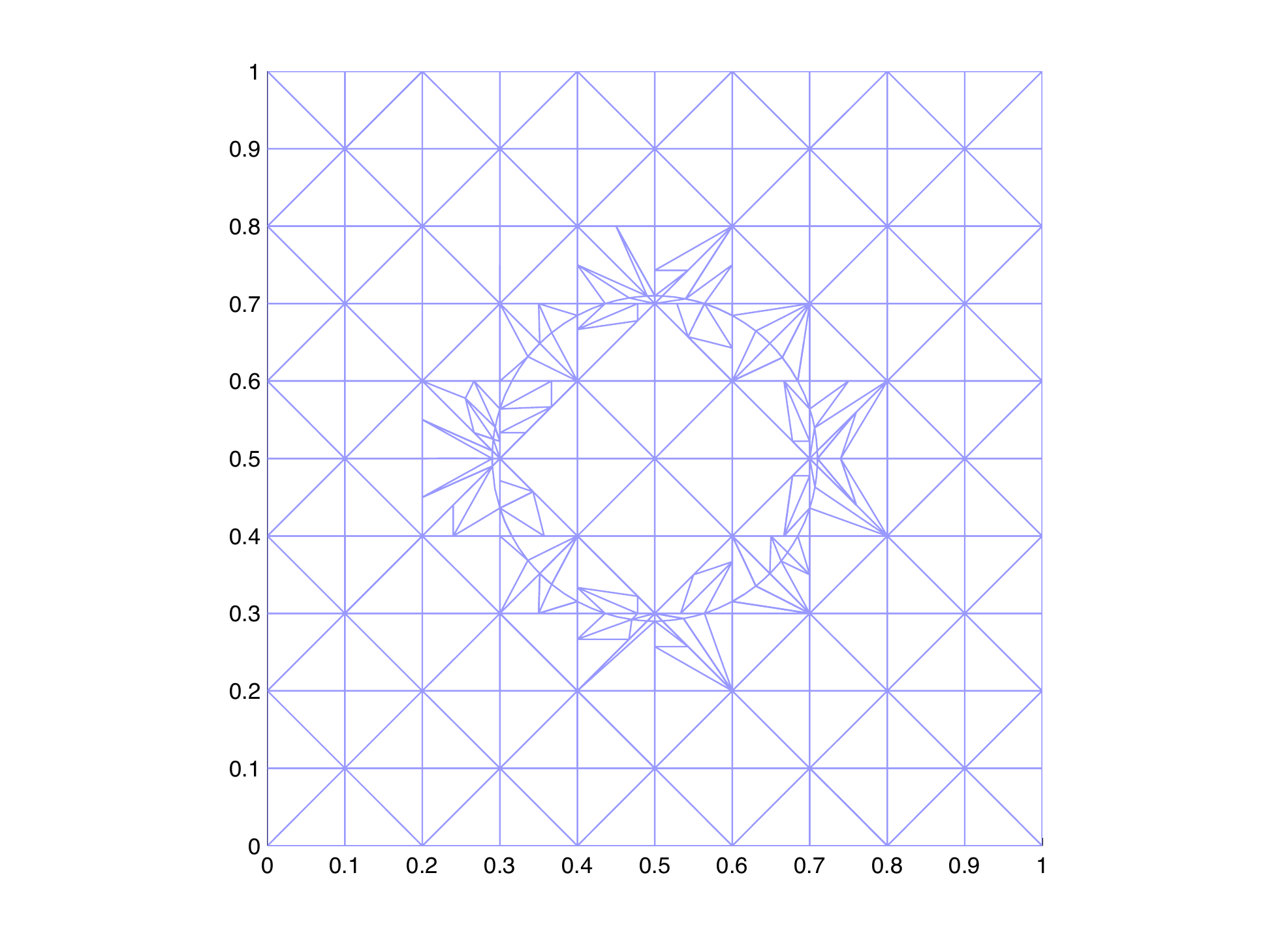}
\centering\caption{Local treatment at the interface using {\sc Qhull} Library.}
\label{qhull_fig}
\end{figure}


As mentioned in the paper \cite{HaslR}, it is possible to define a reinforced stability to prevent
difficulties that can occur when the intersection of the solid and the mesh over the whole domain introduce "very small"
elements. The technique is based on a strategy to select elements which are better to deduce the normal derivative  on $\Gamma$.
A similar approach is given in \cite{Pitkaranta}. This method has been tested for the Dirichlet problem in \cite{HaslR}, but it is not observed substantial improvements with this enriched stabilization, compared to the results obtained with the stabilization method detailed in this paper.
However, we expect to take benefits of this second stabilization method  when the boundary $\Gamma$ is led to move through the time,
in particular in unsteady framework and fluid-structure interactions.

\textcolor{black}{
\section{Application to a fluid-structure interaction problem} \label{secFSI}
The motivation of our approach lies in the perspective of simulations and control of a fluid-solid model for instance. Let us give a simple illustration of that.}

\subsection{Coupling with a moving rigid solid}
In this section, we consider a moving rigid solid which occupies a time-depending domain $\mathcal{S}(t)$. The displacement of a rigid solid is given by
\begin{eqnarray*}
X(\y,t) & = & \bh(t) + \mathbf{R}(t) \y, \quad \y \in \mathcal{S}(0), \\
\mathcal{S}(t) & = & \bh(t) + \mathbf{R}(t)\mathcal{S}(0),
\end{eqnarray*}
where $\bh(t)$ denotes the coordinates of the center of mass of the solid, and $\mathbf{R}(t)$ is the rotation which describes the orientation of the solid with respect to its reference configuration. In dimension 2, this orientation can be given by a single angle $\theta(t)$, and we have
\begin{eqnarray*}
\mathbf{R}(t) & = & \left( \begin{matrix}
\cos(\theta(t)) & -\sin(\theta(t)) \\ \sin(\theta(t)) & \cos(\theta(t))
\end{matrix} \right).
\end{eqnarray*}
In dimension 2, the angular velocity $\omega(t) = \theta'(t)$ is a scalar function.
The fluid domain is given by $\mathcal{O} \setminus \overline{\mathcal{S}(t)} = \mathcal{F}(t)$. The state of the corresponding full system is then defined by the fluid velocity and pressure, ${\bf u}$ and $p$, and the position of the solid given by the coordinates of its center of mass ${\bf h}(t)$ and its angular velocity $\omega(t)$. The coupling between the fluid and the structure is mainly made at the interface $\Gamma$, through the Dirichlet condition
\begin{eqnarray*}
\bu(x,t) & = & \bh'(t) + \omega(t)({\bf x}-\bh(t))^{\bot}, \quad {\bf x} \in \Gamma(t),
\end{eqnarray*}
and through two differential equations which link the position of the solid and the forces that the fluid exerts on its boundary, as follows
\begin{eqnarray}
M\bh''(t) & = & - \int_{\Gamma(t)}\sigma(\bu,p)\bn \d \Gamma - M\bg, \label{PFD} \\
I\omega'(t) & = & - \int_{\Gamma(t)}({\bf x}-\bh(t))^{\bot} \cdot \sigma(\bu,p)\bn \d \Gamma. \nonumber
\end{eqnarray}
\textcolor{black}{The vector $\bg$ denotes the gravity field}. Thus, obtaining a good approximation for $\sigma({\bf u},p){\bf n}$ is essential for simulating the trajectories of the solid.


\textcolor{black}{
\subsection{Illustration: Free fall of a ball}
The full model described above would necessitate particular attention to the time discretization. Indeed, for instance the value of the velocity that we would have to consider in the fluid region released by the solid between two time steps has to be discussed. Thus, instead of considering the full problem, let us consider a simplified approach where the time-dependence aspect is governed only by the position of the solid, and not by the time-derivative of the fluid velocity (which requires to tackle the difficulty aforementioned).\\
A simple illustration consists in simulating in 2D the fall of a rigid ball submitted to the gravity force at low Reynolds number. The state of the fluid
is then governed by the Stokes system we consider in this paper, and the time discretization is only about the dynamics of the solid. The radius of the ball is still $R = 0.21$, and its initial position given by the center of the ball
$C = [x_C,y_C] = [0.5,0.75]$.
By symmetry, if we assume that the initial velocities are null, then the displacement of the ball is only vertical. Thus we impose the Dirichlet condition in the fluid-solid interface as being only
\begin{eqnarray*}
\bu = \bh',
\end{eqnarray*}
and the function $\bh' = (0, \bh_2')^T$ satisfies \eqref{PFD} which is then reduced to the 1D differential equation
\begin{eqnarray}
M\bh_2''(t) & = & -\alpha[\bh(t)]_2\bh_2'(t) - 9.81M,      \label{eqdiffall}
\end{eqnarray}
where $\displaystyle \alpha[\bh(t)] = \int_{\Gamma(t)} \sigma(\hat{\bu},\hat{p})\bn\d \Gamma$, with $\Gamma(t) =  \left\{ \bh(t) + \y \mid \ \y \in \Gamma(0) \right\}$ (the subindex 2 is used for the second component of the vector), and $(\hat{\bu}, \hat{p})$ is the solution of
\begin{eqnarray*}
- \nu \Delta \hat{\bu} + \nabla \hat{p} & = & 0 \quad \text{in } \mathcal{F},  \\
\div \ \hat{\bu} & = & 0 \quad \text{in }\mathcal{F},  \\
\hat{\bu} & = & 0 \quad \text{on } \p \mathcal{O},  \\
\hat{\bu} & = & (0,1)^T \quad \text{on } \Gamma.
\end{eqnarray*}
Indeed, the functions $\bu$ and $p$ are linear we respect to $\bh'$. We discretize \eqref{eqdiffall} with a semi-implicit scheme, as follows
\begin{eqnarray*}
\frac{M}{\Delta t}\left({\bh'}^{n+1}_2 - {\bh'}^{n}_2 \right) & = & - \alpha({\bh}^{n})_2 {\bh'}^{n+1}_2 - 9.81M, \\
\frac{1}{\Delta t}\left({\bh}^{n+1}_2 - {\bh}^{n}_2 \right) & = & {\bh'}^{n+1}_2.
\end{eqnarray*}
For the simulation we choose $h = 0.0125$ for the space step, still $\gamma_0 = 0.05$ for the stabilization parameter, the finite elements triplet P2/P1/P0, $\Delta t = 10^{-4}$ for the time step, $\nu = 1$ and $M = 0.02$. We represent the amplitude of the velocity at different moments in figure 12.
}
\hfill \\

\hspace*{1.0cm}
\begin{tabular} {ccc}
\includegraphics[trim = 1cm 1cm 1cm 1cm, clip, scale=0.3]{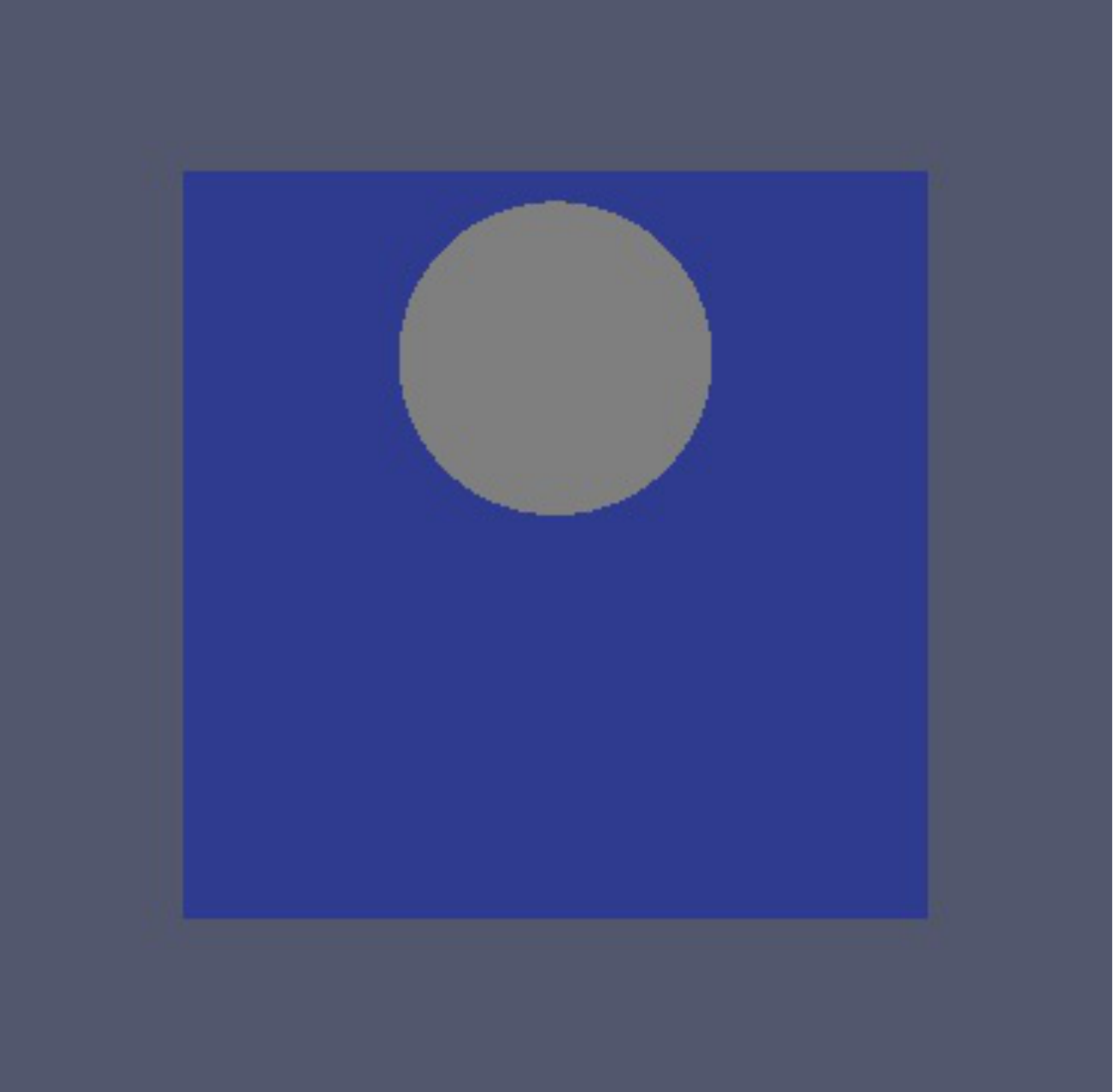} &
\includegraphics[trim = 1cm 1cm 1cm 1cm, clip, scale=0.3]{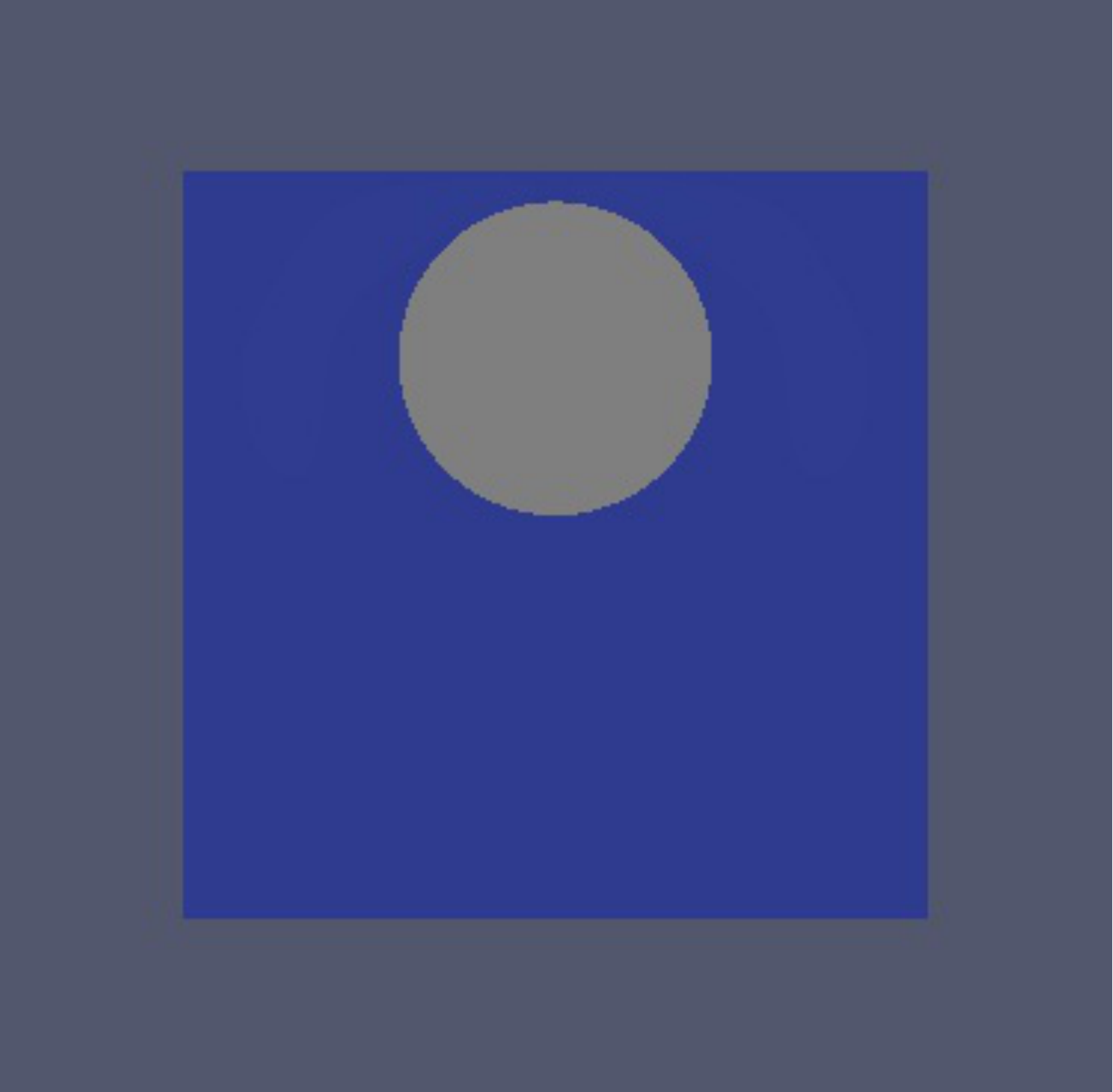} &
\includegraphics[trim = 1cm 1cm 1cm 1cm, clip, scale=0.3]{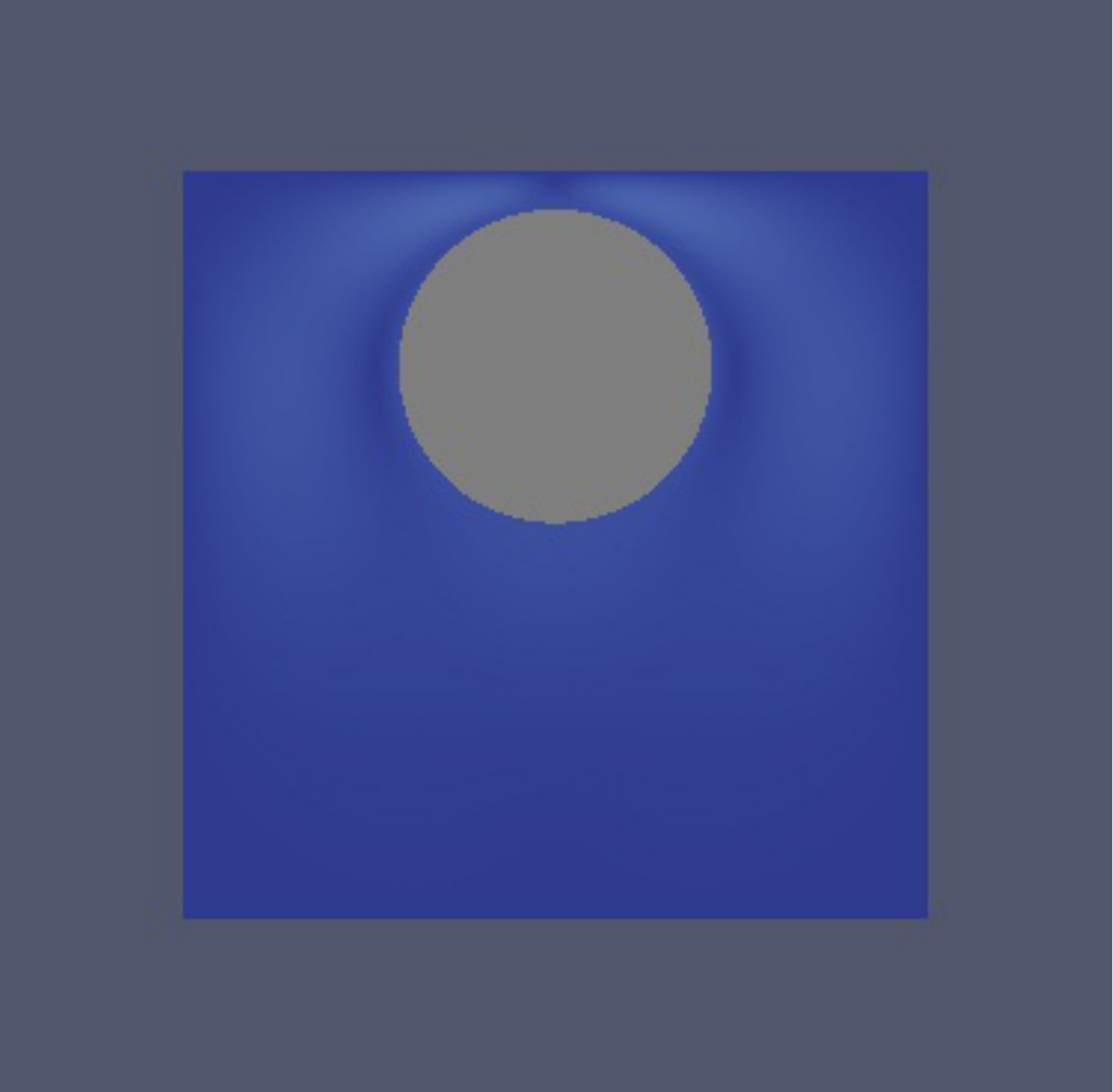}\\
 $t = 0$ & $t = 1$ & $t = 11$\\
\hfill \\
\includegraphics[trim = 1cm 1cm 1cm 1cm, clip, scale=0.3]{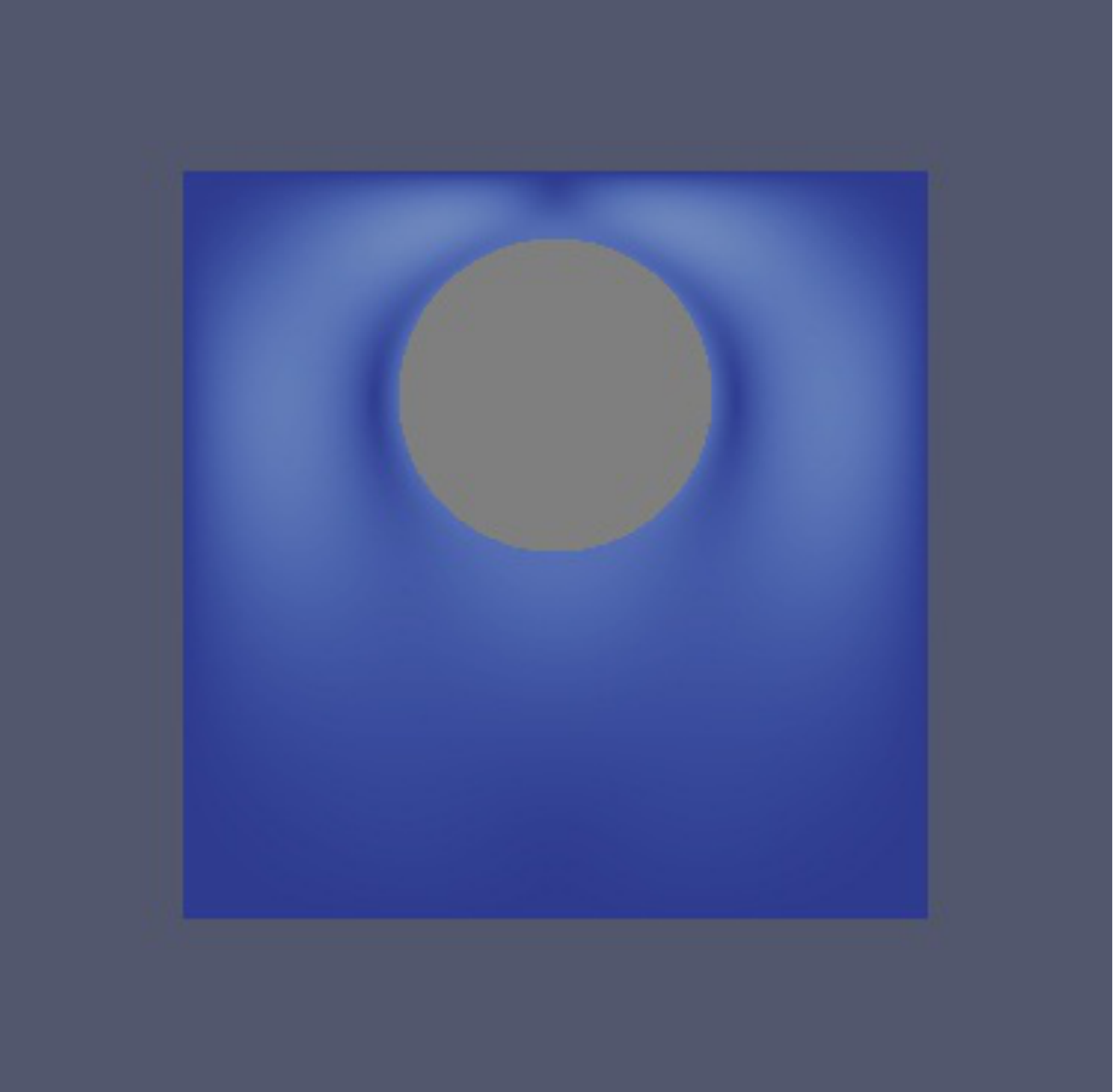} &
\includegraphics[trim = 1cm 1cm 1cm 1cm, clip, scale=0.3]{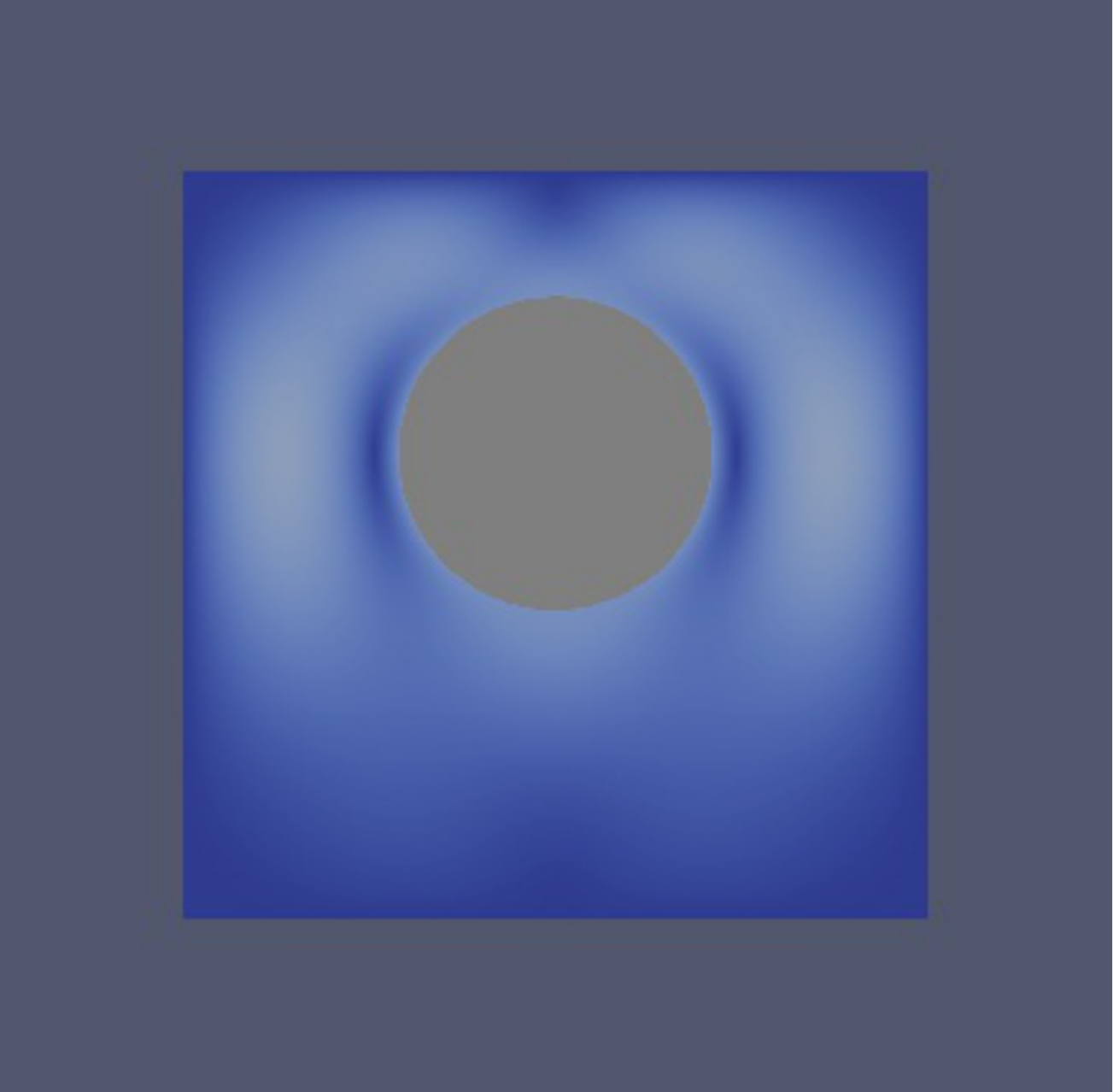} &
\includegraphics[trim = 1cm 1cm 1cm 1cm, clip, scale=0.3]{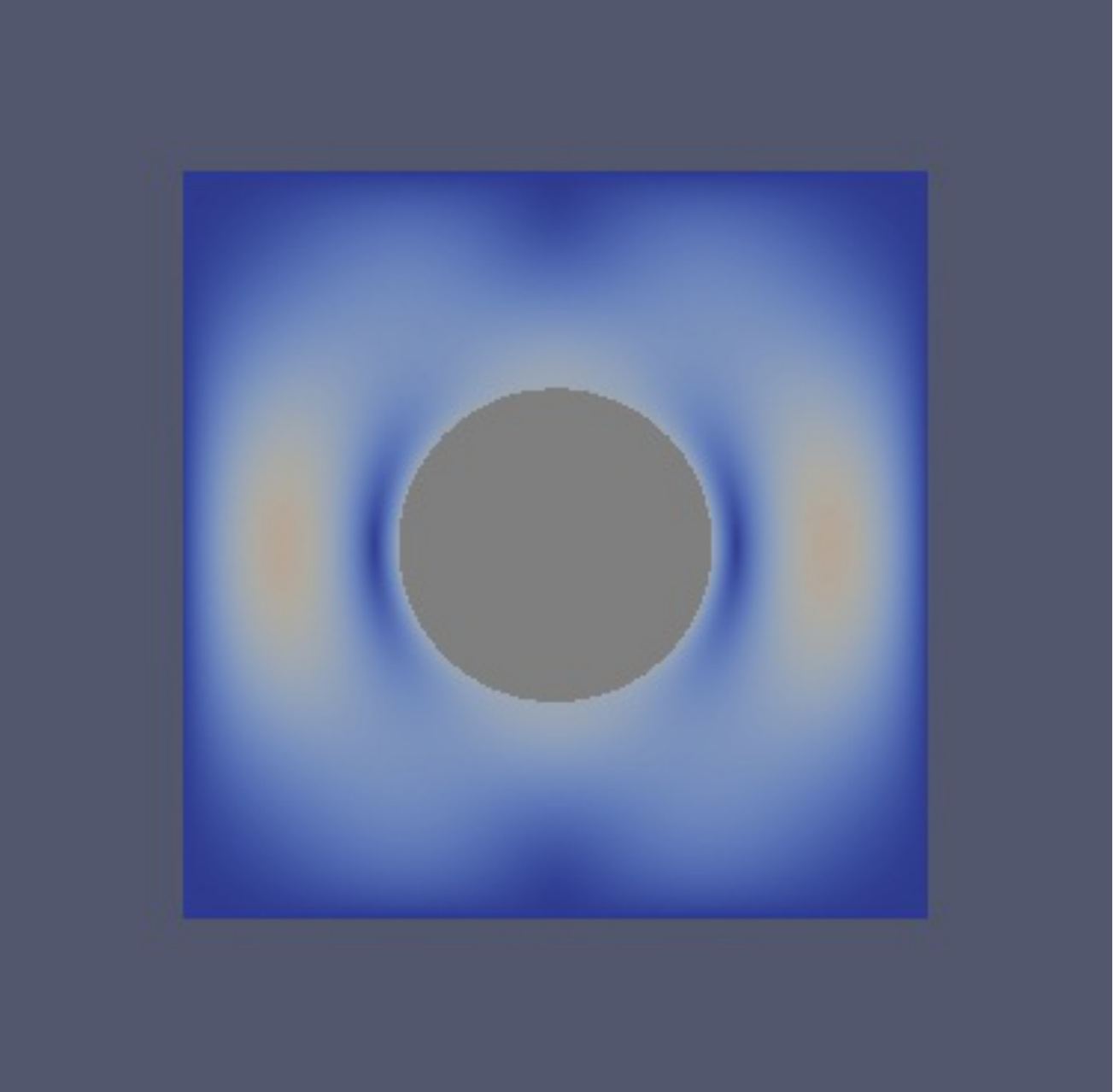}\\
 $t = 21$ & $t = 31$ & $t = 41$\\
\hfill \\
\includegraphics[trim = 1cm 1cm 1cm 1cm, clip, scale=0.3]{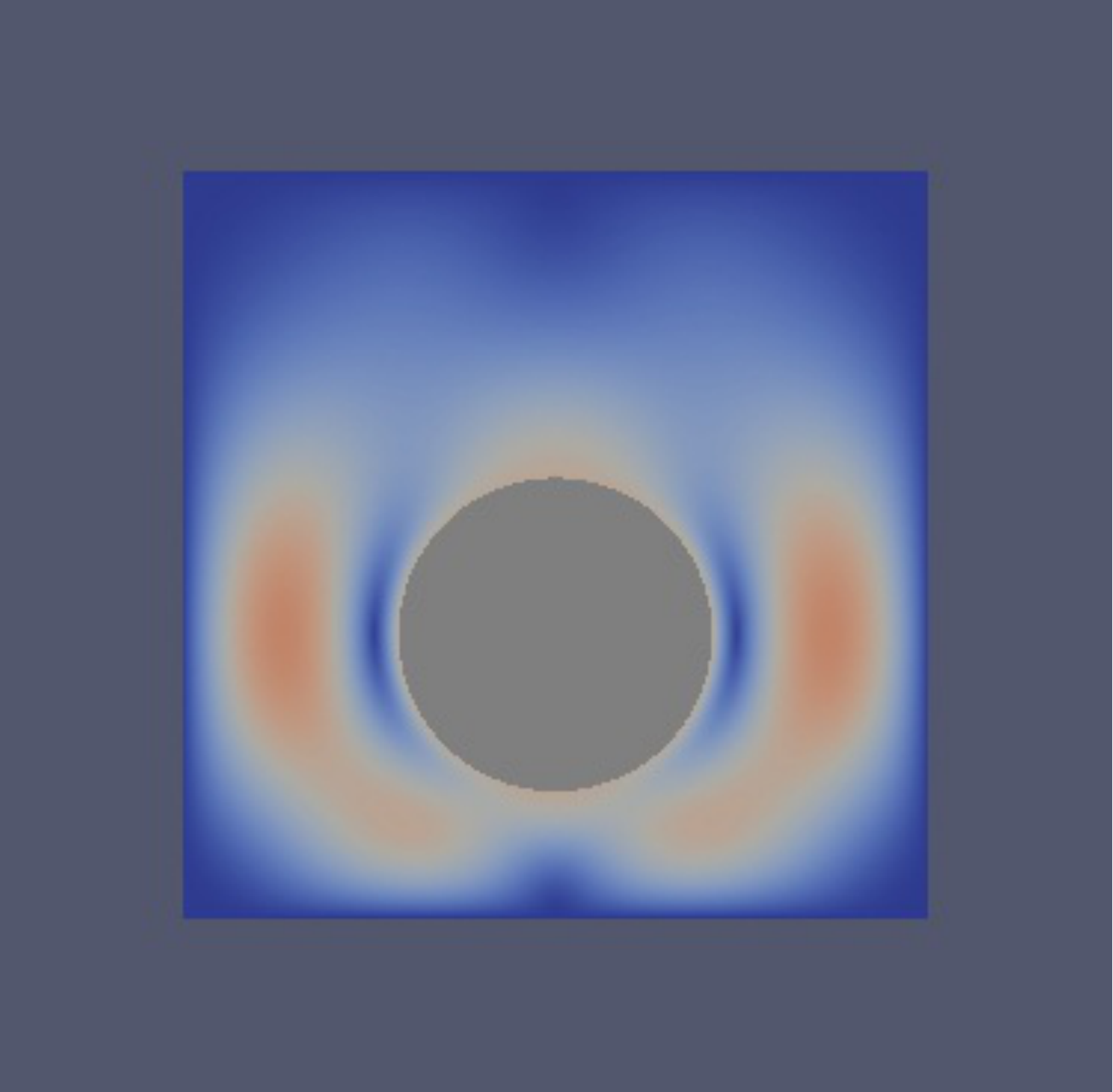} &
\includegraphics[trim = 1cm 1cm 1cm 1cm, clip, scale=0.3]{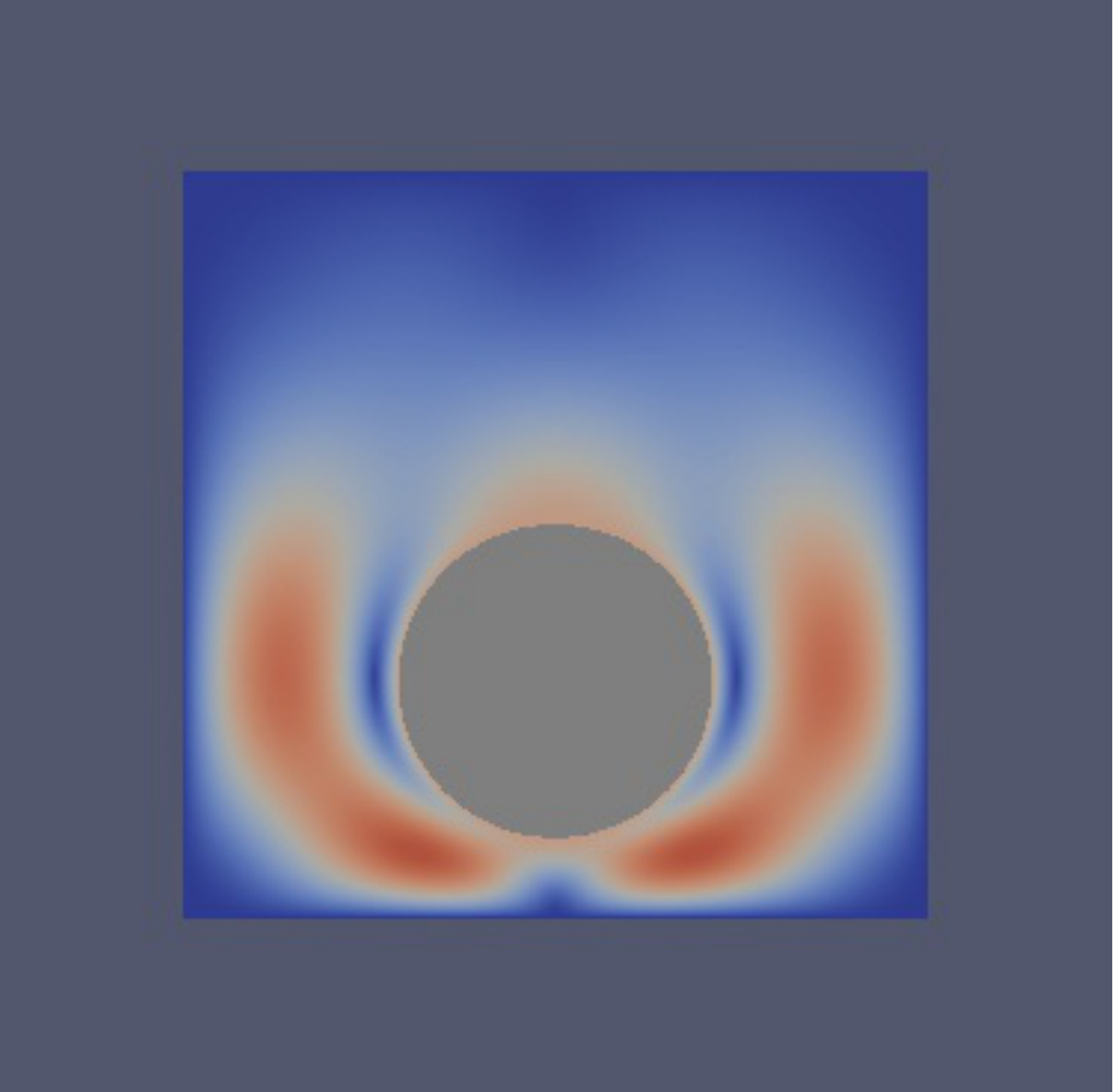} &
\includegraphics[trim = 1cm 1cm 1cm 1cm, clip, scale=0.3]{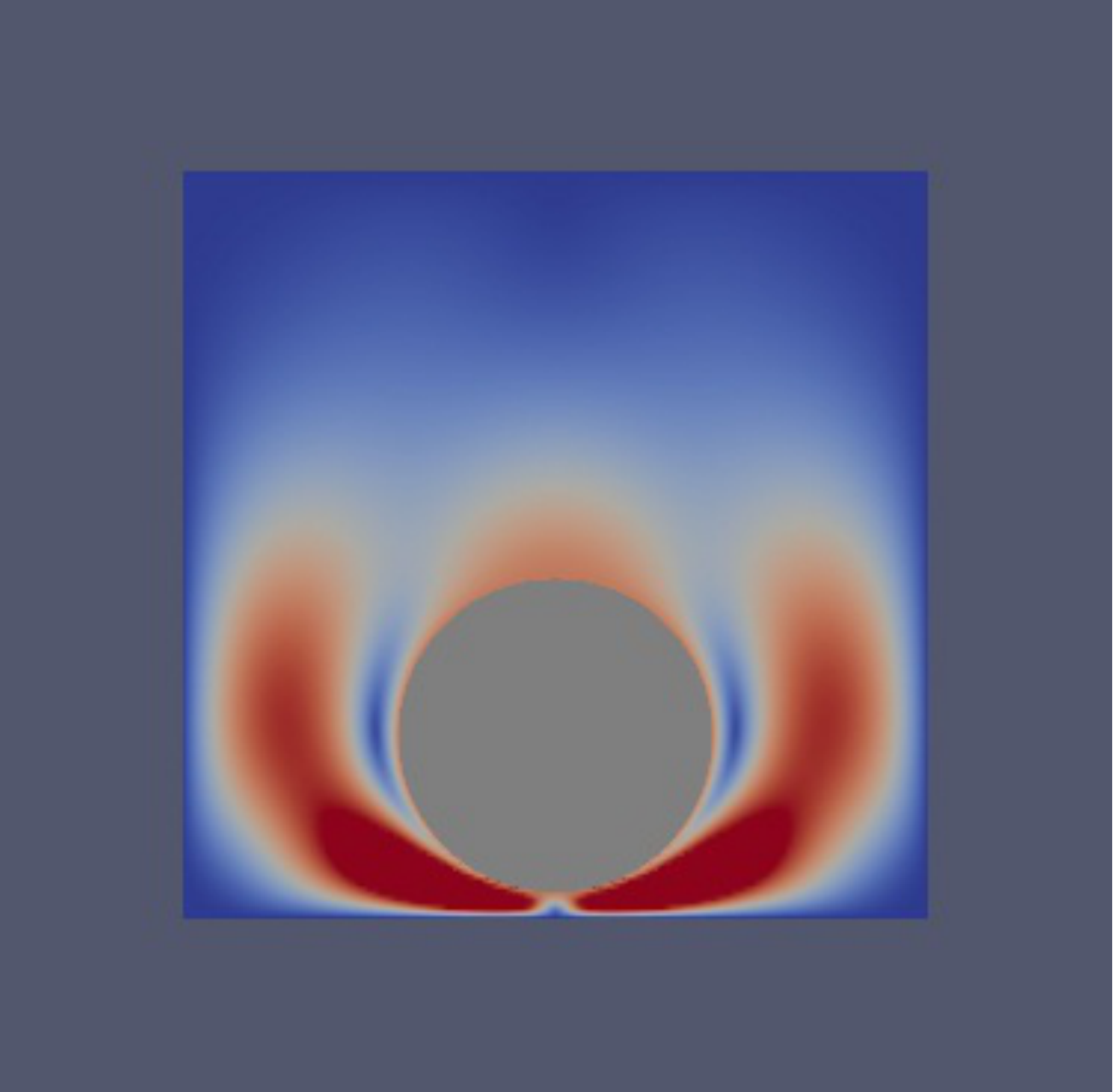}\\
$t = 48$  & $t = 51$  & $t = 54$\\
\hfill \\
\end{tabular}

\vspace*{-0.5cm}

\textcolor{black}{
\begin{figure}[!h]
\centering\caption{Simulation of the free fall of a ball in a Stokes flow.}
\end{figure}
}
\FloatBarrier

\textcolor{black}{
Note that this simulation cannot be carried out without the stabilization technique, because in that case the force that the fluid exerts on the solid is not well-computed. Note also that the contact between the ball and the floor would necessitate a special treatment that we do not develop here.
}

\textcolor{black}{
\section{Conclusion} \label{secCCL}
For Stokes problem which is the corner stone of computations in fluid dynamics, we have proposed a fictitious domain method based on extended finite element method. Dirichlet boundary conditions at the interface is made using Lagrange multiplier. Additional stabilization term is used to ensure an inf-sup  condition and to obtain an optimal convergence of the normal trace of the Cauchy stress tensor $\sigma(\bu,p)\bn$. The mathematical analysis is presented. We have carried out numerical simulations to compare the new method with the classical finite element approximation based on uncut mesh and with the same approach without the introduction of the stabilization term. Computations of convergence rates have been performed and have especially underlined the interest of the stabilization technique in order to compute a good approximation of the normal trace of the Cauchy stress tensor. Besides, this stabilization technique allows a robust behavior of this quantity when the position of the solid changes.\\
In a near future, we plan to perform simulations in an unsteady framework, by solving the incompressible Navier-Stokes equations in a domain where the solid is moving and deforming itself. Our method is particularly interesting in fluid-structure problems for which the role of the boundary is central, like for instance when the shape of the boundary is the unknown of a control problem.
}


\end{document}